\documentclass{amsart}
\usepackage{eucal}
\usepackage{amsmath}
\usepackage{amsthm,amssymb,stmaryrd,color,enumerate,accents,mathtools}

\usepackage[usenames,dvipsnames]{xcolor}
\usepackage[shortlabels]{enumitem}
\usepackage{url}
\usepackage{inconsolata} 
\definecolor{sangria}{rgb}{0.57, 0.0, 0.04}
\definecolor{royalblue}{rgb}{0.0, 0.22, 0.66}
\usepackage[colorlinks=true, urlcolor=sangria ,hyperindex, linkcolor=royalblue, pagebackref=false, citecolor=sangria]{hyperref}
\usepackage{physics} %for the \abs and \norm
\usepackage{bm}
\usepackage[bb=ams, cal=cm, scr=rsfso]{mathalfa}
\usepackage{graphicx}
\usepackage{tikz-cd}  % To draw BBD's octahedron axiom triangle...
\tikzset{
  symbol/.style={
    draw=none,
    every to/.append style={
      edge node={node [sloped, allow upside down, auto=false]{$#1$}}}
  }
}
\usepackage{marginnote}
\usepackage[textcolor=blue,
    linecolor=blue!10!white,
    bordercolor=blue!10!white,
    backgroundcolor=white, ]{todonotes}
\usepackage{fullpage}
\usepackage[all]{xy}
\usepackage{tabularray}
\theoremstyle{plain}
\newtheorem{thm}[subsubsection]{Theorem}
\newtheorem{prop}[subsubsection]{Proposition}
\newtheorem{lem}[subsubsection]{Lemma}
\newtheorem{cor}[subsubsection]{Corollary}

\newtheorem*{thm*}{Theorem}
\newtheorem*{prop*}{Proposition}
\newtheorem*{lem*}{Lemma}
\newtheorem*{cor*}{Corollary}
\newtheorem*{conj*}{Conjecture}

\newtheorem{thmx}{Theorem}

\theoremstyle{definition}
\newtheorem{defn}[subsubsection]{Definition}

\newtheorem{setup}[subsubsection]{Setup}

\newtheorem*{defn*}{Definition}
\newtheorem*{notation*}{Notation}
\newtheorem*{fact*}{Fact}
\newtheorem*{acknowledgement*}{Acknowledgement}

\theoremstyle{remark}
\newtheorem{rmk}[subsubsection]{Remark}

\newtheorem{example}[subsubsection]{Example}

\newtheorem*{rmk*}{Remark}
\newtheorem*{exercise*}{Exercise}
\newtheorem*{observation*}{Observation}
\newtheorem*{convention*}{Convention}
\newtheorem*{example*}{Example}

\newtheorem{rmkx}[thmx]{Remark}

\numberwithin{equation}{subsection}
\numberwithin{figure}{subsection} %% note: always put \label after the caption in figures and tables
\numberwithin{table}{subsection} %%or the numbering will be off
\makeatletter
\let\c@equation\c@subsubsection
\makeatother
\makeatletter
\let\c@table\c@subsubsection
\makeatother
\makeatletter
\let\c@figure\c@subsubsection
\makeatother

\usepackage{xfrac}
\newcommand{\PR}[1]{\left(#1\right)}
\newcommand{\CB}[1]{\left\{#1\right\}}
\newcommand{\BR}[1]{\left[#1\right]}

\newcommand{\DB}[1]{\llbracket#1\rrbracket}
\newcommand{\RG}[1]{\langle#1\rangle}
\newcommand{\pma}[1]{{\begin{pmatrix}#1\end{pmatrix}}}

\renewcommand{\mod}{\bmod}
\newcommand{\ra}{\rightarrow}
\newcommand{\xra}[1]{\xrightarrow{#1}}

\newcommand{\mono}{\hookrightarrow}
\newcommand{\epi}{\twoheadrightarrow}

\newcommand{\risom}{\buildrel\sim\over\rightarrow} 
\newcommand{\ov}{\overline}
\newcommand{\ud}{\underline}
\renewcommand{\hat}{\widehat}
\newcommand{\til}{\widetilde}
\newcommand{\GL}{\mathrm{GL}}

\newcommand{\Iw}{{\mathrm{Iw}}}
\newcommand{\Fl}{\mathrm{Fl}} 

\newcommand{\uG}{\ud{G}}
\newcommand{\uB}{\underline{B}}
\newcommand{\uT}{\ud{T}}

\newcommand{\uC}{\ud{C}} %alcove
\newcommand{\uPhi}{\underline{\Phi}}
\newcommand{\uDel}{\ud{\Delta}}

\newcommand{\tilW}{\til{W}}
\newcommand{\uW}{\underline{W}}
\newcommand{\utilW}{\underline{\til{W}}}
\newcommand{\uOm}{\underline{\Omega}}
\newcommand{\tilw}{\til{w}} %element in ext aff Weyl group
\newcommand{\tilz}{\til{z}} %element in DUAL ext aff Weyl group
 
\newcommand{\ctimes}{\hat{\otimes}}

\DeclareMathOperator{\Spec}{\mathrm{Spec}}

\DeclareMathOperator{\Spf}{\mathrm{Spf}}

\DeclareMathOperator{\Frob}{\mathrm{Frob}}

\DeclareMathOperator{\Hom}{\mathrm{Hom}}

\DeclareMathOperator{\supp}{\mathrm{supp}}
\DeclareMathOperator{\End}{\mathrm{End}}

\DeclareMathOperator{\Ind}{\mathrm{Ind}}
\DeclareMathOperator{\ind}{c\mathrm{\dash Ind}}
\DeclareMathOperator{\cind}{c\mathrm{\dash Ind}}
\DeclareMathOperator{\Lie}{\mathrm{Lie}}

\DeclareMathOperator{\Gal}{\mathrm{Gal}}
\DeclareMathOperator{\Art}{\mathrm{Art}}
\DeclareMathOperator{\Ext}{\mathrm{Ext}}
\DeclareMathOperator{\HT}{{\mathrm{HT}}}

\newcommand{\rhobar}{\overline{\rho}}
\newcommand{\rbar}{\overline{r}}
\newcommand{\mo}{{-1}}

\newcommand{\Fil}{\mathrm{Fil}}

\newcommand{\cris}{\mathrm{cris}}
\newcommand{\Dcris}{D_\cris}
\newcommand{\Vcris}{V_\cris}

\renewcommand{\ss}{{\mathrm{ss}}}

\newcommand{\gr}{\mathrm{gr}}

\newcommand{\ur}{\mathrm{ur}}

\newcommand{\et}{\normalfont{\text{{\'et}}}}

\newcommand{\Diag}{\mathrm{Diag}}
\newcommand{\JH}{\mathrm{JH}}

\newcommand{\WD}{\mathrm{WD}}
\newcommand{\rec}{\mathrm{rec}}

\newcommand{\univ}{\mathrm{univ}}

\newcommand{\sym}{\mathrm{sym}}
\newcommand{\red}{{\mathrm{red}}}

\newcommand{\bss}{{\backslash}}

\newcommand{\Ad}{\mathrm{Ad}}
\newcommand{\Adm}{{\mathrm{Adm}}}
\newcommand{\loccit}{\emph{loc.~cit.}}

\newcommand{\dash}{\text{-}}

\newcommand{\rmor}{{\mathrm{or}}}    
\newcommand{\ix}[1]{^{(#1)}}

\newcommand{\pgma}{(\varphi,\Gamma)}

\newcommand{\pcp}{{\wedge_p}} %p-adic completion

\newcommand{\jj}{{j\in \cJ}}
\newcommand{\tilcJ}{\til{\cJ}}
\newcommand{\tilj}{\til{j}}
\newcommand{\tiljj}{\til{j}\in \tilcJ}

\newcommand{\osig}{\overline{\sigma}}

\newcommand{\Mat}{\mathrm{Mat}}
\newcommand{\Stab}{\mathrm{Stab}}

\renewcommand{\ev}{\mathrm{ev}}
\renewcommand{\det}{\mathrm{det}}

\newcommand{\mc}{\mathcal}
\newcommand{\mf}{\mathfrak}
\newcommand{\mbf}{\mathbf}

\newcommand{\Q}{\mathbf{Q}}

\newcommand{\Qp}{\mathbf{Q}_p}
\newcommand{\Qpbar}{\overline{\mathbf{Q}}_p}
\newcommand{\Qpf}{\mathbf{Q}_{p^f}}
\newcommand{\Z}{\mathbf{Z}}
\newcommand{\Zp}{\mathbf{Z}_p}

\newcommand{\Zpbar}{\overline{\mathbf{Z}}_p}
\newcommand{\R}{\mathbf{R}}
\newcommand{\C}{\mathbf{C}}

\newcommand{\F}{\mathbf{F}}

\newcommand{\Fp}{\mathbf{F}_p}
\newcommand{\Fpbar}{\overline{\mathbf{F}}_p}

\newcommand{\G}{\mathbf{G}}

\newcommand{\T}{\mathbb{T}}

\newcommand{\fm}{{\mf{m}}}

\newcommand{\fp}{{\mf{p}}}

\newcommand{\fA}{{\mf{A}}}

\newcommand{\fI}{{\mf{I}}}

\newcommand{\fM}{{\mf{M}}}
\newcommand{\fN}{{\mf{N}}}

\newcommand{\fP}{{\mf{P}}}

\newcommand{\fS}{{\mf{S}}}

\newcommand{\cC}{{\mc{C}}}

\newcommand{\cG}{{\mc{G}}}
\newcommand{\cH}{{\mc{H}}}
\newcommand{\cI}{{\mc{I}}}
\newcommand{\cJ}{{\mc{J}}}
\newcommand{\cK}{{\mc{K}}}

\newcommand{\cO}{{\mc{O}}}
\newcommand{\cP}{{\mc{P}}}

\newcommand{\cR}{{\mc{R}}}
\newcommand{\cS}{{\mc{S}}}

\newcommand{\cU}{{\mc{U}}}

\newcommand{\cX}{{\mc{X}}}

\newcommand{\cZ}{{\mc{Z}}}

\newcommand{\bfa}{{\mbf{a}}}

\newcommand{\bfn}{{\mbf{n}}}

\newcommand{\rmG}{\mathrm{G}}
\newcommand{\rmK}{\mathrm{K}}
\newcommand{\rmT}{\mathrm{T}}
\newcommand{\rmH}{\mathrm{H}}
\newcommand{\rmM}{\mathrm{M}}

\newcommand{\al}{\alpha}
\newcommand{\be}{\beta}
\newcommand{\ga}{\gamma}
\newcommand{\Del}{\Delta}
\newcommand{\del}{\delta}
\newcommand{\eps}{\epsilon}
\newcommand{\veps}{\varepsilon}

\newcommand{\lam}{\lambda}
\newcommand{\sig}{\sigma}

\newcommand{\om}{\omega}
\newcommand{\Om}{\Omega}
\newcommand{\oom}{\overline{\omega}}

\newcommand{\bal}{\boldsymbol{\alpha}}
\title{Spectral mod $p$ Satake isomorphism for $\GL_n$}
\author{Heejong Lee}
\address{Korea Institute for Advanced Study, 85 Hoegi-ro, Dongdaemun-gu, Seoul 02455, Republic of Korea}
\email{heejonglee@kias.re.kr}

\begin{document}

\maketitle

\begin{abstract}
    Let $K/\Qp$ be a finite extension with residue field $k$. By a work of Emerton--Gee, irreducible components inside the reduced special fibre of the moduli stack of rank $n$ \'etale $\pgma$-modules are labelled by Serre weights of $\GL_n(k)$. Let $\sig$ be a non-Steinberg Serre weight and $\cC_\sig$ be the corresponding irreducible component. Motivated by the categorical $p$-adic local Langlands program, we construct a natural injective map $\cO(\cC_\sig) \mono \cH(\sig)$ from the ring of global functions on $\cC_\sig$ to the Hecke algebra of $\sig$ compatible with the mod $p$ Satake isomorphism by Herzig and Henniart--Vign\'eras in a suitable sense. For sufficiently generic $\sig$, we prove that it is an isomorphism.  As an application, we obtain a natural stratification of the irreducible component whose strata are equipped with a parabolic structure. Our main input is a construction of a morphism from an integral Hecke algebra of a generic tame type to the ring of global functions on a tamely potentially crystalline Emerton--Gee stack.    
    %When $n=2$ and $K/\Qp$ is unramified, such an isomorphism is a consequence of the conjectural categorical $p$-adic local Langlands correspondence. We also discuss conjectural interpretations of the isomorphism in the general case. As a main input, we construct a map from an integral Hecke algebra of a tame type to the ring of global functions on a potentially crystalline Emerton--Gee stack.    
    %We also show that non-vanishing loci of certain functions on $\cC_\sig$ naturally attain a parabolic structure. This provides a partition of the component with locally closed strata, and the unique closed stratum is a locus of mod $p$ analogue of supersingular representations.  
\end{abstract}

\setcounter{tocdepth}{2}
\tableofcontents

\section{Introduction}\label{sec:intro}

Let $K/\Qp$ be a finite extension with ring of integers $\cO_K$, uniformizer $\pi_K$, and residue field $k$. The classical Satake isomorphism implies that the spherical Hecke algebra $\cH^{\mathrm{sph}}_{\GL_n,K}$ of $\GL_n(K)$ is isomorphic to the algebra of symmetric Laurent polynomials in $n$-variables $\C[x_1^\pm,\dots,x_n^\pm]^{S_n}$. By the unramified local Langlands correspondence for $\GL_n(K)$, there is a bijection between characters $\cH^{\mathrm{sph}}_{\GL_n,K} \ra \C$ and $n$-dimensional unramified Frobenius-semisimple Weil--Deligne representations given by evaluating symmetric polynomials in $\cH^{\mathrm{sph}}_{\GL_n,K}$ at (unordered) $n$-tuple of Frobenius eigenvalues with a suitable normalization. Following the recent developments in the categorical local Langlands correspondence (\cite{Hellmann-Iwahori,BCHN,Zhu-cohshv,FS-LLC}), we obtain the following geometric interpretation: a Hecke operator in $\cH^{\mathrm{sph}}_{\GL_n,K}$ defines a global function on the moduli stack of unramified Langlands parameters $\mathrm{Loc}_{\GL_n,K}^{\ur}$, and this induces an isomorphism between $\cH^{\mathrm{sph}}_{\GL_n,K}$ and the ring of global functions on $\mathrm{Loc}_{\GL_n,K}^{\ur}$ (see \cite[Conj.~4.3.1 and Thm.~4.4.7]{Zhu-cohshv}). More generally, the categorical local Langlands correspondence for $\GL_n(K)$ provides an isomorphism between the Bernstein center of a Bernstein block $\Om$ and the ring of global functions on $\WD_{n,\tau}$ the stack of certain Weil--Deligne representations with fixed restriction to the inertia subgroup determined by $\Omega$ (see \cite[Prop.~5.2.10]{EGH22} and \cite[Thm.~5.13]{BCHN}).

Motivated by the categorical $p$-adic local Langlands program, we investigate $p$-adic and mod $p$ analogues of such isomorphisms. In the ``$l=p$'' setting, one considers a moduli stack of rank $n$ projective \'etale $\pgma$-modules $\cX_{n}$ constructed by Emerton--Gee \cite{EGstack} instead of moduli of Weil--Deligne representations. In \cite{EGH22}, Emerton--Gee--Hellmann proposes a categorical formulation of the $p$-adic local Langlands correspondence using the stack $\cX_{n}$ (the \textit{Banach} case in \loccit). In particular, they conjecture the existence of a fully faithful functor $\fA$ from a certain derived category of smooth representations of $\GL_n(K)$ to a certain derived category of coherent sheaves on $\cX_{n}$. Although such functor is only established for $\GL_1$ (see \S7.1 in \loccit; for $\GL_2(\Qp)$, this is the work in progress by Dotto--Emerton--Gee \cite{DEG} globalizing the existing $p$-adic local Langlands correspondence \cite{Col,Pas}), there are certain expected properties of the conjectural functor coming from its connection to the Taylor--Wiles--Kisin patched modules (see \cite[Conj.~6.1.14]{EGH22} and the following remarks).

We set up some notations. For a field $F$ with separable closure $\ov{F}$, we write $G_F := \Gal(\ov{F}/F)$. We write $I_K$ for the inertial subgroup of $G_K$. Let $E/\Qp$ be a sufficiently large finite extension with ring of integers $\cO$ and residue field $\F$. Let $B\subset \GL_n$ be the upper-triangular Borel and $T\subset \GL_n$ be the maximal diagonal torus. We write $\eta=(n-1,n-2,\dots,0)\in X^*(T)^{[K:\Qp]}$ for $[K:\Qp]$-tuple of a shift of half sum of positive roots.

Let $\lam\in X^*(T)^{[K:\Qp]}$ be a dominant cocharacter and $\tau:I_K \ra \GL_n(E)$ be an inertial type. Let $\cX_{n}^{\lam+\eta,\tau}\subset \cX_n$ be the closed substack whose $\Qpbar$-points parameterize potentially crystalline representations of $G_K$ of Hodge type $\lam+\eta$ and inertial type $\tau$. Let $\sig(\tau)$ be the smooth irreducible representation of $\GL_n(\cO_K)$ over $E$ associated to $\tau$ via the inertial local Langlands correspondence \cite[Thm.~3.7]{6author}. We define a locally algebraic type $\sig(\lam,\tau) := V(\lam )\otimes_\cO \sig(\tau)$ where $V(\lam )$ is the irreducible $\Res_{K/\Qp}(\GL_n)$-representation of highest weight $\lam $ restricted to $\GL_n(\cO_K)$. We denote by $\cH(\GL_n(K),\sig(\lam ,\tau))$ the Hecke algebra of $\sig(\lam ,\tau)$,~i.e.~the endomorphism ring of the compact induction $\cind_{\GL_n(\cO_K)}^{\GL_n(K)}\sig(\lam ,\tau)$. The aforementioned isomorphism between the Bernstein center and the ring of global functions on $\WD_{n,\tau}$ induces a morphism (see \cite[\S6.1.9]{EGH22})\footnote{Note that the target of the map (6.1.12) in \cite{EGH22} is the ring of global functions on the rigid generic fiber of $\cX_{n}^{\lam+\eta,\tau}$. It is expected that the image of this map is contained in the ring of bounded functions, which is equal to $\cO(\cX_{n}^{\lam+\eta,\tau})[\tfrac{1}{p}]$. Since our discussion here is mostly motivational, we take this for granted.}
\begin{align}\label{psi}
    \Psi^{\lam ,\tau}: \cH(\GL_n(K),\sig(\lam ,\tau)) \ra \cO(\cX_{n}^{\lam+\eta,\tau})[\tfrac{1}{p}].
\end{align}
Let $\sig^\circ(\lam ,\tau) \subset \sig(\lam ,\tau)$ be a $\GL_n(\cO_K)$-stable $\cO$-lattice. It is expected that $\fA((\cind_{\GL_n(\cO_K)}^{\GL_n(K)}\sig^\circ(\lam ,\tau))^\pcp)$ is a (pro-)coherent sheaf (i.e.~concentrated in degree $0$) supported on $\cX_n^{\lam+\eta,\tau}$ (here, the superscript $\pcp$ denotes the $p$-adic completion). Moreover, the action of $\cH(\GL_n(K),\sig(\lam ,\tau))$ on $\fA((\cind_{\GL_n(\cO_K)}^{\GL_n(K)}\sig^\circ(\lam ,\tau))^\pcp)[1/p]$ given by the functoriality of $\fA$ should coincide with the one given by $\Psi^{\lam ,\tau}$ and the natural action of $\cO(\cX^{\lam+\eta,\tau})[1/p]$.

We remark that a version of $\Psi^{\lam ,\tau}$ with the target replaced by $R_{\rhobar}^{\lam+\eta,\tau}[1/p]$ the generic fiber of the potentially crystalline deformation ring of a continuous representation $\rhobar: G_K \ra \GL_n(\F)$ was obtained in \cite[Thm.~4.1]{6author}. Furthermore, the image of the integral Hecke algebra $\cH(\GL_n(K), \sig^\circ(\lam ,\tau))$ is expected to be contained in the normalization of $R_{\rhobar}^{\lam+\eta,\tau}$ (Rmk.~4.21 in \loccit). In particular, it will be contained in $R_{\rhobar}^{\lam+\eta,\tau}$ if the deformation ring is known to be normal. One can make a similar expectation for $\cX_n^{\lam+\eta,\eta}$.

Our first main result shows that $\Psi^{\lam ,\tau}(\cH(\GL_n(K),\sig^\circ(\lam ,\tau)))$ is indeed contained in $\cO(\cX_{n}^{\lam+\eta,\tau})$ when $\tau$ is tame and sufficiently generic with respect to $\lam$. We remark that the genericity assumption excludes the crystalline case and also force the Hodge--Tate weights to be small with respect to $p$.

\begin{thmx}[Corollary \ref{cor:hecke-global-func}]
    If $\tau$ is tame and sufficiently generic, then there is a morphism 
    \begin{align*}
    \Psi^{\lam ,\tau}: \cH(\GL_n(K),\sig^\circ(\lam ,\tau)) \ra \cO(\cX_{n}^{\lam+\eta,\tau})
\end{align*}
    which recovers \eqref{psi} after inverting $p$.
\end{thmx}
We obtain this from a similar morphism with the target replaced by $\cO(Y^{\le\lam+\eta,\tau})$ where $Y^{\le\lam+\eta,\tau}$ is the moduli stack of Breuil--Kisin modules of Hodge type $\le\lam+\eta$ and inertial type $\tau$ without any genericity assumption on $\tau$ (Theorem \ref{thm:hecke-global-func}). This is the main content of \S\ref{sec4}. The genericity assumption in the above theorem appears when relating the stacks $\cX_n^{\lam+\eta,\tau}$ and $Y^{\le\lam+\eta,\tau}$. 

The morphism \eqref{psi} is obtained from a similar morphism with the target replaced by $\cO(\WD_{n,\tau})$ by taking pullback along a morphism from the rigid generic fiber of $\cX_n^{\lam+\eta,\tau}$ to $\WD_{n,\tau}$. This method is not suitable for an integral variant because Weil--Deligne representations do not behave well with coefficient rings in which $p$ is not inverted. Instead, Breuil--Kisin modules behave well integrally, and when the coefficient ring is a finite $\cO$-flat algebra, we can associate a Weil--Deligne representation to a Breuil--Kisin module. Following this idea, for generic tame $\tau$, \cite{LLMPQ-FL} constructed certain functions on $Y^{\le\eta,\tau}$ corresponding to Hecke operators in $\cH(\GL_n(K),\sig(\tau))$. We generalize their construction to any tame $\tau$ and Hodge type. When $\tau$ is not regular, meaning $\tau$ is not multiplicity free, the construction of functions is slightly more involved, reflecting the similar behavior on the Hecke algebra side. Furthermore, it is important to show that our functions are divisible by an appropriate power of $p$. A result of de Jong \cite{deJong95} says that the ring of global functions on a formal scheme over $\cO$ is equal to the ring of global functions on its rigid generic fiber bounded by 1 if the formal scheme is $\cO$-flat and normal. Therefore, the desired divisibility condition follows from the normality of $Y^{\le\lam+\eta,\tau}$. We prove the normality of $Y^{\le\lam+\eta,\tau}$ (Corollary \ref{cor:normal}) essentially by repeating the argument in \cite{CL-ENS-2018-Kisinmodule-MR3764041} where $\tau$ is assumed to be a principal series. We remark that, in general, $\cX^{\lam+\eta,\tau}$ can fail to be normal. Thus, one cannot hope to prove the above result by proving the normality of $\cX^{\lam+\eta,\tau}$ (also, see Remark \ref{rmk:fail}).

After constructing global functions on $Y^{\le\lam+\eta,\tau}$, we can define $\Psi^{\lam ,\tau}$ by matching them with Hecke operators explicitly.  In particular, our construction of $\Psi^{\lam ,\tau}$ is independent of the previous ones with $p$ inverted (but is obviously inspired by them). To do this, we give a presentation of the integral Hecke algebra $\cH(\GL_n(K),\sig^\circ(\lam ,\tau))$ using the results of Bushnell--Kutzko \cite{BKsmooth}, Schneider--Zink \cite{SZ}, and Dat \cite{Dat-caracteres}. When $\tau$ is tame, then $\sig(\tau)$ is a representation of $\GL_n(k)$ (it can be viewed as $\GL_n(\cO_K)$-representation via inflation). Associated to $\tau$, there exists a Levi subgroup $M\subset \GL_n$ and a representation $\lam_M$ of $M(k)$ such that $\sig(\tau)$ is a direct summand of $\Ind_{P(k)}^{\GL_n(k)}\lam_M$ where $P=BM$ is the standard parabolic containing $M$. There are two morphisms $t_P$ that relates $\cH(M(K),\lam_M)$ with $\cH(\GL_n(K), \Ind_{P(k)}^{\GL_n(k)}\lam_M)$ and $s_P$ that relates $\cH(\GL_n(K), \Ind_{P(k)}^{\GL_n(k)}\lam_M)$ with $\cH(\GL_n(K), \sig(\tau))$. The morphism $t_P$ does not respect integral structures of Hecke algebras but this can be understood explicitly. When $\tau$ is regular, $\Ind_{P(k)}^{\GL_n(k)}\lam_M$ is irreducible and $s_P$ is an identity map. When $\tau$ is not regular, $\sig^\circ(\tau)$ is in general not a direct summand of $\Ind_{P(k)}^{\GL_n(k)}\lam^\circ_M$ for some choice of lattice $\lam^\circ_M\subset \lam_M$. This makes it unclear whether $s_P$ respects integral structures. We show that $s_P$ indeed respects integral structures (Proposition \ref{prop:surj}) using somewhat technical computations of Hecke operators.   This is the main content of \S\ref{sec3}.

%Also, our construction is completely explicit: in \S\ref{sec3}, we give a presentation of the integral Hecke algebra, and in \S\ref{sec4}, we construct explicit global functions on $Y^{\le\lam,\tau^\vee}$ and match them with Hecke operators one by one. We also remark that our construction of $\Psi^{\lam-\eta,\tau}$ is independent of the previous ones (but is obviously inspired by them).

%\begin{rmk}
%    In \cite{LLMPQ-FL}, the authors construct various functions on open substacks of $Y^{\le\eta,\tau}$ under genericity conditions. At the level of $\Qpbar$-points, these functions are, in turn, related to Frobenius eigenvalues of Weil--Deligne representations associated with Breuil--Kisin modules. Our global functions on $Y^{\le\lam,\tau}$ are exactly the same kinds but their constructions are slightly more involved mainly because we do not put any genericity assumption. 
%\end{rmk}

Our main motivation to extend $\Psi^{\lam ,\tau}$ integrally is to obtain applications to the underlying reduced substack $\cX_{n,\red}\subset \cX_n$. We recall that irreducible components inside $\cX_{n,\red}$ are labelled by isomorphism classes of \textit{Serre weights}, i.e.~irreducible $\F$-representations of $\GL_n(k)$. Let $\sig$ be a Serre weight and $\cC_{\sig}$ be the corresponding irreducible component. When $n=2$ and $\sig$ is generic, the expectation coming from the Taylor--Wiles--Kisin patching implies that $\fA(\cind_{\GL_2(\cO_K)}^{\GL_2(K)}\sig)$ is a line bundle on $\cC_\sig$. By full faithfulness of $\fA$, this induces an isomorphism between the Hecke algebra $\cH(\GL_2(K),\sig) = \End_{\GL_2(K)}(\cind_{\GL_2(\cO_K)}^{\GL_2(K)}\sig)$ and the ring of global functions $\cO(\cC_\sig)$. In general, the support of $\fA(\cind_{\GL_n(\cO_K)}^{\GL_n(K)}\sig)$ is expected to be the conjectural Breuil--M\'ezard cycle associated to $\sig$. We will return to this point later.

\S\ref{sec5} is mostly devoted to our second main result which establishes a natural isomorphism between $\cH(\GL_n(K),\sig)$ and $\cO(\cC_\sig)$ for generic $\sig$. It is natural in the sense that it is compatible with $\Psi^\tau := \Psi^{0,\tau}$ and the mod $p$ Satake isomorphism. To state this result, we introduce some notations.

The structure of $\cH(\GL_n(K),\sig)$ is well-understood by the mod $p$ Satake isomorphism of Herzig \cite{HerzigSatake} and Henniart--Vign\'eras \cite{HV-satake}. They construct an injective morphism $\cS_\sig: \cH(\GL_n(K),\sig) \ra \cH(T(K),\mu)$ and identify its image. Here, $\mu$ is the highest weight of $\sig$. % and $w_0$ is the longest element in the Weyl group. 
As a result, $\cH(\GL_n(K),\sig)$ is isomorphic to the polynomial ring $\F[y_1,\dots,y_{n-1},y_n^\pm]$.

Let $\cX_{T,\red}:=\cX_{1,\red}^{\oplus n}$ be the reduced Emerton--Gee stack for $T(K)$ and $\cC_{\mu}$ be its irreducible component corresponding to $\mu$. It follows from the local class field theory that there is a natural isomorphism $\ov{\Psi}_\mu: \cH(T(K),\mu) \risom \cO(\cC_\mu)$ (see \S\ref{subsec:tori}). When $\sig$ is \textit{non-Steinberg} (see Definition \ref{defn:SW-generic}), there is a natural morphism $\cC_{\mu+\eta} \ra \cC_\sig$. Also, there is an isomorphism $\cC_{\mu}\simeq \cC_{\mu+\eta}$ given by twisting by a character. We define $S_\sig: \cO(\cC_\sig) \ra \cO(\cC_{\mu})$ to be the pullback along the composition $\cC_{\mu} \simeq \cC_{\mu+\eta} \ra \cC_\sig$. This is an avatar of $\cS_\sig$ on the Galois side. 

Finally, let $\til{\mu}$ be the Teichm\"uller lift of $\mu$. By the local class field theory, $\til{\mu}$ defines a principal series tame inertial type $\tau$ and $\sig^\circ(\tau)=\Ind_{B(k)}^{G(k)} \til{\mu}$ is a lattice in a principal series tame type such that the cosocle of $\osig(\tau)$ is isomorphic to $\sig$. We introduce a natural morphism $\cR^{\sig^\circ(\tau)}_\sig : \cH(\GL_n(K),\sig^\circ(\tau)) \ra \cH(\GL_n(K),\sig)$ under a mild assumption on $\mu$ (Theorem \ref{thm:hecke-modp-red}). On the Galois side, if $\sig$ is sufficiently generic, $\cC_\sig$ is contained in $\cX_{n}^{\eta,\tau}$, and we write the restriction map $R^{\tau}_\sig: \cO(\cX_{n}^{\eta,\tau}) \ra \cO(\cC_\sig)$.

\begin{thmx}[Proposition \ref{prop:upperbound} and Theorem \ref{thm:spec-satake}]\label{thm:main}
    Suppose that $\sig$ is a non-Steinberg Serre weight. There exists a morphism $\ov{\Phi}_\sig : \cO(\cC_\sig) \ra \cH(\GL_n(K),\sig)$ which fits in the following commutative diagram
    \[
    \begin{tikzcd}
        \cH(\GL_n(K),\sig) \arrow[r, hook, "\cS_\sig"] & \cH(T(K),\mu)  \\
        \cO(\cC_\sig) \arrow[r, hook, "S_\sig"]  \arrow[u, "\ov{\Phi}_\sig"]& \cO(\cC_{\mu}) \arrow[u, "\ov{\Psi}_{\mu}^\mo"].
    \end{tikzcd}
    \]
    Moreover, if $\sig$ is sufficiently generic, then $\ov{\Phi}_\sig$ is an isomorphism and its inverse denoted by $\ov{\Psi}_\sig$ fits in the following commutative diagram
    \begin{equation}\label{diagram2}
        \begin{tikzcd}
        \cH(\GL_n(K),\sig^\circ(\tau)) \arrow[r, two heads, "\cR^{\sig^\circ(\tau)}_\sig"] \arrow[d, "\Psi^{\tau}"] & \cH(\GL_n(K),\sig) \arrow[r, hook, "\cS_\sig"] \arrow[d, "\ov{\Psi}_\sig"] & \cH(T(K),\mu) \arrow[d, "\ov{\Psi}_{\mu}"] \\
        \cO(\cX_n^{\eta,\tau}) \arrow[r, two heads, "R^{\tau}_\sig"] & \cO(\cC_\sig) \arrow[r, hook, "S_\sig"] & \cO(\cC_{\mu}).
    \end{tikzcd}
    \end{equation}
\end{thmx}

%\begin{rmk}
%    The full faithfulness of $\fA$ further implies an identification of $E_1$-rings between the derived Hecke algebra of $\sig$ and a mod $p$ spectral Hecke algebra (a conjectural mod $p$ analogue of \cite{Feng-spechecke}); see \cite[Rmk.~6.1.21]{EGH22}. Our result can be viewed as the $H^0$ part of this equivalence.    
%\end{rmk}

\begin{rmkx}
    The map $R^{\tau}_{\sig}: \cO(\cX_{n}^{\eta,\tau}) \ra \cO(\cC_\sig)$ can be defined for any tame inertial type $\tau$ and irreducible component $\cC_\sig$ inside $\cX^{\eta,\tau}$. On the representation side, we define $\cR^{\sig^\circ(\tau)}_\sig : \cH(\GL_n(K),\sig^\circ(\tau)) \ra \cH(\GL_n(K),\sig)$ under a mild assumption on $\tau$ and somewhat strong assumption on $\sig \in \JH(\osig(\tau))$ (see Theorem \ref{thm:hecke-modp-red}) (here, $\osig(\tau)$ denotes the mod $p$ reduction of $\sig^\circ(\tau)$). We prove that
    \begin{align*}
        \ov{\Psi}_\sig \circ R^{\tau}_{\sig}  = \cR^{\sig^\circ(\tau)}_\sig \circ \Psi^\tau
    \end{align*}
    when all morphisms are defined (Corollary \ref{cor:spec-hecke-modp}). However, it is not obvious how to generalize $\cR^{\sig^\circ(\tau)}_\sig$ to all Jordan--H\"older factors of $\osig(\tau)$.
\end{rmkx}

\begin{rmkx}
    We discuss previous results on $\cO(\cC_\sig)$. Suppose that $K/\Qp$ is unramified. For $n=2$, the main result of \cite{GKKSW} shows that $\cO(\cC_\sig)\simeq \F[y_1,y_2^\pm]$ for possibly non-generic $\sig$. Moreover, they give a presentation of $\cC_\sig$ as a stack quotient and, in particular, prove that $\cC_\sig$ is smooth. For arbitrary $n$, \cite{HJKLW} proves that $\cO(\cC_\sig)\simeq \F[y_1,\dots,y_{n-1},y_n^\pm]$ for sufficiently generic Serre weight $\sig$ (which is slightly stronger than our assumption) by working purely in characteristic $p$. Furthermore, based on the results in \cite{LLMPQ-FL}, \cite[\S7]{HJKLW} explains an idea of how to match global functions on $\cC_\sig$ restricted to its ordinary locus to mod $p$ reduction of certain normalized Hecke operators. The commutativity of the left square in the diagram \eqref{diagram2} realizes this idea precisely. We also remark that our results are independent of those of \cite{GKKSW,HJKLW}. %Our main contributions are pinning down the morphism between $\cO(\cC_\sig)$ and $\cH(\GL_n(K),\sig)$ using the mod $p$ Satake isomorphism while allowing $K/\Qp$ to be ramified and   % with its relationship with other morphisms as in the second part of the above theorem. In particular, this allows us to evaluate a function $\ov{f}\in \cO(\cC_\sig)$ at any $\rhobar \in \cC_\sig(\Fpbar)$ by lifting them to $f\in \cO(\cX^{\eta,\tau}_n)$ and $\rho \in \cX^{\eta,\tau}_n(\Zpbar)$, respectively, where we have an explicit description of $f(\rho)$ in terms of the Frobenius eigenvalues of the Weil--Deligne representation associated to $\rho$. In contrast, the result in \cite{HJKLW} can only describe the value of $\ov{f}$ at $\rhobar$ contained in a certain dense open locus in $\cC_\sig$.
\end{rmkx}

\begin{rmkx}
    It is interesting to ask whether there is a similar isomorphism between $\cH(\sig)$ and $\cO(\cC_\sig)$ for groups other than $\GL_n$. For general reductive groups, $\cH(\sig)$ is not necessarily commutative. Still, it is commutative if the group is quasi-split \cite[\S1.8]{HV-satake}. Moreover, the Emerton--Gee stacks for tame groups are constructed by Lin \cite{LinEG1,LinEG2}. So one can still hope to obtain a similar isomorphism in this setting. We hope to return to this in future work.
\end{rmkx}

In \S\ref{sec6}, we discuss an application of the previous theorem. For generic $\sig$ and $i\in \CB{1,2,\dots,n}$, let $\ov{T}_i$ be the unique up-to-scalar double coset operator in $\cH(\GL_n(K),\sig)$ supported on
\begin{align*}
    \GL_n(\cO_K)\Diag(\pi_K^\mo, \dots, \pi_K^\mo,1,\dots,1) \GL_n(\cO_K)
\end{align*}
where $\Diag(\pi_K^\mo, \dots, \pi_K^\mo,1,\dots,1)$ is the diagonal matrix whose first $i$ entries are $\pi_K^\mo$. We define $\ov{f}_i:=\ov{\Psi}_\sig(\ov{T}_i) \in \cO(\cC_\sig)$.

By its construction, $\cC_\sig$ contains a dense open locus $\cU_\sig$ whose $\Fpbar$-points are of the form
\begin{align*}
    \rhobar = \pma{\chi_1 & * & \cdots & * \\ 
     0 & \chi_2 & \cdots & * \\ 
     \text{\rotatebox{90}{$\cdots$}} &  & \text{\rotatebox{135}{$\cdots$}} &  \text{\rotatebox{90}{$\cdots$}} \\ 
     0 & \cdots & 0 & \chi_n}
\end{align*}
where its semisimplification $\rhobar^\ss = \oplus_{i=1}^n \chi_i$ defines a $\Fpbar$-point in $\cC_{\mu+\eta}$. Then the function $\ov{f}_i$ maps the above $\rhobar$ to $\prod_{k=1}^i \chi_k(\Frob_K)$ where $\Frob_K \in G_K$ is a  geometric Frobenius. In particular, $\cU_\sig$ is contained in the non-vanishing locus of $\ov{f}_i$ for all $i=1,\dots, n$. For generic $\sig$, we prove that similar non-vanishing loci in $\cC_\sig$ attain a ``parabolic structure'', meaning its points consist of Galois representations valued in a parabolic subgroup whose projections to the Levi subgroup form an irreducible component in the reduced Emerton--Gee stack for the Levi subgroup.

Let $I= \CB{i_1<i_2<\dots<i_r}\subset \CB{1,2,\dots,n-1}$. We define $\cU_{\sig,I}\subset \cC_\sig$ to be the intersection of non-vanishing loci of $\ov{f}_i$ for $i\in I$. This defines a stratification on $\cC_\sig$ with strata given by $\cC_{\sig,I}:= \cU_{\sig,I} \bss \cup_{I\subset I'}\cU_{\sig,I'}$. When $I=\emptyset$, we call $\cC_{\sig}^\ss:=\cC_{\sig,\emptyset}$ the \textit{supersingular locus} of $\cC_\sig$. We say that a continuous representation $\rhobar:G_K \ra \GL_n(\F)$ is \textit{supersingular of weight $\sigma$} if $\rhobar\in \cC_\sig^\ss(\F)$.

Let $P_I\subset \GL_n$ the standard parabolic subgroup with Levi factor $M_I\simeq \prod_{k=1}^{r+1} \GL_{i_{k}-i_{k-1}}$ where we write $i_0=0$ and $i_{r+1}=n$. We can view any $\rhobar$ of the above form as a $G_K$-representation valued in $P_I(\Fpbar)$. By composing the quotient map $P_I\epi M_I$, we get a $\Fpbar$-point in $\cX_{M_I,\red}:= \prod_{k=1}^{r+1} \cX_{i_{k}-i_{k-1},\red}$. The closure of the image of all such $\rhobar$ is an irreducible component of $\cX_{M_I,\red}$ which we denote by $\cC_{\sig_I}$. Note that $\cC_{\sig_I}$ is a product of irreducible components $\cC_{\sig_k}\subset \cX_{i_{k}-i_{k-1},\red}$ for $k=1,\dots,r+1$.

\begin{thmx}[Theorem \ref{thm:modp-parabolic}]
    Suppose that $\sig$ is sufficiently generic. There is a natural morphism
    \begin{align*}
        \cU_{\sig,I} \ra \cC_{\sig_I}
    \end{align*}
    whose restriction to $\cC_{\sig,I}$ lands in $\cC_{\sig_I}^\ss:= \prod_{k=1}^{r+1} \cC_{\sig_k}^{\ss}$. 
\end{thmx}

\begin{rmkx}
    It is an intriguing question that how one can classify all $\Fpbar$-points of $\cC_\sig$. In fact, the answer to this question should be closely related to explicit Serre weight conjectures. By the above result, at least for sufficiently generic $\sig$, this question inductively reduces to the question of classifying all $\Fpbar$-points of $\cC_\sig^\ss$, namely all supersingular $G_K$-representations over $\Fpbar$ of weight $\sig$. Unfortunately, we expect that classifying supersingular $G_K$-representations over $\Fpbar$ of weight $\sig$ is equally difficult to the original question.
\end{rmkx}

The above theorem follows from an analogous parabolic structure of certain non-vanishing loci in $\cX^{\eta,\tau}_n$. In turn, this parabolic structure follows from the parabolic structure of local models of $\cX^{\eta,\tau}_n$ studied in \cite{LLLM-extreme}. There is a notion of \textit{shape} $\tilw = \tilw(\rhobar,\tau)$ of the pair $(\rhobar,\tau)$ of a tame inertial type $\tau$ and a semisimple representation $\rhobar \in \cX^{\eta,\tau}_n(\F)$. It is an element in the ($f$-fold copy of the) extended affine Weyl group and can be thought of as the relative position of $\rhobar$ and $\tau$. For each shape $\tilw$, there is an open neighbourhood $\cX^{\eta,\tau}_n(\tilw) \subset \cX^{\eta,\tau}_n$ and a quasi-projective variety over $\Spec \cO$ (a local model of $\cX^{\eta,\tau}_n(\tilw)$) whose $p$-adic completion is smoothly equivalent to $\cX^{\eta,\tau}_n(\tilw)$. In \loccit, the authors show that whenever the shape $\tilw$ is ``decomposable'', the local model admits a natural parabolic structure. In general, this parabolic structure cannot be transferred to $\cX^{\eta,\tau}_n(\tilw)$. We show that when the decomposability of $\tilw$ follows from the simultaneous decomposability of $\tau$ and $\rhobar$, $\cX^{\eta,\tau}_n(\tilw)$ also attains a parabolic structure (Theorem \ref{thm:reducibility-in-families}). We remark that for a single potentially crystalline representation, this parabolic structure can be obtained from its associated weakly admissible filtered $\varphi$-module (see Proposition \ref{prop:reducibility}).

In \S\ref{subsec:LGC}, we discuss a role of the isomorphism $\ov{\Psi}_\sig$ in the mod $p$ local-global compatibility conjecture. Let $F/F^+$ be a CM extension and fix a place $\til{v}$ of $F$ above $p$. Let $\rbar: G_F \ra \GL_n(\F)$ be a continuous representation. Using a certain Hecke isotypic component of a space of algebraic automorphic forms at infinite level at $v:= \til{v}|_{F^+}$ for a definite unitary group of rank $n$, one can construct an admissible smooth representation $\pi(\rbar)$ of $\GL_n(K)$. This is a candidate for a (Galois-to-automorphic) mod $p$ local Langlands correspondence and is expected to be determined by $\rbar|_{G_{F_{\til{v}}}}$. We provide the following small evidence for this conjecture. Let $\sig$ be a Serre weight such that $\rbar|_{G_{F_{\til{v}}}}\in \cC_\sig(\F)$. Under a genericity condition on $\sig$ and some standard hypothesis, we prove that the action of $\cH(\sig)$ on $\Hom_{\GL_n(\cO_K)}(\sig,\pi(\rbar)|_{\GL_n(\cO_K)})$ is given by the character (Theorem \ref{thm:modpLGC})
\begin{align*}
    \cH(\sig) \xra{\ov{\Psi}_\sig} \cO(\cC_\sig) \xra{\ev_{\rbar|_{G_{F_w}}}} \F.
\end{align*}
We deduce this from a stronger version for a Taylor--Wiles--Kisin patched module (Theorem \ref{thm:modpLGC-family}).

Finally in \S\ref{subsec:BMcycle}, we speculate on a possible relationship between $\cH(\sig)$ and the ring of global functions on a (conjectural) Breuil--M\'ezard cycle $\cZ_\sig$. As we discussed above, it is expected that $\cZ_\sig$ is given by the support of the conjectural (pro-)coherent sheaf $\fA(\cind_{\GL_n(\cO_K)}^{\GL_n(K)}\sigma)$. By full faithfulness of $\fA$, it is natural to consider a morphism from $\cO(\cZ_\sig)$ to $\cH(\sig)$. For $K/\Qp$ unramified and generic $\sig$, we use some known properties of Breuil--M\'ezard cycles in (strongly) limited cases to prove that $\cO(\cZ_\sig)_\red$ is isomorphic to $\cH(\sig)$ (Proposition \ref{prop:BMcycle}).

\subsection*{Acknowledgements} I would like to thank Florian Herzig and Daniel Le for numerous discussions which have improved the quality of this paper significantly. Special thanks go to Bao V.~Le Hung for suggesting to use the normality of moduli stack of Breuil--Kisin modules. I would like to thank Matthew Emerton, Toby Gee, and Chol Park for helpful discussions and their encouragement, Jean-Fran\c{c}ois Dat and Brandon Levin for helpful correspondence, and Kalyani Kansal and Shenrong Wang for inspiring conversation on their works. I would also like to thank Toby Gee, Florian Herzig, and an anonymous referee for comments on an earlier version of this article.

I learned about the categorical $p$-adic local Langlands program and the question about global functions and Hecke algebras in the lecture series given by Matthew Emerton, Toby Gee, and Eugen Hellmann at the 2022 IH\'ES Summer School on the Langlands Program. I would like to thank the organizers for the pleasant event, and the institution for its hospitality. Part of the work was carried out during the author's visit to the trimester program ``The Arithmetic of the Langlands Program'' at the Hausdorff Institute for Mathematics funded by the Deutsche Forschungsgemeinschaft (DFG, German Research Foundation) under Germany's Excellence Strategy--EXC-2047/1--390685813 in June 2023 and to the Ulsan National Institute of Science and Technology in July 2023. I would like to heartily thank these institutions for their support and excellent working environment.

\section{Notations and preliminaries}\label{sec2}

\subsection{Galois theory}
Let $F$ be a field. We fix once and for all a separable closure $\ov{F}$ of $F$ and write $G_F:=\Gal(\ov{F}/F)$. Let $K/\Qp$ be a finite extension with ring of integers $\cO_K$, uniformizer $\pi_K$, and residue field $k$. We denote by $e$ its ramification index of $K/\Qp$ and by $f$ its inertial degree. We write $q:=p^f = \#k$. We write $I_K \subset G_K$ for the inertia subgroup and $W_K$ for the Weil group of $K$. For $m\ge 1$, we choose a compatible system of $\pi_{K,m}$ a $(q^{m}-1)$-th root of $\pi_K$. We write $K_\infty := \cup_{m\ge 1}K(\pi_{K,m})$. We denote by $\Frob_K$ a geometric Frobenius in $G_K$. We denote by $K_0\subset K$ the maximal unramified subextension and by $\varphi$ the absolute arithmetic Frobenius on $K_0$. 

Let $E\subset \Qpbar$ be a subfield which is finite over $\Qp$ with ring of integers $\cO$, uniformizer $\varpi$, and residue field $\F$. When an integer $n$ is fixed in the body of the paper, we assume that $E$ contains the Galois closures of the unramified extensions of $K$ of degree $\le n$. We denote by $\tilcJ$ the set of embeddings $K \mono E$ and by $\cJ$ the set of embeddings $k\mono \F$. We view $\cJ$ as a quotient of $\tilcJ$ with respect to the equivalence relation $\tilj\sim \tilj'$ if and only if $\tilj|_{K_0} = \tilj'|_{K_0}$ for $\tilj,\tilj' \in \tilcJ$. We fix an embedding $\sig_0 : k \mono \F$ and define $\sig_j = \sig_0 \circ \varphi^{-j}$. This identifies $\cJ$ with $\Z/f\Z$. We also fix an embedding $K\mono E$ that extends $\sig_0$ and often view $K$ as a subfield of $E$ using the embedding.

For $n\ge 1$, we denote by $\omega_{K} : G_K \ra \cO_{K_0}^\times$ the character defined by $g(\pi_K)= \omega_K(g)\pi_K$ for $g\in G_K$. For an embedding $\sig: K_0 \mono E$, we write $\omega_{K,\sig} = \sig \circ \omega_K$. For $x\in \F^\times$, we write $\ur_x$ for the unramified character that maps $\Frob_K$ to $x$. 

A \textit{Hodge type} or \textit{Hodge--Tate weight} is a dominant cocharacter in $X_*(T)^{\tilcJ}$. If $\lam$ is a Hodge type, we write $h_\lam := \max_{\tilj \in \tilcJ, \al\in \Phi^+}\CB{\RG{\lam_{\tilj},\al^\vee}}$. 
Let $\veps$ denote the $p$-adic cyclotomic character. If $V$ is a de Rham representation of $G_K$ over $E$, then for each $\tilj \in \Hom_{\Qp}(K,E)$, we let $\HT_{\tilj}(V)$ denote the multiset of Hodge--Tate weights labelled by $\tilj$ normalized so that $\HT_{\tilj}(\veps)=\CB{1}$. We write $\WD(V)$ for the Frobenius semisimplification of the Weil--Deligne representation associated to $V$ as in \cite[App.~B]{CDT}. In particular, $V \mapsto \WD(V)$ is covariant. For a Hodge type $\lam$ and inertial type $\tau:I_K \ra \GL_n(E)$, we say that $V$ is potentially crystalline of type $(\lam,\tau)$  if it has Hodge type $\lam$ and $\WD(V)|_{I_{K}}\simeq \tau$ with $N=0$. Note that our convention agrees with \cite{LLLMlocalmodel,LLLM-extreme}, but is opposite of \cite{6author,EGstack,EGH22}.

We normalize the Artin map $\Art_{F_v} : F_v^\times \risom W^{\mathrm{ab}}_{F_v}$ so that uniformizers are mapped to geometric Frobenius elements. We let $\rec$ denote the local Langlands correspondence as in the introduction to \cite{HT}. Fix once and for all an isomorphism $\iota:\Qpbar\risom \C$. We define the local Langlands correspondence  $\rec_p$ over $\Qpbar$ by $\iota\circ \rec_p = \rec \circ \iota$ and we define $r_p (\pi) := \rec_p(\pi\otimes \abs{\det}^{(n-1)/2})$. Beaware that the exponent of $\abs{\det}$ is different from \cite{6author}. Note that $\rec_p$ depends only on $\iota^\mo(\sqrt{p})$ and $r_p$ is independent of the choice of $\iota$.

If $\rho: G_K \ra \GL_n(\cO)$ is a continuous representation, we write $\rhobar := \rho \mod \varpi$.

%Let $F$ be a number field. For a place $v$ of $F$, we let $F_v$ be the completion of $F$ at $v$ with the ring of integers $\cO_{F_v}$, a uniformizer $\varpi_v$, and the residue field $k_v$ of size $q_v$. We write $\Frob_{F_v}$ for a geometric Frobenius element in $G_{F_v}$. 

\subsection{Reductive groups}
Let $G$ be a split connected reductive group over $\Z$. We write $B\subset G$ for a choice of Borel subgroup, $T\subset B$ for a maximal torus, and $U\subset B$ for the unipotent radical of $B$.  Let $\Phi^+ \subset \Phi$ be the subset of positive roots in the set of roots for $(G,B,T)$. We denote by $\Del$ the set of simple roots. We write $X^*(T)$ for the group of characters of $T$ and $X_*(T)$ for the group of cocharacters of $T$.  %We write $\Lam_R \subset X^*(T)$ and $\Lam_R^\vee\subset X_*(T)$ for the root lattice and coroot lattice. 
We let $W$ denote the Weyl group, $W_a$ denote the affine Weyl group, and $\tilW$ denote the extended affine Weyl group for $G$.

Let $M\subset G$ be a Levi subgroup. We define a subset $W^M \subset W$
    \begin{align*}
        W^M:= \CB{w\in W  \mid l(s_\al w) > l(w) \ \forall \al \in \Del_M}.
    \end{align*}
    Then the natural map $W^M \ra W_M\bss W$ is a bijection and $w^M\in W^M$ is the unique minimal length representative of the coset $W_M w^M$. We also define ${}^MW^M := W^M \cap (W^M)^\mo$. Then the map ${}^MW^M \ra W_M \bss W / W_M$ is a bijection and ${}^Mw^M \in {}^MW^M$ is the unique minimal length representative of the double coset $W_M {}^Mw^M W_M$.

We write $G^\vee$ for the split reductive group over $\Z$ defined by the root datum $(X_*(T),X^*(T),\Phi^\vee,\Phi)$. We write $T^\vee \subset G^\vee$ for the induced maximal split torus. We have isomorphisms $X^*(T^\vee)\simeq X_*(T)$ and $X_*(T^\vee)\simeq X^*(T)$. We let $W^\vee$, $W_a^\vee$, and $\tilW^\vee$ denote the Weyl group, affine Weyl group, and extended affine Weyl group for $G^\vee$.

%We denote by $\cO_p$ a finite \'etale $\Z_p$-algebra. In the body of this paper, we take $\cO_p$ to be either $\cO_{K}$ or $\cO_F\otimes_{\Z}\Z_p$ for some totally real field $F$ in which $p$ is unramified. Let $F_p = \cO_p[1/p]$. Then $F_p$ is isomorphic to a product $\prod_{v\in S_p}F_v$ for a finite set $S_p$ and finite unramified extensions $F_v/\Qp$. Also there is an isomorphism $\cO_p\simeq \prod_{v\in S_p}\cO_{F_v}$ where $\cO_{F_v}$ is the ring of integer of $F_v$.

We define $G_0 := \Res_{\cO_{K_0} / \Z_p}G_{/\cO_{K_0}}$ and $\uG := (G_0)_{/\cO}$.  %Let $B$ be a choice of Borel subgroup and $T\subset B$ be a maximal split torus. 
We define $B_0, \uB$ and $T_0, \uT$ similarly to $G_0, \uG$. Then $(\uG,\uB,\uT)$ is naturally identified with ($G^\cJ_{/\cO},B^\cJ_{/\cO},T^\cJ_{/\cO})$.
 The root datum of $(\uG,\uB,\uT)$ is given by
 \begin{align*}
    (X^*(\uT),X_*(\uT),\uPhi,\uPhi^\vee) \simeq (X^*(T)^\cJ,X_*(T)^\cJ,\Phi^\cJ,\Phi^{\vee,\cJ}).
 \end{align*}
We have %$\uLam_R \simeq \Lam_R^\cJ$, 
$\uW\simeq W^\cJ$, $\uW_a\simeq W_a^\cJ$, $\utilW\simeq \tilW^\cJ$, and similarly for %$\uLam_R^\vee$, 
$\uW^\vee$, $\uW_a^\vee$, $\utilW^\vee$. We define an automorphism $\pi$ of $X^*(\uT)$ and $X_*(\uT^\vee)$ by the formula $\pi(\lam)_{\sig} =\pi(\lam)_{\sig\circ \varphi^\mo}$ for all $\lam \in X^*(\uT)$. 

%Let $\varphi$ be the absolute Frobenius on $\cO_p/p$ and its lift to $\cO_p$. If $S$ is a set and $s=(s_\sig)_{\sig \in \cJ} \in S^\cJ$, then we define $\pi(s)$ by $\pi(s)_\sig = s_{\sig \circ \varphi^\mo}$. When $\cO_p = \cO_K$, we fix an embedding $\sig_0: K \mono E$ and define $\sig_j := \sig_0\circ \varphi^{-j}$ for $j\in \Z/f \Z$. This identifies $\cJ$ with $\Z/f\Z$.

Let $G=\GL_n$. We write $\mathrm{I}_n$ for the identity matrix in $\GL_n$. We choose $B\subset G$ to be the upper triangular Borel subgroup. This determines the sets of postive roots $\Phi^+\subset \Phi$ and coroots $\Phi^{+,\vee}\subset \Phi^\vee$. We identify $X^*(T) \simeq \Z^n$ in the usual way. We write $X^*(T)_+\subset X^*(T)$ for the subset of dominant characters and $X^*(T)_-\subset X^*(T)$ for the subset of antidominant characters. For $i=1,\dots,n$, we let $\veps_i$ be the $i$-th standard basis and $\om_i$ be the $i$-th fundamental weight, i.e.~$\om_i=(1,\dots,1,0,\dots,0)$ with $i$-many $1$. For $I\subset \CB{1,\dots,n}$, we write $\veps_I := \sum_{i\in I}\veps_i$. We define $\eta:=(n-1,n-2,\dots,0) \in X^*(T)$. Sometimes, we write $\eta$ to denote a finite copy of $(n-1,n-2,\dots,0)$. For any integer $m$, we write $\ud{m}:=(m,m,\dots,m)\in X^*(T)$ as well as for its finite copy. If $\mu$ is a cocharacter and $a$ is a unit scalar, we write $a^\mu$ to denote $\mu(a)$. We denote by $w_0$ the longest element in $W$. We define the lowest $\eta$-shifted $p$-restricted alcove by
\begin{align*}
    \uC_0 := \CB{ \mu \in X^*(T)^\cJ \otimes \R \mid 0< \RG{\mu+\eta,\al^\vee} < p, \ \forall \al\in \uPhi^+}.
\end{align*}
We define the usual dot action of $\utilW$ on $X^*(\uT)$ by
\begin{align*}
    t_\nu w \cdot \mu := \nu+ w(\mu+\eta)-\eta.
\end{align*}
We denote by $\uOm$ the stablizer of $\uC_0$ in $\utilW$.

We write $\rmG := \GL_n(K)$, $\rmK:= \GL_n(\cO_K)$, $\rmK_1 := \ker( \rmK \epi \GL_n(k))$, $\rmT := T(K)$, and $\Iw$ for the inverse image of $B(k)$ under the map $\rmK \epi \GL_n(k)$. When we consider the group $\GL_n$ for various $n$ at the same time, we add subscript $n$ to various objects associated to $\GL_n$ to prevent confusion,~e.g.~$W_n,\tilW_n$, $\eta_n$, $\uC_{n,0}$,  etc. We write $\Ind$ for the unnormalized parabolic induction, $\cind$ for the compact induction, and $i_P^G$ for the normalized parabolic induction (\cite[\S1.4]{herzigAWS}). We write $\Diag(a_1,\dots ,a_n)$ for the diagonal matrix in $\GL_n$ with entries $a_1,\dots, a_n$.

\subsection{Tame inertial types}\label{subsec:tame-inertial-types}

Let $A$ be an $\cO$-algebra. A \textit{tame inertial type over $A$} is a representation $\tau: I_K \ra \GL_n(A)$ with an open kernel and extends to a representation $G_K \ra \GL_n(A)$. When $A=\cO$, we have the following data attached to $\tau$ (cf.~\cite[\S2.4 and 5.1]{LLLMlocalmodel}):
\begin{enumerate}
    \item Since the tame inertia is abelian, we can find an unramified extension $K'/K$ of say, degree $r$, with residue field $k'$ and an embedding $\sig'_0:W(k') \ra \cO$ extending $\sig_0$ such that $\tau\simeq \oplus_{i=1}^n \chi_i$ where $\chi_i = \omega_{K',\sig_0}^{\bfa'_i}$ with integers $\bfa'_i \in [0,p^{fr}-1)$ for $i=1,\dots,n$.
    \item We define $s_\tau \in W$ to be an element such that $\chi_{s_\tau(i)}^q \simeq \chi_i$ and if $\chi_{s_\tau^k(i)} \simeq \chi_i$ for $k\in \Z$ then ${s_\tau^k(i)} = i$. Such $s_\tau$ is not necessarily unique, but its choice will not affect our results. The order of $s_\tau$, denoted by $\abs{s_\tau}$, is called the \textit{level} of $\tau$. We can and do choose $K'$ so that $r = \abs{s_\tau}$.
    \item Let $\cJ':= \CB{\sig: W(k') \mono \cO}$. For each $j'\in \cJ'$, there is $\bal'_{j'} \in [0,p-1]^n$ ($p$-restricted cocharacter) such that for any $1 \le i\le n$, not all $\bal'_{j',i}$ is equal to $p-1$ and $\sum_{j'\in\cJ'} p^{j'}\bal'_{j'} = \bfa':= (\bfa'_1,\dots,\bfa'_n)$. Note that $\bal'_{j'}$ is uniquely characterized by these properties. Also, we have $s_\tau^\mo(\bal'_{j'}) = \bal'_{j'+f}$.
    \item Let $\bfa^{\prime(j')}:= \sum_{k=0}^{f'-1}\bal'_{-j'+k}p^k$; note that $\bfa^{\prime(0)}= \bfa'$ and $\bfa^{\prime(j')}  \equiv \bfa' p^{j'} \mod p^{fr}-1$.
    \item For each $j'\in\cJ'$, let $s_{\rmor,j'}\in W$ be an element such that $s_{\rmor,j'}^\mo(\bfa^{\prime(j')})$ is dominant and $s_{\rmor,j'+f} = s_\tau s_{\rmor,j'}$. Such elements always exist but are not necessarily unique. The element $s_{\rmor}:= (s_{\rmor,j'})_{j'\in\cJ'} \in W^{\cJ'}$ is called an \textit{orientation} of $\tau$.

\end{enumerate}

For each $s_\tau$ orbit $I\subset \{1,2,\dots,n\}$, let $\tau_I := \oplus_{i\in I}\chi_i$. We say \textit{$\tau$ is regular} if $\tau_I$ are pairwise non-isomorphic for all $I$. We say \textit{$\tau$ is a principal series} if $s_\tau=1$.

\begin{rmk}
    In the above notation, suppose that $\bfa^{\prime(f'-1)}$ is dominant. This implies that $\bal'_0$ is dominant. This can be arranged by reordering $\chi_1,\dots,\chi_n$. We choose $s_{\rmor,f'-1}=1$ so that $s_{\rmor,f-1}=s_\tau$.  % and the resulting $s_\tau$ and $s_{\rmor,j'}$ are the conjugation of those above by the element in $W$ corresponding to the reordering. 
    We further define
    \begin{enumerate}[(i)]
        \item $s= (s_{j})_{\jj}\in W^\cJ$ where $s_0 := s_{\rmor,0}$ and $s_{j} := s_{\rmor,j-1}^\mo s_{\rmor,j}$ for $j=1,\dots,f-1$, and
        \item $\mu = (\mu_{j})_{\jj} \in X^*(T)^{\cJ}$ so that $\mu_0 + \eta_0 = \bal'_0$ and $\mu_{j} + \eta_{j} = s_js_{j+1}\cdots s_{f-1}(\bal'_{f-j}) = s_{\rmor,{j-1}}^\mo s_\tau (\bal'_{f-j})$ for $j=1,\dots,f-1$.
    \end{enumerate}
     Note that $\mu_{j} + \eta_{j}$ is dominant by the definition of $s_{\rmor,j'-1}$. If $\bal'_{j'}$ is regular, i.e.~$\RG{\bal'_{j'},\al^\vee}\neq 0$ for all $\al\in \Del$, for all $j'\in\cJ'$, then $(s,\mu)$ is a \textit{lowest alcove presentation} of $\tau$ in the sense of \cite[\S2.4]{LLLMlocalmodel}. In that case, we write $\tau=\tau(s,\mu+\eta)$. 
\end{rmk}

For an integer $m \ge 1$, we say that a tame inertial type $\tau$ is \textit{$m$-generic} if it admits a lowest alcove presentation $\tau\simeq \tau(s,\mu+\eta)$ such that $m< \RG{\mu+\eta,\al^\vee}<p-m$ for all $\al\in \uPhi^+$. For a lowest alcove presentation $(s,\mu)$ of $\tau$ and an element $\tilw = t_\nu w \in \utilW$, we define
\begin{align*}
    \prescript{\tilw}{}{(s,\mu)} = (w s \pi(w)^\mo , \tilw\cdot \mu - w s \pi(w)^\mo\pi(\nu)).
\end{align*}

\subsection{Inertial local Langlands}\label{subsec:ILLC} Let $\tau$ be a tame inertial type over $E$. There is an association $\tau \mapsto \sig(\tau)$ called a tame \textit{inertial local Langlands correspondence} where $\sig(\tau)$ is a representation of $\GL_n(k)$ and satisfies the following property (\cite[Thm.~3.7]{6author}): {if $\pi$ is any irreducible smooth $\Qpbar$-representation of $\rmG$, then the restriction $\pi|_{\rmK}$ contains $\sig(\tau)$ (viewed as a $\rmK$-representation via inflation) as a $\rmK$-subrepresentation if and only if $\rec_p(\pi)|_{I_K} \simeq \tau$ and the monodromy operator $N$ is equal to $0$ on $\rec_p(\pi)$.} Now we explain this association in a more concrete term (cf.~\cite[Prop.~2.5.5]{LLLMlocalmodel}).

We define $\bfn:= \CB{1,2,\dots,n}$. Let $\fP_\tau$ be the set of $s_\tau$-orbits in $\bfn$. For $I\in \fP_\tau$ of size $m$, let $\tau_I := \oplus_{i\in I} \chi_i$. Note that $\tau_I$ are precisely sub inertial type of $\tau$ whose extension to $G_K$ is irreducible. For any $i\in I$, $\chi_i$ defines an $k$-primitive $\F_{q^m}^\times$-character. Then $\sig(\tau_I)$ is the cuspidal Deligne--Lusztig representation of $\GL_m(k)$ corresponding to $\chi_i$ (see \cite[Lem.~4.4]{HerzigDuke} and \cite[Prop.~2.4.1(i)]{EGHweightcyc}).

%Let $w_\tau \in W$ be an element such that for any $I\in \fP_\tau$, $w(I)\subset \bfn$ is an interval, and for any $\ov{I}\in \ov{\fP}_\tau$, $\cup_{I\in \ov{I}}w(I) \subset \bfn$ is an interval.  
Let $M \subset \GL_n$ be a standard Levi subgroup corresponding to the partition $\fP_\tau$ and let $\rmM:=M(K)$. The pair $(M(\cO_K),\lam_M)$ where $\lam_M := \otimes_{I\in \fP_\tau} \sig(\tau_I)$ is a semisimple Bushnell--Kutzko type for a Bernstein component $\Omega_M$ of $\rmM$. Let $P\subset \GL_n$ be the upper triangular parabolic subgroup whose Levi factor is $M$. Let $\cP\subset \rmK$ be the parahoric subgroup given by the inverse image of $P(k)\subset \GL_n(k)$ and $\lam$ be the inflation of $\lam_M$ along $\cP\epi M(k)$. The pair $(\cP, \lam)$ is a semisimple Bushnell--Kutzko type for a Bernstein component $\Omega$ of $\GL_n(K)$. Then $\sig(\tau)$ is the irreducible direct summand of $\Ind_\cP^\rmK \lam$ given by a certain idempotent endomorphism $e_{\rmK}$ as in \cite[\S3.6]{6author}. When $\tau$ is regular, $\Ind_\cP^\rmK \lam=\sig(\tau)$ is already irreducible.

\subsection{More notations on tame inertial types}\label{subsec:more-not}
Let $\tau$ be a tame inertial type over $E$. We introduce more notations associated to $\tau$. We write $\tau_M$ for $\tau$ viewed as an $M(E)$-valued tame inertial type. We define $W_{\tau_M}:= N_{G(E)}(\tau_M)/M(E)$ the relative normaliser. Then $W_{\tau_M}$ naturally acts on $\fP_{\tau}$. We define $\ov{\fP}_{\tau}$ to be the set of $W_{\tau_M}$-orbits in $\fP_{\tau}$. For $\ov{I}\in \ov{\fP}_{\tau}$, $\oplus_{I\in \ov{I}} \tau_I$ is the $\tau_I$-isotypic part of $\tau$ for some (equiv.~any) $I\in \ov{I}$. We can also lift the $W_{\tau_M}$-action on $\fP_{\tau}$ to $\bfn$: for $w\in W_{\tau_M}$ and $i\in I$ for some $I\in \fP_{\tau}$, $w(i)\in w(I)$ is the unique element such that $\chi_i \simeq \chi_{w(i)}$. In this way, we often view $W_{\tau_M}$ as a subgroup of $W$ (using the usual identification of $W$ with the permutation group on $\bfn$). Note that $W_{\tau_M}$-orbits in $\bfn$ are precisely the fibers of the map $i\mapsto \chi_i$. We denote by $\ov{\bfn}_\tau$ the set of $W_{\tau_M}$-orbits in $\bfn$.     

%We let  Let $\sim_\tau$ be an equivalence relation on $\bfn$ defined by $i\sim_\tau j$ if and only if $\chi_i \simeq \chi_j$. We write $\ov{\bfn}_\tau := \bfn/\sim_\tau$ and $\ov{\fP}_\tau:= \fP_\tau/\sim_\tau$.% to be the partition of $\fP_\tau$ given by the following equivalence relation: for $I,J\in \fP_\tau$, $I\sim J$ if and only if there is a (necessarily unique) bijection $\iota : I \ra J$ such that $\chi_{\iota(i)} \simeq \chi_i$ for all $i\in I$. 

%We usually write $I$ for an element in $\fP_{\tau}$ and $\ov{I}$ for an element in $\ov{\fP}_{\tau}$. We also write $\fI$ for a subset in $\fP_{\tau}$ and $\ov{\fI}$ for a subset in $\ov{\fP}_{\tau}$. We define $\norm{I}:=\#I$. Similarly, we define $\norm{\ov{I}}:= \sum_{I\in \ov{I}} \norm{I}$, $\norm{\fI}:= \sum_{I\in \fI} \norm{I}$, and $\norm{\ov{\fI}}:= \sum_{\ov{I}\in \ov{\fI}} \norm{\ov{I}}$. \Red{(!!) this is not used at the end?}

%We define the \textit{size} of $I$, $\ov{I}$, $\fI$, and $\ov{\fI}$ in the following way. For a set $S$, let $P(S)$ be its power set and $\cup: P^2(S) \ra P(S)$ be the map given by taking union. Then $I$, $\ov{I}$, $\fI$, and $\ov{\fI}$ are elements in $P^i(\bfn)$ for $i=1,2,2,3$, respectively. Then the size $\abs{I}$

It is convenient to assume that for any $\ov{I}\in \ov{\fP}_\tau$, any $I\in \ov{I}$ and $\cup_{I\in \ov{I}}I$ are intervals in $\bfn$. %\Red{(!!)Moreover, when $I=\{i, i+1, \dots, i+j\}$, we assume that $s_{\tau}(i+t)=i+t+1$ for $0 \ge t \ge j-1$ and $s_{\tau}(i+j)=i$.} 
Since this can be arranged by reordering the characters $\chi_i$, there is no harm in making this assumption.

\begin{rmk}
    In \S\ref{sec3}--\ref{sec4}, we work with an arbitrary tame inertial type $\tau$. Starting from \S\ref{sec5}, we assume that $\tau$ is regular, in which case the relative normaliser $W_{\tau_M}$ becomes trivial. Thus, if the reader is mainly interested in the regular case, they may ignore the role of $W_{\tau_M}$. 
\end{rmk}

\begin{example}\label{ex:main}
    We provide examples that motivate the notations introduced above. These examples will be repeatedly used throught the paper.  
    \begin{itemize}
        \item [(PS)] (Regular principal series) Suppose that $\tau\simeq \oplus_{i=1}^n \chi_i$ where $\chi_i=\omega_{K,\sig_0}^{\bfa_i}$ and $\bfa_i\in[0,q-1)$ for $i=1,\dots,n$ are distinct. In this case, $s_{\tau}$ trivial since $\chi_i^{q}=\chi_i$ and all $\chi_i$'s are distinct. Thus, the Levi subgroup $M$ is the torus $T$, and $W_{\tau_T}$ is trivial.

        \item [(C)] (Cuspidal type) Let $K'/K$ be an unramified extension of degree $n$ and $\chi:=\omega_{K',\sig_0}^{\bfa'}$ to be a primitive character (i.e.~its order does not divide $q^{i}-1$ for $0<i<n$). Suppose that $\tau \simeq \oplus_{i=0}^{n-1} \chi^{-q^i}$. In this case, $s_{\tau}$ is the $n$-cycle $(12\dots n)$. The Levi subgroup $M$ is $G$, and $W_{\tau_T}$ is trivial.

        \item [(Irr)] (An irregular case) Let $G=\GL_4$ and $K=\Qp$. Let $K'/\Qp$ be the quadratic unramified extension and $\chi:= \omega_{K',\sig_0}^{\bfa'}$ be a primitive character. Let $\tau_0:= \chi \oplus \chi^p$ and $\tau:= \tau_0\oplus \tau_0$. In this case, $s_{\tau}$ is $(12)(34)$. The Levi subgroup $M$ is isomophic to $\GL_2\times \GL_2$, and $W_{\tau_M}=\{1,w^M\}$ where $w^M$ acts on $\fP_{\tau} = \{ \{1,2\}, \{3,4\}\}$ by interchanging the two elements. On $\bfn$, $w^M$ acts by $(13)(24)$. %Note this identifies $W_{\tau_M}$ with $W^M$.
    \end{itemize}
\end{example}

\section{Integral and mod $p$ Hecke algebras}\label{sec3}

This section is devoted to integral and mod $p$ Hecke algebras. In \S\ref{subsec:hecke-generalities} and \ref{subsec:modp-satake}, we recall some basic results on Hecke algebras and the mod $p$ Satake isomorphism. In \S\ref{subsec:int-hecke-alg}, we give a presentation of an integral Hecke algebra $\cH(\rmG,\sig^\circ(\tau))$ (as well as its locally algebraic variant in \S\ref{subsec:loc-alg-hecke}) where $\tau$ is a tame inertial type and $\sig^\circ(\tau)\subset \sig(\tau)$ is a certain $\rmK$-stable $\cO$-lattice. We obtain this presentation from the results of Bushnell--Kutzko \cite{BKsmooth}, Schneider--Zink \cite{SZ}, and Dat \cite{Dat-caracteres} following \cite[\S3.6]{6author} closely. In \S\ref{subsec:hecke-modp}, we construct a morphism $\cH(\rmG,\sig^\circ(\tau)) \ra \cH(\rmG,\sig)$ for certain Jordan--H\"older factors $\sig$ of $\osig(\tau)=\sig^\circ(\tau)\mod \varpi$.

\subsection{Hecke algebras: generalities}\label{subsec:hecke-generalities}

Let $\rmG$ be a locally profinite group and $R$ be an $\cO$-algebra. We say that a $R[\rmG]$-module $V$ is \textit{smooth} if any vector $v\in V$ has an open stabilizer in $\rmG$.

\begin{defn}
    Let $\rmG$ be a locally profinite group and $\rmH\subset \rmG$ be an open compact subgroup. For an $\cO$-algebra $R$ and a smooth $R[\rmH]$-module $V$, we define the \textit{Hecke algebra of $V$} to be
\begin{align*}
    \cH(\rmG,V) := \End_{\rmG}(\ind_\rmH^{\rmG}V).
\end{align*}
We omit $\rmG$ in the notation when it is clear from the context.
\end{defn}

\begin{rmk}\label{rmk:hecke*function}
    We can identify $\cH(\rmG,V)$ with the set of functions $\psi: \rmG \ra \End_{R}(V)$ that is compactly supported and $\psi(h_1 g h_2) = h_1 \circ \psi(g) \circ h_2$ for all $g\in \rmG$ and $h_1,h_2\in \rmH$. For $g\in \rmG$ and $v\in V$, we denote by $[g,v]\in \ind_\rmH^{\rmG}V$ the function supported on $H g^\mo$ that maps $g^\mo$ to $v$. Note that $g'[g,v] = [g'g,v]$ and $[gh,v] = [g,hv]$ for $g'\in \rmG$ and $h\in \rmH$. For $\psi\in \cH(\rmG,V)$ and $[1,v]\in \ind_\rmH^{\rmG}V$, we have
\begin{align*}
    \psi([1,v]) = \sum_{g\in \rmH \backslash \rmG} [g^\mo, \psi(g)(v)] = \sum_{g\in \rmH \backslash \rmG} g^\mo[1, \psi(g)(v)].
\end{align*}
\end{rmk}

For any subgroup $\rmH\subset \rmG$ and $g\in \rmG$, we write $\rmH\ix{g} := g^\mo \rmH g$. If $V$ is a $\rmH$-representation, then we write $V\ix{g}$ to be the representation of $\rmH\cap \rmH\ix{g}$ such that $h\in \rmH\cap \rmH\ix{g}$ acts by $ghg^\mo\in \rmH$ via the original action.

We specialize to the case $\rmG=\GL_n(K)$ and $\rmH=\rmK = \GL_n(\cO_K)$. Given an antidominant cocharacter $\mu \in X_*(T)$, we define $P_\mu = M_\mu N_\mu$ to be the parabolic subgroup with Levi factor $M_\mu$ and unipotent radical $N_\mu$ given by the condition that $\al$-th entry vanishes for all roots $\al$ such that $\RG{\mu,\al^\vee}>0$. We define $\ov{P}_\mu = M_\mu \ov{N}_\mu$ to be the opposite parabolic. Note that
\begin{itemize}
    \item $\mu(\pi_K)$ centralize $M_\mu(\cO_K)$,
    \item $\ov{N}_\mu(\cO_K) \subset \ov{N}_\mu(\cO_K)\ix{\mu(\pi_K)}$, and
    \item $\ov{N}_\mu(\cO_K) , N_\mu (\cO_K)\ix{\mu(\pi_K)}, M_\mu (\cO_K)\subset \rmK \cap \rmK\ix{\mu(\pi_K)}$.
\end{itemize}
%$\mu(\pi)$ centralize $M_\mu(\cO_K)$, , and 

%Note that we have
%\begin{align*}
%    \ov{N}_\mu , N\ix{\mu(\pi)^\mo}_\mu, M_\mu \subset \rmK \cap \rmK\ix{\mu(\pi)^\mo}.
%\end{align*}

We recall some standard lemmas. 

\begin{lem}\label{lem:cind-decomp}
    Let $R$ be an $\cO$-algebra and $V$ be a smooth $R[\rmK]$-module. We have the following decomposition of $R[\rmK]$-modules
\begin{align*}
    \ind_\rmK^{\rmG} V |_{\rmK} &\simeq \oplus_{\mu \in X_*(T)_-} \ind_{\rmK}^{\rmK \mu(\pi_K)\rmK} V |_{\rmK}
    \\
    &\simeq \oplus_{\mu \in X_*(T)_-} \ind_{\rmK \cap\rmK\ix{\mu(\pi_K)}}^\rmK V\ix{\mu(\pi_K)}.
\end{align*}
\end{lem}
\begin{proof}
The first isomorphism follows from the Cartan decomposition. The second isomorphism is given by the map sending $f: \rmK \mu(\pi_K) \rmK \ra V $ to $g: \rmK \ra V$ such that $g(k) = f(\mu(\pi_K) k)$ for $k\in \rmK$.
%\begin{align*}
%    \{f: \rmK \mu(\pi) \rmK \ra V \mid f(kg) = kf(g) \ \forall k\in \rmK, \ g\in \rmG \} 
%    &\simeq \{f: \rmK \ra V\ix{\mu(\pi)} \mid f(k_1k_2) = k_1 f(k_2) \ \forall k_1\in \rmK\cap \rmK^{\mu(\pi)^\mo}, \ k_2 \in \rmK\}
%    \\
%    f &\mapsto (k\in \rmK \mapsto f(\mu(\pi)k) ).
%\end{align*}
\end{proof}

\begin{lem}\label{lem:hecke-decomp}
    Let $R$ be an $\cO$-algebra and $V$ be a smooth $R[G(k)]$-module. We view $V$ as a $R[\rmK]$-module via inflation. We have the following decomposition of $R$-modules
\begin{align*}
    \End_\rmG(\ind_\rmK^{\rmG}V)  \simeq \oplus_{\mu \in X_*(T)_-} \Hom_{M_\mu(k)}(V_{\ov{N}_\mu(k)}, V^{N_\mu(k)})
\end{align*}
which maps $\phi : \rmG \ra \End_{R}V$ to $\PR{\phi(\mu(\pi_K))}_{\mu \in X_*(T)_-} $. This is well-defined because $\phi(\mu(\pi_K))$ factors through $V \epi V_{\ov{N}_\mu(k)}$ and its image is contained in $V^{N_\mu}\subset V$.
\end{lem}

\begin{proof}
    This follows from the following series of isomorphisms
\begin{align*}
    \End_\rmG(\ind_\rmK^{\rmG} V) &\simeq \Hom_{\rmK}(V, \ind_\rmK^\rmG V) 
    \\
    &\simeq \oplus_{\mu\in X_*(T)_-} \Hom_{\rmK}(V, \ind_{\rmK \cap\rmK\ix{\mu(\pi_K)}}^{\rmK} V\ix{\mu(\pi_K)}) \\
    &\simeq \oplus_{\mu\in X_*(T)_-} \Hom_{\rmK\cap \rmK\ix{\mu(\pi_K)}}(V|_{\rmK\cap \rmK\ix{\mu(\pi_K)}}, V\ix{\mu(\pi_K)})
    \\
    &\simeq \oplus_{\mu\in X_*(T)_-} \Hom_{M_\mu(k)}(V_{\ov{N}_\mu(k)}, (V\ix{\mu(\pi_K)})^{N_\mu (\cO_K)\ix{\mu(\pi_K)}})
    \\
    &\simeq \oplus_{\mu \in X_*(T)_-} \Hom_{M_\mu(k)}(V_{\ov{N}_\mu(k)}, V^{N_\mu(\cO_K)}).
\end{align*}
    The first and the last isomorphisms are straightforward. The second isomorphism follows from Lemma \ref{lem:cind-decomp}. The third isomorphism is the Frobenius reciprocity and is given by composing the map sending $f:\rmK \ra V\ix{\mu(\pi_K)}$  to $f(1)$. The fourth isomorphism holds because $\rmK_1 \cap \rmK^{(\mu(\pi_K))}$ acts trivially on $V$ and $N_\mu(\cO_K)\ix{\mu(\pi_K)} \subset \rmK_1 \cap \rmK^{(\mu(\pi_K))}$ and similarly $\rmK \cap \rmK_1^{(\mu(\pi_K))}$ acts trivially on $V\ix{\mu(\pi_K)}$ and $\ov{N}_\mu(k) \subset \rmK \cap \rmK_1\ix{(\mu(\pi_K))}$.
\end{proof}

We record a variant of the previous lemmas using a parahoric subgroup of $\rmG$ instead of $\rmK$. Let $P=MN$ be a standard upper triangular parabolic subgroup of $\GL_n$. Let $\cP\subset \rmK$ be the parahoric subgroup defined as the preimage of $P(k)$. 

\begin{lem}\label{lem:hecke-decomp-para}
Let $V$ be an $R[P(k)]$-module which we view as an $R[\cP]$-module via inflation.
    \begin{enumerate}
        \item We have a decomposition of $R[\cP]$-modules
        \begin{align*}
            \ind_{\cP}^\rmG V \simeq \oplus_{\tilw \in W_M\bss\tilW/W_M} \ind_{\cP}^{\cP \tilw \cP} V \simeq \oplus_{\tilw \in  W_M\bss\tilW/ W_M} \ind_{\cP\cap \cP\ix{\tilw}}^{\cP} V\ix{\tilw}.
        \end{align*}
        
        \item For $\mu \in X_*(T)$ that is antidominant as a cocharacter for $M$, we have
        \begin{align*}
            \Hom_{\cP}(V,\ind_{\cP\cap \cP\ix{\mu(\pi_K)}}^{\cP} V\ix{\mu(\pi_K)}) &\simeq \Hom_{M(k)\cap M_\mu(k)} (V_{M(k)\cap \ov{N}_\mu(k)}, V^{M(k)\cap N_\mu(k)})\\
            \varphi &\mapsto \varphi(1).
        \end{align*}
    \end{enumerate}
\end{lem}
\begin{proof}
    This can be proven as Lemma \ref{lem:cind-decomp} and \ref{lem:hecke-decomp}.
\end{proof}

We end this subsection by recording the following well-known results on the transitivity of compact induction and their Hecke algebras.

\begin{lem}\label{lem:cind-transitive}
    Let $V$ be as in the previous Lemma. There is an isomorphism of $R[\rmG]$-modules
    \begin{align*}
         \ind_{\rmK}^\rmG( \ind_\cP^\rmK V))  &\risom \ind_{\cP}^\rmG V 
         \\
         \varphi:\rmG \ra \ind_\cP^\rmK V &\mapsto \psi: (kg\mapsto \varphi(g)(k)) \ \forall k\in \rmK, g\in \rmG.
    \end{align*}
\end{lem}

\begin{lem}\label{lem:hecke-transitive}
    Following the notation in the previous lemma, we view $\cH(\rmG,V)$ (resp.~$\cH(\rmG, \ind_\cP^\rmK V)$) as the space of $\cP$-biequivaraint (resp.~$\rmK$-biequivariant) functions from $\rmG$ to $\End_R (V)$ (resp.~$\End_R(\ind_\cP^\rmK V)$). Then the natural isomorphism between $\cH(\rmG,V)$ and $\cH(\rmG, \ind_\cP^\rmK V)$ is given by
    \begin{align*}
        \phi\in \cH(\rmG,V) \mapsto \psi: \biggl(g\in \rmG \mapsto \Bigl([k_1,v_1]\in\ind_\cP^\rmK V \mapsto (k_2\in \rmK \mapsto \phi(k_2 g k_1)(v)\in V ) \in \ind_\cP^\rmK V \Bigr)\biggr).
    \end{align*}
\end{lem}

\subsection{Hecke algebra of Serre weights}\label{subsec:modp-satake} A \textit{Serre weight} of $\GL_n(k)$ is an irreducible $\F$-representation of $\GL_n(k)$. Recall that the set of $q$-restricted dominant weights
\begin{align*}
    X_f^*(T):= \CB{\lam \in X^*(T) \mid 0 \le \RG{\lam,\al^\vee} \le p^f-1, \  \forall \al^\vee \in \Del^\vee}.
\end{align*}
We also have $X^0(T) := \CB{\lam \in X^*(T) \mid \RG{\lam,\al^\vee}=0, \  \forall \al^\vee \in \Del^\vee}$. 
For $\lam\in X_f^*(T)$, we define $F(\lam)$ to be the socle of the dual Weyl module of highest weight $\lam$, which we view as a $\GL_n(k)$-module via restriction (see \cite[\S3.1]{HerzigDuke} for detail). There is a bijection (Thm.~3.10 in \loccit)
\begin{align*}
    \frac{X_f^*(T)}{(q-1)X^0(T)} &\simeq \CB{\text{Serre weights of $\GL_n(k)$}}/\simeq \\
    \lam  &\mapsto F(\lam).
\end{align*}

Recall that we have the restricted group $G_0 = \Res_{\cO_{K_0}/\Zp}(G_{/\cO_{K_0}})$ and $G_0(\Fp)=\GL_n(k)$. We can classify Serre weights of $\GL_n(k)$ by applying the above formalism to $G_0$. Define the set of $p$-restricted dominant weights
\begin{align*}
    X_1^*(\uT):= \CB{\ud{\lam} \in X^*(\uT) \mid 0 \le \RG{\ud{\lam},\al^\vee} \le p-1, \  \forall \al^\vee \in \uDel^{\vee}}
\end{align*}
and $X^0(\uT):= \CB{\ud{\lam} \in X^*(\uT) \mid  \RG{\ud{\lam},\al^\vee} =0, \  \forall \al^\vee \in \uDel^{\vee}}$. Then we have a bijection (see \cite[Lem.~9.2.4]{GHS-JEMS-2018MR3871496})
\begin{align*}
    \frac{X_1^*(\uT)}{(p-\pi)X^0(\uT)} &\simeq \CB{\text{Serre weights of $\GL_n(k)$}}/\simeq \\
   \ud{\lam}  &\mapsto F(\ud{\lam}).
\end{align*}
If $\sig = F(\ud{\lam})$ for $\ud{\lam} \in X_1^*(\uT)$, then $\sig = F(\lam)$ where $\lam = \sum_{j=1}^{f} \lam_j p^{f-j} \in X_f^*(T)$.

Let $M\subset G$ be a Levi subgroup. For $\lam\in X^*_f(T)$, we write $F_M(\lam)$ for the Serre weight of $M(k)$ with highest weight $\lam$. 

\begin{defn}\label{defn:SW-generic}
    Let $\sig$ be a Serre weight of $\GL_n(k)$. We choose $\lam \in X_f^*(T)$ and $\ud{\lam}\in X_1^*(\uT)$ such that $\sig=F(\lam) = F(\ud{\lam})$.
    \begin{enumerate}
        \item We say that $\sig$ is \textit{non-Steinberg} if $\RG{\lam,\al^\vee} < p^f-1$ for all $\al^\vee \in \Del^\vee$.
        %\item We say that $\sig$ is \textit{regular} if $0<\RG{\lam,\al^\vee} < p^f-1$ for all $\al^\vee \in \Del^\vee$.
        \item Let $M\subset \GL_n$ be a standard Levi subgroup. We say that $\sig$ is \textit{$M$-regular} if $\Stab_{W}(\lam)\subset W_M$. We say that $\sig$ is \textit{regular} if it is $T$-regular. In other words, the $W$-orbit of $\lam$ contains $n!$ distinct characters of $T(k)$.
        %\item For a subset $S\subset W$, We say that $\sig$ is \textit{$S$-regular} if any elements in $S$-orbit of $\lam$ are distinct as characters of $T(k)$.
        \item A pair $(\tilw,\ud{\om})$ with $\tilw \in \utilW$ and $\ud{\om}\in \uC_0+\eta$ is called an \textit{lowest alcove presentation of $\sig$} if $\sig = F(\pi^\mo(\tilw)\cdot (\ud{\om}-\eta))$.  
        \item Let $m\in \Z_{\ge0}$ be an integer. We say that $\sig$ is \textit{$m$-deep} if it admits a lowest alcove presentation $(\tilw,\ud{\om})$ such that for each $\al \in \uPhi^{+}$, there exists an integer $n_\al$ such that
        \begin{align*}
                n_\al p + m < \RG{\ud{\om},\al^\vee} < (n_\al +1)p -m.
        \end{align*}
    \end{enumerate}
\end{defn}

\begin{lem}[{\cite[Lem~2.3]{HerzigModp}}]\label{lem:inv-coinv-SW}
    Let $P=MN$ be a standard parabolic of $\GL_n$. Then the natural $M(k)$-equivariant map  $\sig^{N(k)} \ra \sig_{\ov{N}(k)}$ is an isomorphism, and they are isomorphic to $F_M(\lam)$. In particular, when $P=B$, we have $\sig^{U(k)} \simeq \sig_{\ov{U}(k)} \simeq \lam$.
\end{lem}

As in \S\ref{subsec:hecke-generalities}, we define $\cH(\rmG,\sig):= \End_\rmG(\ind_\rmK^\rmG \sig)$ for a Serre weight $\sig=F(\lam)$. By Lemma \ref{lem:hecke-decomp} and \ref{lem:inv-coinv-SW}, $\cH(\rmG,\sig)$ has a basis $\ov{T}_\mu$ for $\mu\in X_*(T)_-$ where $\ov{T}_\mu$ has support $\rmK\mu(\pi_K)\rmK$ and $\ov{T}_\mu(\mu(\pi_K))$ is the inverse of the natural isomorphism $\sig^{N_{-\mu}(k)} \risom \sig_{{N}_{\mu}(k)}$ in Lemma \ref{lem:inv-coinv-SW}.

%\begin{rmk}
%    We can use the opposite Borel $\ov{B}$ and choose a basis of $\cH(\rmG,\sig)$ given by $\ov{T}_{w_0(\mu)}$ for $\mu \in X_*(T)_-$ where $\ov{T}_{w_0(\mu)}$ has support $\rmK w_0(\mu)(\pi_K)\rmK = \rmK \mu(\pi_K)\rmK$ and $\ov{T}_{w_0(\mu)} (w_0(\mu)(\pi_k))$ is the inverse of $\sig^{N_{-w_0(\mu)}(k)} \risom \sig_{{N}_{w_0(\mu)}(k)}$. In fact, $\ov{T}_{w_0(\mu)}$ is equal to $\ov{T}_\mu$. For $1 \le i \le n$, we write $\ov{T}_i =\ov{T}_{-\om_i}$.
%\end{rmk}

We also have the Hecke algebra $\cH(\rmT,\lam):= \End_{\rmT}(\ind_{T(\cO_K)}^\rmT \lam)$ of $\lam$ where we view $\lam$ as a character $T(k)\xra{\lam}k^\times \xra{\sig_0}\F^\times$. Similar to the previous paragraph, $\cH(\rmT,\lam)$ has a basis $\ov{T}_\mu$ for $\mu\in X_*(T)$ where $\ov{T}_\mu$ has support $T(\cO_K)\mu(\pi_K)$ and $\ov{T}_\mu(\mu(\pi_K))$ is simply the identity map of $\lam$. %Let $X_*(T)_-\subset X_*(T)$ be the subset of antidominant cocharacters. 
We define a subalgebra of $\cH(\rmT,\lam)$
\begin{align*}
    \cH^-(\rmT,\lam) := \CB{\phi\in \cH(\rmT,\lam) \mid \text{$\phi$ is supported on $\cup_{\mu\in X_*(T)_-} T(\cO_K)\mu(\pi_K)$}}.
\end{align*}
%and similarly define $\cH^+(\rmT,\lam)$.

Now we recall the mod $p$ Satake isomorphism. Note that in the following statements $\sig^{U(k)}\simeq \lam$.% and $\sig_{U(k)}\simeq w_0(\lam)$.

\begin{thm}[Thm.~1.2 in \cite{HerzigSatake}]\label{thm:mod-p-satake}
    There is an injective $\F$-algebra homomorphism
    \begin{align*}
        \cS_\sig: \cH(\rmG,\sig) &\ra \cH(\rmT,\sig^{U(k)})
        \\
        \phi &\mapsto \PR{t \mapsto \sum_{u\in U(K)/U(\cO_K)} \phi(tu)|_{\sig^{U(k)}} }
    \end{align*}
    with image $\cH^-(\rmT,\sig^{U(k)})$.
\end{thm}

%We can apply Theorem \ref{thm:mod-p-satake} using the opposite unipotent subgroup $\ov{U}$. Since $\sig^{\ov{U}(k)}\simeq \sig_{U(k)}$, we get the following variant of $\cS'_\sig$ (also, see \cite[Cor.~2.19]{HerzigModp}).

%\begin{cor}\label{cor:mod-p-satake}
%    There is an injective $\F$-algebra homomorphism
%    \begin{align*}
%        \cS_\sig : \cH(\rmG,\sig) \ra \cH(\rmT, \sig_{U(k)})
%    \end{align*}
%    with image $\cH^+(\rmT,\sig_{U(k)})$.
%\end{cor}

\begin{rmk}\label{rmk:modp-hecke-presentation}
    We identify $\cH(\rmT,\lam)$ with $\F[x_1^\pm,\dots,x_n^\pm]$ by identifying $\ov{T}_{-\veps_i}$ with $x_i$ for all $1\le i \le n$. Then the subalgebra $\cH^-(\rmT,\lam)$ is identified with $\F[y_1,\dots,y_{n-1},y_n^\pm]$ where $y_i := x_{1}x_2 \dots x_{i}$. Suppose that $\sig$ is regular. By \cite[Prop.~1.4]{HerzigSatake}, under the isomorphism
    \begin{align*}
        \cH(\rmG,\sig) \risom \cH^-(\rmT, \lam) \risom \F[y_1,\dots,y_{n-1},y_n^\pm],
    \end{align*}
    $\ov{T}_{i}=\ov{T}_{-\om_i}\in \cH(\rmG,\sig)$ is mapped to $\ov{T}_{-\om_i} = y_i\in \cH^-(\rmT,\lam)$.
\end{rmk}

\subsection{Integral Hecke algebras of tame types}\label{subsec:int-hecke-alg}  
In this subsection, we freely use the notations introduced in \S\ref{subsec:tame-inertial-types}--\ref{subsec:more-not}. For the remainder of the section, fix a tame inertial type $\tau$ and write $\sig(\tau)$ for the $\rmK$-representation constructed in \S\ref{subsec:ILLC}. We are interested in the Hecke algebra $\cH(\rmG,\sig^\circ(\tau))$ where $\sig^\circ(\tau)$ is a $\rmK$-stable $\cO$-lattice in $\sig(\tau)$.

Let $\cH(\rmM,\lam_M)$ be the Hecke algebra of $\lam_M$. By its construction, $\lam_M$ is \textit{cuspidal}, meaning that for the unipotent radical $\mathrm{N}'$ of any proper parabolic subgroup of $\rmM$, the $\mathrm{N}'$-coinvariant $(\lam_{M})_{\mathrm{N}'}$ is trivial. By Lemma \ref{lem:hecke-decomp}, this implies that $\cH(\rmM,\lam_M)$ has a basis $t_\mu$ for $\mu \in X_*(Z(M))$ such that $t_\mu$ has support $M(\cO_K)\mu(\pi_K)$ and $t_\mu(\mu(\pi_K))$ is the identity map of $\lam_M$. Note that we view $X_*(Z(M))$ as a subgroup of $X_*(T)$ via $Z(M)\mono T$. Similarly, there is a unique Hecke operator $T_\mu \in \cH(\rmG,\lam)$ for $\mu\in X_*(Z(M))$ such that it has support $\cP \mu(\pi_K)\cP$ and $T_\mu(\mu(\pi_K))$ is the identity map of $\lam$. 

As explained in \cite[\S3.6]{6author}, we have maps between Hecke algebras
\begin{align*}
    t_P &: \cH(\rmM,\lam_M) \ra \cH(\rmG,\lam) \\
    s_P &: \cH(\rmG,\lam) \ra \cH(\rmG,\sig(\tau)).
\end{align*}
The construction of $t_P$ goes back to \cite[Cor.~7.12]{BKsmooth}. We use its normalized version in \cite{BKssGLn}, which makes $t_P$ being independent of the choice of $P$. It is the unique map such that for $\mu\in X_*(Z(M))$ that is dominant as a cocharacter for $T\subset \GL_n$, $t_P$ maps $t_\mu$ to $\del_P^{1/2}(\mu(\pi_K))T_{\mu}$. Here, $\del_P$ is the modulus character sending $g\in P(K)$ to $\abs{\det(\Ad_N(g))}_{K}$ where $\Ad_N(g)$ is the adjoint action of $g$ on $\Lie N$. Note that the relative normaliser $W_{\tau_M}$ naturally acts on $\cH(\rmM,\lam_M)$. Then, \cite[Prop.~2.1]{Dat-caracteres} shows that $t_P$ restricts to an isomorphism $\cH(\rmM,\lam_M)^{W_{\tau_M}} \risom Z(\cH(\rmG,\lam))$. %Here, $W_{[M,\omega]}$ is the relative normalizer defined in \cite[\S3.6]{6author}. It can be identified with the group of permutations on the set $\fP_{\tau}$ that maps $I\in \fP_\tau$ to $J$ such that $J\sim_\tau I$. In other words, 
%\begin{align*}
%    W_{[M,\omega]} \simeq \fS_{\ov{\fP}_\tau}:= \prod_{\ov{I}\in \ov{\fP}_\tau}  S_{\ov{I}}
%\end{align*}
%where  $S_{\ov{I}}$ is the symmetric group on the set $\ov{I}$. 

We define $s_P$ by $s_P(\phi)=e_K * \phi * e_K$ for $\phi\in \cH(\rmG,\lam)$ where $e_{\rmK}$ is a certain idempotent in $\cH(\rmG,\lam)$ explained below (see also, \cite[\S4]{Dat-caracteres}). 
Note that $s_P$ is only a morphism of modules. Since the image of $\cH(\rmM,\lam_M)^{W_{\tau_M}}$ under $t_P$ is contained in the center of the target, $s_P \circ t_P : \cH(\rmM,\lam_M)^{W_{\tau_M}} \ra \cH(\rmG,\sig(\tau))$ is an algebra homomorphism and is an isomorphism by Thm.~4.1 in \loccit. Note that when $\tau$ is regular, $s_P$ is just the identity map.

As a conclusion of the above discussion, we have the following presentation of $\cH(\rmG,\sig(\tau))$.

%We first discuss their rational structures.

%\begin{lem}
%    The Hecke algebra $\cH(M,\lam_M)$ is isomorphic to $E[x_I^{\prime\pm}]_{I\in \fP_\tau}$ and $W_{[M,\omega]}\simeq \prod_{\ov{I}\in \ov{\fP}_\tau} S_{\ov{I}}$ acts by permuting $x_I'$.
%\end{lem}

\begin{lem}\label{lem:hecke-tame-type-rational-presentation}
    There is an isomorphism of $E$-algebras
    \begin{align*}
        \cH(\rmG,\sig(\tau)) &\simeq E[x_I^{\pm}]_{I\in \fP_{\tau}}^{W_{\tau_M}}
    \end{align*}
    which maps $s_P\circ t_P(t_{-\veps_I})$ to $x_I$ for $I\in \fP_\tau$.
\end{lem}

\begin{example}
    We describe $\cH(\rmG,\sig(\tau))$ explicitly for $\tau$ as in Example \ref{ex:main}.
    \begin{itemize}
        \item [(PS)] Since $\bfn=\fP_\tau$ and $W_{\tau_M}$ is trivial, we have $\cH(\rmG,\sig(\tau))\simeq E[x_1^\pm,\dots, x_n^\pm]$ where $x_i$ is identified with $q^{-(n-1)/2}T_{-\veps_i}$. The quantity $q^{-(n-1)/2}$ can be explained by Lemma \ref{lem:tame-type-formulas}.

        \item [(C)] In this case, $\fP_{\tau}= \{\bfn\}$ and $W_{\tau_M}$ is trivial. Thus, $\cH(\rmG,\sig(\tau)) \simeq E[x_{\bfn}^\pm]$ where $x_{\bfn}$ is identified with $T_{-\veps_{\bfn}}$.

        \item [(Irr)] Let $I_1=\{1,2\}$ and $I_2=\{3,4\}$ denote the two elements in $\fP_{\tau}$. Then $\cH(\rmG,\sig(\tau)) \simeq E[x_{I_1}^\pm,x_{I_2}^\pm]^{W_{\tau_M}} \simeq E[x_{I_1}+x_{I_2},  (x_{I_1}x_{I_2})^\pm]$.
    \end{itemize}
\end{example}

Now we consider integral structures. We choose a $\rmK$-stable $\cO$-lattice $\sig^\circ(\tau)\subset \sig(\tau)$ in the following way. For each $I\in \fP_\tau$ of size $m$, we choose a $\GL_m(k)$-stable $\cO$-lattice $\sig^\circ(\tau_I)\subset \sig(\tau_I)$. This induces $\cO$-lattices $\lam_M^\circ \subset \lam_M$, $\lam^\circ \subset \lam$, and $\Ind_\cP^\rmK \lam^\circ \subset \Ind_\cP^\rmK \lam$. Recall that $\sig(\tau)$ is the direct summand of $\Ind_\cP^{\rmK} \lam$ given by the idempotent $e_{\rmK}$. We define $\sig^\circ(\tau):= e_{\rmK} (\Ind_\cP^{\rmK} \lam^\circ)$. %Since $\sig(\tau)$ is an irreducible direct summand of $\Ind_\cP^\rmK \lam$ and likewise $\cH(\sig(\tau))$ is a direct summand of $\cH(G,\lam)$, the morphism $s_P$ maps $\cH(G,\lam^\circ)$ to $\cH(G,\sig^\circ(\tau))$. 
The map $s_P$ restricts to $\cH(\rmG,\lam^\circ) \ra \cH(\rmG,\sig^\circ(\tau))$. We first prove the following technical result. %We define $\cH'(\rmG,\sig^\circ(\tau))$ to be the image of $Z(\cH(\rmG,\lam^\circ))$ under $s_P$.

\begin{prop}\label{prop:surj}
    The map $s_P$ restricts to an $\cO$-algebra isomorphism $Z(\cH(\rmG,\lam^\circ)) \risom \cH(\rmG,\sig^\circ(\tau))$.
\end{prop}

We first need an explicit description of $e_{\rmK}$. Recall that we can view $\lam$ as a representation of $P(k)$. We can identify $\cH(G(k),\lam)$ as the subalgebra of $\cH(\rmG,\lam)$ consisting of functions supported on $\rmK$. Recall from \S\ref{subsec:more-not} we view $W_{\tau_M}$ as a subgroup of $W$. Let $T_{w}\in \cH(G(k),\lam)$ be the function supported on $P(k)w P(k)$ that maps $w$ to the operator permuting the factors of $\lam=\otimes_{I\in \fP_\tau}\sig(\tau_I)$ using the action of $w$ on $\fP_\tau$, i.e.~for $v=\otimes_{I\in \fP_\tau}v_I \in \lam$, $T_{w}(w)(v)=\otimes_{I\in \fP_\tau}v_{w(I)}$. We also introduce a normalized length function on $W_{\tau_M}$. For each $\ov{I}\in \ov{\fP}_\tau$, we write $\# \ov{I}$ for the size of $\ov{I}$ and $n_{\ov{I}}:= \#I$ for some (equiv.~any) $I\in \ov{I}$. Then we can identify $W_{\tau_M}$ with $\prod_{\ov{I}\in \ov{\fP}_\tau} W(\GL_{\# \ov{I}})$. For $w\in W_{\tau_M}$ corresponding to $(w_{\ov{I}}) \in \prod_{\ov{I}\in \ov{\fP}_\tau} W(\GL_{\# \ov{I}})$, we define
%\Red{(!!) $\tilw$ is not a good idea}
\begin{align*}
    \til{\ell}(w) := \sum_{\ov{I}\in \ov{\fP}_\tau} \ell(w_{\ov{I}})n_{\ov{I}}(n_{\ov{I}}-1)/2
\end{align*}
where $\ell(w_{\ov{I}})$ is the usual length of $w_{\ov{I}}\in W(\GL_{\# \ov{I}})$.

\begin{lem}\label{lem:e_K}
    We have $e_{\rmK} = \sum_{w\in W_{\tau_M}}\eps(-1)^{\ell(w)} q^{-\til{\ell}(w)}T_{w}$ where $\eps(-1)\in \{\pm 1\}$. 
\end{lem}

%The factor $[G(k):P(k)]^\mo$ is due to our choice of the Haar measure, which we normalize so that the measure of $\cP$ is 1. Note that $p \nmid [G(k):P(k)]$.

\begin{proof}
   We consider the case of simple type, i.e.~$M\simeq (\GL_{n'})^{d}$ with $n=d{n'}$ and $\lam =  \lam_0^{\otimes {d}}$ for some cuspidal representation $\lam_0$ of $\GL_{n'}(k)$. Then there is an algebra isomorphism \cite[Theorem 5.1]{howe-moy}
    \begin{align*}
       \Psi: \cH(\GL_{d}(k_{n'}),1_{B_{d}(k_{n'})}) \simeq \cH(G(k),\lam)
    \end{align*}
    where $k_{n'}/k$ is a finite extension of degree $n'$, $B_{d}\subset \GL_{d}$ is the upper triangular Borel subgroup, and $1_{B_{d}(k_{n'})}$ is the trivial representation of $B_{d}(k_{n'})$. Then $e_{\rmK}$ is defined to be the image of the characteristic function on $\GL_{d}(k_{n'})$ under $\Psi$. (Caution: the characteristic function on $\GL_{d}(k_{n'})$ is an idempotent with respect to the choice of Haar measure such that $\mu(\GL_{d}(k_{n'}))=1$ and one has to choose the Haar measure for $G(k)$ and $\rmG$ accordingly. Perhaps, more natural choice of Haar measures is the one such that $\mu(B_d(k_{n'}))=1$ on $\GL_{d}(k_{n'})$ and $\mu(\cP)=1$ on $\rmG$. Since different choices of Haar measure change $e_{\rmK}$ up to an integer prime-to-$p$, we decided to ignore this issue.)   In this case, $W_{\tau_M}$ can be identified with $W(\GL_{d})$. For each $w\in W(\GL_{d})$, let $t_{w}\in\cH(\GL_{d}(k_{n'}),1_{B_{d}(k_{n'})})$ be the characteristic function on $B_{d}(k_{n'}) w B_{d}(k_{n'})$. Then the characteristic function on $\GL_{d}(k_{n'})$ is $\sum_{w\in W(\GL_{d})}t_{w}$. %Note that there is a natural injection $W_r \mono W_M \bss W / W_M$ which we denote by $w\mapsto \til{w}$. For $w\in W_r$, let $T_w\in\cH(\GL_r(k_d),1_{B_r(k_d)})$ (resp.~$T_{\til{w}}\in \cH(G(k),\lam)$) be the function supported on $B_r(k_d) w B_r(k_d)$ (resp.~$\cP \til{w} \cP$) that maps $w$ to the identity map of $1_{B_r(k_d)}$ (resp.~$\til{w}$ to the operator permuting the factors of $r$-fold tensor product $\lam=\lam_0^{\otimes r}$ by the permutation $w$). 
    As explained in \loccit, $\Psi$ maps $t_w$ to $\eps(-1)^{\ell(w)}q^{-\ell(w)n'(n'-1)/2}T_{{w}} = \eps(-1)^{\ell(w)}q^{-\til{\ell}(w)}T_{{w}}$ where $\eps(-1)\in \{\pm 1\}$ such that $(-\mathrm{I}_{n'})v=\eps(-1)v$ for $v\in \lam_0$. This proves the claim.

    The general case follows easily from the simple type case (see \cite[\S4.2]{Dat-caracteres}).
\end{proof}

%\begin{rmk}
%    As we discuss in the next example, it is not clear whether $s_P$ maps $Z(\cH(\rmG,\lam^\circ))$ surjectively onto $\cH(\rmG,\sig^\circ(\tau))$. However, when $\tau$ is regular, $s_P$ is the identity map, and $\cH(\rmG,\lam^\circ)$ is commutative and is equal to $\cH'(\rmG,\sig^\circ(\tau))=\cH(\rmG,\sig^\circ(\tau))$. 
%\end{rmk}

\begin{example}
    We consider Example \ref{ex:main}(Irr). Let $\lam_0:=\sig(\tau_0)$ be the cuspidal representation of $E[\GL_2(\Fp)]$ associated to $\tau_0$. We define a $M(\Fp)$-representation $\lam:= \lam_0\otimes \lam_0$. We identify $\cH(G(\Fp),\lam):=\End_{G(\Fp)}(\Ind_{P(\Fp)}^{G(\Fp)} \lam)$ with the subalgebra of $\cH(\rmG,\lam)$ consisting of Hecke operators with support contained in $\rmK$. %Let $w^M\in W^M$ be the unique nontrivial element. Let $T_1$ (resp.~$T_{w^M}$) be the element in $\cH(G(\Fp),\lam)$ with support $P(\Fp)$ (resp.~$P(\Fp)w^M$) which maps $\mathrm{I}_4$ to the identity map on $\lam$ (resp.~$w^M$ to the map on $\lam$ taking $v_1\otimes v_2$ to $v_2\otimes v_1$). %We normalize $T_{w^M}$ by $T_{w^M}^*:= \eps(-1)T_{w^M}/p$ where $\eps(-1)$ is the sign of the endomorphism of $\lam_0$ given by $-\mathrm{I}_2$. 
    We have
    \begin{align*}
        e_{\rmK} :=T_1+\eps(-1)p^\mo T_{w^M}.
    \end{align*}
    %is an idempotent in $\cH(G(\Fp),\lam)$ (see \cite[\S5]{howe-moy}). (To be precise, the factor $1/(1+p^2)$ is due to our choice of measure. Here, we normalize the counting measure on $\GL_4(\Fp)$ so that the measure of $P(\Fp)$ is one.) 
    Then $\sig(\tau)= e_{\rmK}(\Ind_{P(\Fp)}^{G(\Fp)}\lam)$ is the irreducible direct summand of $\Ind_{P(\Fp)}^{G(\Fp)}\lam$ associated to $\tau$.

    Let $\lam_0^\circ\subset \lam_0$ be a $\GL_2(\Fp)$-stable $\cO$-lattice. This induce an $\cO$-lattice $\Ind_{P(\Fp)}^{G(\Fp)}\lam^\circ$ and $\sig^\circ(\tau) := e_{\rmK}(\Ind_{P(\Fp)}^{G(\Fp)}\lam^\circ)$. Due to the $p^\mo$ in the formula of $e_{\rmK}$, $\sig^\circ(\tau)$ is not contained in $\Ind_{P(\Fp)}^{G(\Fp)}\lam^\circ$ but only contained in $p^\mo\Ind_{P(\Fp)}^{G(\Fp)}\lam^\circ$. Thus, the surjection $\Ind_{P(\Fp)}^{G(\Fp)}\lam^\circ\epi \sig^\circ(\tau)$ does not split. This is why a priori it is unclear whether
    \begin{align*}
        e_{\rmK}: Z(\cH(\rmG,\lam^\circ)) \ra \cH(\rmG,\sig^\circ(\tau))
    \end{align*}
    is surjective. 
    (This does not happen if one replaces $P\subset \GL_4$ by $B\subset \GL_2$. For example, the direct sum decomposition of $\Ind_{B(\Fp)}^{\GL_2(\Fp)} E$ into the trivial representation and the Steinberg representation holds integrally.)
\end{example}

\begin{lem}\label{lem:O-basis1}
    The set of $\phi_{\mu}:=\sum_{\mu'\in W_{\tau_M}(\mu)} T_{\mu'}$ for $\mu\in X_*(Z(M))$ is an $\cO$-basis of $Z(\cH(\rmG,\lam^\circ))$.
\end{lem}
\begin{proof}
    Recall the isomorphism $\cH(\rmM,\lam_M)^{W_{\tau_M}}\simeq Z(\cH(\rmG,\lam))$ induced by $t_P$. By Lemma \ref{lem:tame-type-formulas}(2) and Remark \ref{rmk:convolution} below, 
    \begin{align*}
        t_P(\sum_{\mu'\in W_{\tau_M}(\mu)} t_{\mu'}) = \delta_P^{1/2}(-\mu(\pi_K))\phi_{\mu}.
    \end{align*}
    This explains that $\phi_{\mu}$ is in $Z(\cH(\rmG,\lam))$. 
%    First of all, $\phi_{\mu}$ is in $Z(\cH(\rmG,\lam))$ because $t_P(w(t_{\mu})) = w(t_P(t_\mu))$ for $w\in W_{\tau_M}$ by Lemma \ref{lem:tame-type-formulas}(2) below. Furthermore, the set of $\phi_{\mu}$ for all $\mu\in X_*(Z(M))_-$ forms an $E$-basis of $Z(\cH(\rmG,\lam))$ by the isomorphism $\cH(\rmM,\lam_M)^{W_{\tau_M}}\simeq Z(\cH(\rmG,\lam))$. 
Then the claim follows from $Z(\cH(\rmG,\lam^\circ)) = Z(\cH(\rmG,\lam)) \cap \cH(\rmG,\lam^\circ)$ and the decomposition in Lemma \ref{lem:hecke-decomp-para}. 
\end{proof}

\begin{rmk}\label{rmk:P-to-K}
    We can apply Lemma \ref{lem:hecke-transitive} to describe $\phi_{\mu}$ as a function from $\rmG$ to $\End(\Ind_{\cP}^{\rmK}\lam^\circ)$ which we denote by $\psi_\mu$. Then $\psi_{\mu}$ is supported on $\rmK \mu(\pi_K) \rmK$ and $\psi_\mu(\mu(\pi_K))$ maps $[k_1,v]$ to a function given by
\begin{align*}
    k_2\in \rmK \mapsto \phi_{\mu}(k_2 \mu(\pi_K) k_1)(v) \in \lam^\circ.
\end{align*}
%In particular, $Z(\cH(\rmG, \Ind_{\cP}^{\rmK}\lam^\circ))$ admits an $\cO$-basis given by $\psi_\mu$ for $\mu\in X_*(Z(M))_-$.
\end{rmk}

\begin{lem}\label{lem:O-basis2}
    The set of $s_P(\phi_{\mu})$ for $\mu\in X_*(Z(M))$ is an $\cO$-basis of $\cH(\rmG,\sig^\circ(\tau))$.
\end{lem}

\begin{proof}

Since $s_P$ induces an isomorphism $Z(\cH(\rmG,\lam)) \simeq \cH(\rmG,\sig(\tau))$, the set of $s_P(\phi_\mu)$ for $\mu \in X_*(Z(M))$ for an $E$-basis of $\cH(\rmG,\sig(\tau))$ by Lemma \ref{lem:O-basis1}.

For any $\psi \in Z(\End_{\rmG}(\cind_\rmK^\rmG\Ind_{\cP}^{\rmK}\lam))$ and $f\in \cind_{\rmK}^\rmG \Ind_{\cP}^{\rmK}\lam$ (so that $e_{\rmK}(f) \in \cind_\rmK^\rmG \sig(\tau)$), we have
    \begin{align*}
        s_P(\psi)(e_{\rmK}(f)) = e_{\rmK} \circ \psi \circ e_{\rmK} (e_{\rmK} (f)) = \psi \circ e_{\rmK}(f).
    \end{align*}
    Thus, $s_P(\psi)$ is equal to the restriction $\psi|_{\cind_\rmK^\rmG \sig(\tau)}$. In turn, for $s_P(\psi)$ viewed as a function from $\rmG$ to $\End(\sig(\tau))$, $s_P(\psi)(g)$ is equal to $\psi(g)|_{\sig(\tau)}$ for all $g\in \rmG$. Because of this, we simply write $\psi$ instead of $s_P(\psi)$ from now on. As explained in Remark \ref{rmk:P-to-K}, $\psi_\mu$ is supported on $\rmK \mu(\pi_K) \rmK$. 

    Let $P'=M'N'$ be the minimal standard parabolic subgroup containing $P$ and $W_{\tau_M}$. By definition, $\psi_\mu = \psi_{w(\mu)}$ for $w\in W_{\tau_M}$. Thus, we pose a kind of anti-dominant condition on $\mu$: $\RG{\al,\mu}\le 0$ for $\al \in \Phi^+_{M'}$. Then there is a unique such $\mu$ for each $\psi_\mu$.  If $\mu' \in X_*(Z(M))$ satisfies $\mu'\neq \mu$ and $\mu' = w(\mu)$ for some $w\in W \backslash W_{\tau_M}$, then $\psi_\mu$ and $\psi_{w(\mu)}$ are distinct operators with the same support. Let $B_\mu$ to be the set of such $\mu'$ satisfying the anti-dominance condition $\RG{\al,\mu'}\le 0$ for $\al \in \Phi^+_{M'}$. We claim that the set of $\psi_{\mu'}$ for $\mu'\in B_\mu$ form an $\cO$-basis of the subspace of operators in $\cH(\rmG,\sig^\circ(\tau))$ supported on $\rmK \mu(\pi_K)\rmK$. Equivalently, we show that for the linear combination
    \begin{align*}
        \psi:= \sum_{\mu'\in B_\mu} c_{\mu'}\psi_{\mu'}
    \end{align*}
    with $c_{\mu'}\in \cO$, its image in $\cH(\rmG,\sig^\circ(\tau))/\varpi$ is non-zero unless $c_{\mu'}\equiv 0 \mod \varpi$ for all $\mu' \in B_\mu$. % $\cH(\rmG,\sig^\circ(\tau))/\varpi$. 
    This will be verified by evaluting $\psi_{\mu'}$ at a specific vector. 

    %Recall the isomorphism $W_{\tau_M} \simeq \prod_{\ov{I}\in \ov{\fP}_\tau} W(\GL_{\#\ov{I}}$. We define $W_1:= \prod_{\ov{I}\in \ov{\fP}_\tau, n_{\ov{I}}=1} W(\GL_{\#\ov{I}}$ and $W_{2}:=\prod_{\ov{I}\in \ov{\fP}_\tau, n_{\ov{I}}>1} W(\GL_{\#\ov{I}}$ so that $W_{\tau_M}\simeq W_1 \times W_2$. This gives a factorization  $e_{\rmK} = e_1 e_2$ where $e_i = \sum_{w\in W_i}\eps(-1)^{\ell(w)} q^{\til{\ell}(w)} T_w$ for $i=1,2$. By the definition of $\til{\ell}$, $\til{\ell}(w)=0$ for $w\in W_1$, and there is a unique element $w\in W_2$ that $\til{\ell}(w)$ attains the maximal value. In particular, $e_1$ induces an integral decomposition
    %\begin{align*}
    %    \Ind_\cP^\rmK \lam^\circ = e_1(\Ind_\cP^\rmK \lam^\circ) \oplus (1-e_1)(\Ind_\cP^\rmK \lam^\circ).
   %\end{align*}
    %Since $e_1$ commutes with $e_2$, it also induces a decomposition of $e_2(\Ind_\cP^\rmK \lam^\circ)$ into $\sig^\circ(\tau)$ and $(1-e_1)e_2(\Ind_\cP^\rmK \lam^\circ)$. 

    %By Lemma \ref{lem:hecke-decomp}, $\psi(\mu(\pi_K))$ is in $\Hom_{M_\mu(k)}(\sig(\tau)_{\ov{N}_\mu(k)},\sig(\tau)^{N_\mu(k)})$. 
    Let $w_{0,\tau_M}\in W_{\tau_M}$ be the longest element. By Lemma \ref{lem:e_K}, $\sig^\circ(\tau) \subset q^{-\til{\ell}(w_{0,\tau_M})} \Ind_\cP^\rmK \lam^\circ$. If $\psi \equiv 0 \mod \varpi$, then
    \begin{align*}
        \psi(\mu(\pi_K))(e_{\rmK}([1,v])) \in \varpi q^{-\til{\ell}(w_{0,\tau_M})}\Ind_{\cP}^\rmK \lam^\circ
    \end{align*}
    for any $v\in \lam^\circ \setminus \varpi \lam^\circ$. To deduce a contradiction, we show that $\psi(\mu(\pi_K))(e_{\rmK}([1,v]))(w_{0,\tau_M}) \notin \varpi q^{-\til{\ell}(w_{0,\tau_M})}\lam^\circ$ unless $c_\mu \equiv 0\mod \varpi$. By applying the same argument for $\mu$ replaced by $\mu' \in B_\mu$, this proves the claim. % \Red{(!!) notation here seems too lengthy...}

    %Recall from Lemma \ref{lem:e_K} that $e_{\rmK} = \sum_{w\in W_{\tau_M}} q^{-\til{\ell}(w)}T_w$. 
    For $w\in W_{\tau_M}$, we have
    \begin{align*}
        \cP w \cP =  \cP w \til{U}_w, \ \ \til{U}_w:=\BR{\prod_{\al\in \Phi^+ \cap w^\mo(\Phi^-)} U_{\al}(k)} %\subset (N\cap M' )(\cO_K)
    \end{align*}
    where the bracket in the definition of $\til{U}_w$ denotes the Teichm\"uller lift. By the way how we view $W_{\tau_M}$ as a subgroup of $W$, $w$ is the unique minimal length representative of the double coset $W_M w W_M$. This implies that $\Phi^+ \cap w^\mo(\Phi^-)$ is contained in $\Phi^+_{M'}$ and thus $\til{U}_w \subset (N\cap M')(\cO_K)$. Similarly, we can see that $w\til{U}_w w^\mo \subset (\ov{N}\cap M')(\cO_K)$. %This in particular implies that both $w(\Phi^+_M) \cap \Phi^-$ and $\Phi^+ \cap w^\mo(\Phi^-_M)$ are empty. Thus, we have $w \til{U}_{w} w^\mo \subset \ov{N}(\cO_K)$. Since $w\in W_{\tau_M}$ maps $\Phi^+\backslash \Phi^+_{M'}$ to itself, we furthermore have $w\til{U}_w w^\mo \subset \ov{N}(\cO_K)\cap M'(\cO_K)$. 
    
    By Lemma \ref{lem:e_K} and Remark \ref{rmk:hecke*function}, we have
    \begin{align*}
        e_{\rmK}([1,v]) = \sum_{w\in W_{\tau_M}} \sum_{u\in \til{U}_w} \eps(-1)^{\ell(w)} q^{-\til{\ell}(w)}[u^\mo w^\mo,T_{{w}}({w})(v)].
    \end{align*}
    If $\til{\ell}(w) < \til{\ell}(w_{0,\tau_M})$, we have  
    \begin{align*}
        \psi(\mu(\pi_K))(q^{-\til{\ell}(w)}[u^\mo w^\mo,T_{{w}}({w})(v)])(w_{0,\tau_M}) \in \varpi q^{-\til{\ell}(w_{0,\tau_M})}\lam^\circ.
    \end{align*}
    Note that $\til{\ell}(w)=\til{\ell}(w_{0,\tau_M})$ if and only if $w = w' w_{0,\tau_M}$ for $w'\in \prod_{\ov{I}\in\ov{\fP}_\tau, n_{\ov{I}}=1} W(\GL_{\#\ov{I}})$. For such $w'$, $u\in \til{U}_{w' w_{0,\tau_M}}$, and $\mu'\in B_{\mu}$, we claim that 
    \begin{align*}
        \psi_{\mu'}(\mu(\pi_K))(q^{-\til{\ell}(w_{0,\tau_M})}[u^\mo (w'w_{0,\tau_M})^\mo,T_{{w'w_{0,\tau_M}}}({w' w_{0,\tau_M}})(v)])(w_{0,\tau_M}) =  
            q^{-\til{\ell}(w_{0,\tau_M})}T_{w_{0,\tau_M}}(w_{0,\tau_M})(v)
    \end{align*}
    if $w'=1$, $u=1$, and $\mu'=\mu$, and is zero otherwise.  In particular, 
    \begin{align*}
        \psi(\mu(\pi_K))(e_{\rmK}([1,v]))(w_{0,\tau_M}) \in c_\mu q^{-\til{\ell}(w_{0,\tau_M})}T_{w_{0,\tau_M}}(w_{0,\tau_M})(v)   + \varpi q^{-\til{\ell}(w_{0,\tau_M})}\lam^\circ.
    \end{align*}
    This explains that $\psi \not\equiv 0 \mod \varpi$ if $c_\mu \not\equiv 0 \mod \varpi$.
    
    To complete the proof, it remains to verify our claim. Using the formula in Remark \ref{rmk:P-to-K}, we have
    \begin{align*}
        \psi_{\mu'}(\mu(\pi_K))([u^\mo (w'w_{0,\tau_M})^\mo , T_{{w'w_{0,\tau_M}}}({w'w_{0,\tau_M}})(v)]) ( w_{0,\tau_M}) = \phi_{\mu'}(w_{0,\tau_M} \mu(\pi_K) u^\mo (w'w_{0,\tau_M})^\mo)(T_{{w'w_{0,\tau_M}}}({w'w_{0,\tau_M}})(v)). %\neq 0 \ \ \text{only if } \  g \mu(\pi_K) u^\mo w_{0,\tau_M}^\mo \in \cP w(\mu)(\pi_K) \cP \ \  \text{for some $w\in W_{\tau_M}$}.
    \end{align*}
    The support of $\phi_{\mu'}$ is $\cP' \mu'(\pi_K) \cP'$ where $\cP'$ is the parahoric generated by $\cP$ and $W_{\tau_M}$. This shows that the above equation is zero unless $\mu'=\mu$.
    
    Thus, we take $\mu'=\mu$. When $u=1$, we get the desired result. The above equation is non-zero only if there exist $p_1,p_2 \in \cP$ such that
    \begin{align*}
        w_{0,\tau_M} \mu(\pi_K) u^\mo (w'w_{0,\tau_M})^\mo = p_1 w(\mu)(\pi_K) p_2
    \end{align*}
    for some $w\in W_{\tau_M}$. Equivalently, we need
    \begin{align*}
        w' (w'w_{0,\tau_M} u (w'w_{0,\tau_M})^\mo)^\mo p_2^\mo = w_{0,\tau_M}(\mu)(\pi_K)^\mo p_1 w(\mu)(\pi_K).
    \end{align*}
    The left hand side is in $G(\cO_K)$. This implies $w = w_{0,\tau_M}$ because otherwise, there is a diagonal block of the right hand side with an entry not in $\cO_K$. Since $(w'w_{0,\tau_M} u (w'w_{0,\tau_M})^\mo)^\mo \mod \pi_K \subset (\ov{N}\cap M')(k)$, the left hand side modulo $\pi_K$ is in $P'(k)$ . Using this together with the condition $\RG{\al,\mu}\le 0$ for $\al \in \Phi^+_{M'} $, the right hand side modulo $\pi_K$ is in $P(k)$. This in turn implies that $w' (w'w_{0,\tau_M} u (w'w_{0,\tau_M})^\mo)^\mo \mod \pi_K$ is in $P(k)$. Again using $(w'w_{0,\tau_M} u (w'w_{0,\tau_M})^\mo)^\mo \mod \pi_K \subset (\ov{N}\cap M')(k)$, this implies that $w'=1$. Finally, we get $u=1$ because $u=1$ if and only if $u \mod \pi_K = 1$.
\end{proof}

\begin{proof}[Proof of Proposition \ref{prop:surj}]
    This is a direct consequence of Lemma \ref{lem:O-basis2}
\end{proof}

Our next goal is to understand the convolution product on $\cH(\rmG,\sig^\circ(\tau))$. By Lemma \ref{lem:O-basis2}, $\cH(\rmG,\sig^\circ(\tau))$ admits an $\cO$-basis given by $s_P(\phi_{\mu})$. 
%From Lemma \ref{lem:hecke-decomp} and \ref{lem:hecke-tame-type-rational-presentation}, $\cH'(\rmG,\sig^\circ(\tau))$ admits an $\cO$-basis given by $\sum_{\mu\in \ov{\mu}}T_{\mu}$ where $\ov{\mu}$ is a $W_{\tau}$-orbit in $X_*(Z(M))$. 
We would like to express $s_P(\phi_{\mu})$ in terms of $s_P(T_{-\veps_I})$'s for various $I\in \fP_{\tau}$ and their symmetric polynomials with appropriate normalization. This requires us to understand the convolution on $\cH(\rmG,\lam)$ by comparing it with the convolution on $\cH(M,\lam_M)$. This is given in \cite[Cor.~6.13]{BKsmooth}, from which we deduce the following:

\begin{lem}\label{lem:conv-formula}
    For $\mu_1,\mu_2 \in X_*(Z(M))$, we have
     \begin{align*}
        T_{\mu_1}*T_{\mu_2} = \left\{\begin{array}{rl}
           T_{\mu_1 + \mu_2}  & \text{if $\mu_1,\mu_2$ are both dominant or antidominant }  \\
            q^{\sum_{\al\in \Phi^+\bss \Phi^+_M} \RG{\alpha,\mu_1} - \max\{\RG{\alpha,\mu_1+\mu_2}, 0\} } T_{\mu_1 + \mu_2} & \text{if $\mu_1$ is dominant and $\mu_2$ is antidominant.}
        \end{array}\right.
    \end{align*}    
\end{lem}

%Note that the above Lemm shows that $T_{\mu_1} * T_{\mu_2}$ is supported on $[\Iw (\mu_1+\mu_2)(\pi)\Iw]$. This implies that there is no relation between $T_{e_i}$'s for $i=1,\dots,n$, and so $\cH'(G,\mu_\cO)[1/p]$ is isomorphic to $E[x_1,\dots,x_n]$ by sending $T_{e_i}$ to $x_i$. The following Proposition determines the image of $\cH(\sig^\circ(\tau))$ under this isomorphism. \comment{revise this paragraph later}

% For a subset $I \subset \bfn$, let $\veps_{I} := \sum_{i \in I} \veps_{i}$. 

As applications, we obtain the following useful formulas.  %For $I\in \fP_\tau$, we write $T_I := T_{\veps_I}$.

\begin{lem}\label{lem:tame-type-formulas}
    We have the following formulas:
    \begin{enumerate}
        \item For $I \in \fP_\tau$, $T_{\veps_I}^\mo = q^{-\#{I}(n-\#{I})}T_{-\veps_I}$.
        \item For $I\in\fP_\tau$, $t_P(t_{\pm\veps_{I}}) = q^{-\#{I}(n-\#{I})/2} T_{\pm\veps_{I}}$.
        \item Let $\mu =  \sum_{I\in \fP_\tau} c_I \veps_{I} $ for some integers $c_I \in \Z_{\ge 0}$. Let $<_\mu$ be a strict total order on $\fP_\tau$ such that if $I<_{\mu} J$, then $c_I \le c_J$. Then we have
    \begin{align*}
        \prod_{I\in \fP_\tau} T_{-\veps_{I}}^{c_I} =  q^{\sum_{I\in \fP_\tau} c_I\#{I}(\sum_{I <_{\mu}J} \#{J}) } T_{-\mu}.
    \end{align*}
    \item For any $\mu$ as above, normalize $T_{-\mu}$ by 
    \begin{align*}
        T_{-\mu}' := q^{\min_{w\in W}\RG{\eta,w(\mu)}}T_{-\mu}.
    \end{align*}
    Then, we have $\prod_{I\in \fP_{\tau}}(T'_{-\veps_I})^{ c_I} = T_{-\mu}'$.
    \end{enumerate}
\end{lem}

\begin{rmk}\label{rmk:convolution}
    Item (1) shows that $T_{\veps_I}$ is not invertible in $\cH(\rmG,\lam^\circ)$ unless $I=\bfn$ or $\emptyset$. This is different from $t_{\veps_I}$ which are invertible in $\cH(\rmM,\lam_M^\circ)$ and explains that $t_P$ does not respect the integral structures. Since the quantity $q^{-\#{I}(n-\#{I})/2}$ in item (2) only depends on the size of $I$, it shows that the action of $W_{\tau_M}$ on the image of $t_P$ induced by its action on $\cH(\rmM,\lam_M)$ is compatible with the $W_{\tau_M}$-action permuting $T_I$ by its action on indices $I\in \fP_\tau$. Note that when $\veps_I$ is dominant, $q^{-\#{I}(n-\#{I})/2} = \delta_P^{1/2}(\veps_I(\pi_K))$. The order in item (3) $<_\mu$ is not necessarily unique but the formula does not depend on its choice. 
%    Let $w_\mu \in W$ such that for $I,J \in \fP_\tau$, $I <_\mu J$ if and only if $w_\mu(i) < w_\mu(j)$ for any $i\in I$ and $j\in J$. This is just to rearrange elements in $\fP_{\tau}$. Then
    %\begin{align*}
    %     \prod_{I\in \fP_\tau} q^{c_I \#{I}(\#I-1)/2}T_{-\veps_{I}}^{c_I} =  q^{\sum_{I\in \fP_\tau}c_I \sum_{i\in I}\eta_{w_\mu(i)}} T_{-\mu}.
    %\end{align*}
    Later in \S\ref{subsec:global-func}, we match $T'_{\mu}$ with (unnormalized) Frobenius eigenvalues of certain Weil--Deligne representations and $T_\mu$ with the normalized ones. We also note that when $\mu=\veps_{I}$, $\min_{w\in W}\RG{\eta,w(\veps_I)} = \#{I}(\#{I}-1)/2$. %The exponents of $q$ in the definition of $T'_{\mu}$ can be best understood in terms of the $q$-adic valuation of Frobenius eigenvalues of a certain Weil--Deligne representation. When $\mu=\veps_{I}$, $\min_{w\in W}\RG{\eta,w(\veps_I)} = \#{I}(\#{I}-1)/2$. % On the Galois side, $\eta_{w_\mu(i)}$ can be thought of as the $q$-adic valuation of ``$i$-th Frobenius eigenvalue'' and the power of $q$ here is related to the normalization of coefficients of the characteristic polynomial of the Frobenius. 
    %We will see how the computation here matches the Galois side in \S\ref{subsec:global-func}.
\end{rmk}

\begin{proof}
    Item (1) and (2) follow from Lemma  \ref{lem:conv-formula} by expressing $\veps_{I}$ as a sum of a dominant cocharacter and an antidominant cocharacter. Item (3) follows similarly, but a little more involved. Let us write $I<_\tau J$ for $I,J\in \fP_\tau$ if and only if $i<j$ for any $i\in I$ and $j\in J$. By item (2), we have
    \begin{align*}
        t_P(t_{-\mu}) = t_P(\prod_{I\in \fP_\tau} t_{-\veps_{I}}^{c_I}) = \prod_{I\in \fP_\tau}q^{-c_I\#{I} (n-\#{I})/2}T_{-\veps_{I}}^{c_I}.
    \end{align*}
    On the other hand, we can write $-\mu= \mu_+ - \mu_-$ for dominant cocharacters
    \begin{align*}
        \mu_+ := -\sum_{I\in \fP_\tau} c_I \sum_{I \le_\tau J} \veps_{J}, \ \ 
        \mu_- := -\sum_{I\in \fP_\tau} c_I \sum_{I <_\tau J} \veps_{J}.
    \end{align*}
    We obtain another expression of $t_P(t_{-\mu})$ given by 
    \begin{itemize}
        \item $t_P(t_{-\mu})  = \del_P^{1/2}(-\mu(\pi_K)) T_{\mu_+} T_{\mu_-}^\mo  = \del_P^{1/2}(\mu(\pi_K))  q^{-\sum_{I\in\fP_\tau}c_I (\sum_{J\le_\tau I}\#{J})(\sum_{J>_\tau I}\#{J})} T_{\mu_+}T_{-\mu_-}$,
        \item $\del_P^{1/2}(\mu(\pi_K)) = q^{\sum_{I\in\fP_\tau} c_I \#{I} (\sum_{J>_\tau I} \#{J}- \sum_{J<_\tau I} \#{J} )/2}$, and  
        \item $T_{\mu_+}T_{-\mu_-} = q^{\sum_{I <_\tau J} \#{I} \#{J} (- \max\CB{c_J-c_I,0} + \sum_{I <_\tau L \le_\tau J}c_L )} T_{-\mu}$.
    \end{itemize}
    Note that in the last bullet point, the $c_I$-coefficient in the exponent of $q$ is equal to
    \begin{align*}
        \#{I}(\sum_{\substack{I<_\tau J \\ I<_\mu J} }\#{J} -\sum_{\substack{J<_\tau I  \\ J<_\mu I} }\#{J}) + \sum_{J_1<_\tau I \le_\tau J_2} \#{J_1}\#{J_2}
    \end{align*}
    (here we need $<_\mu$ to be \textit{strict} to prevent double counting). The claimed formula is obtained by comparing the $c_I$-coefficients in the exponents of $q$ in two expressions of $t_P(t_{-\mu})$. Finally, item (4) is immediate from item (3).
\end{proof}

To give a presentation of $\cH(\rmG,\sig^\circ(\tau))$, we introduce a few more notations.  %Let $\cO[x_{I}]_{I\in \fP_{\tau}}$ be the polynomial $\cO$-algebra generated by $x_I$ for all $I\in \fP_\tau$. 
Recall the normalization $T'_{-\veps_I}:= q^{\#{I}(\#{I}-1)/2} T_{-\veps_I}$. %Similarly, we define $x_I' := q^{\#{I}(\#{I}-1)/2} x_I$. 
Let $\ov{I}\in \ov{\fP}_\tau$ of size $\#\ov{I}=r$; i.e.~$\ov{I}=\{I_1,\dots, I_r\}$ with $I_1,\dots, I_r \in \fP_{\tau}$. For $1 \le d \le r$, we define $T'_{\ov{I},d} := \sym_{d}(T'_{-\veps_{I_1}},\dots,T'_{-\veps_{I_r}})$ the degree $d$ elementary symmetric polynomial in variables $T'_{-\veps_{I_1}},\dots,T'_{-\veps_{I_r}}$. For $\ov{\fI}\subset \ov{\fP}_\tau$ and a collection of integers $d_{\ov{\fI}}=(d_{\ov{I}})_{\ov{I}\in \ov{\fI}}$ such that $1 \le d_{\ov{I}}\le \#\ov{I}$, we define $T'_{\ov{\fI}, d_{\ov{\fI}}} := \prod_{\ov{I}\in \ov{\fI}} T'_{\ov{I},d_{\ov{I}}}$. Recall $n_{\ov{I}}= \# I$ for some (equiv.~any) $I\in \ov{I}$. For the pair $(\ov{\fI}, d_{\ov{\fI}})$, we define $n(\ov{\fI}, d_{\ov{\fI}}) := \sum_{\ov{I}\in \ov{\fI}} d_{\ov{I}}n_{\ov{I}}$. 
%\begin{align*}
%    n(\ov{\fI}, d_{\ov{\fI}}) := (\sum_{\ov{I}\in \ov{\fI}} d_{\ov{I}}(\#\ov{I})^\mo \sum_{I\in \ov{I}} \#I)(\sum_{\ov{I}\in \ov{\fI}} d_{\ov{I}}(\#\ov{I})^\mo \sum_{I\in \ov{I}} \#I -1)/2
%\end{align*}
Finally, we define $\til{T}_{\ov{\fI}, d_{\ov{\fI}}} := q^{-n(\ov{\fI}, d_{\ov{\fI}})(n(\ov{\fI}, d_{\ov{\fI}})-1)/2} T'_{\ov{\fI}, d_{\ov{\fI}}}$.

\begin{prop}\label{prop:int-hecke-tame-types}
    Following the above notations, %the Hecke algebra $\cH(G,\sig^\circ(\tau))$ is generated by $T_{w_\tau(\fI)}$ for $\fI\subset \fP_\tau$ and $T_{w_\tau(\fP_\tau)}^\mo$. Moreover, 
    $\cH(\rmG,\sig^\circ(\tau))$ is generated as an $\cO$-subalgebra of $\cH(\rmG,\sig(\tau))$ by $\til{T}_{\ov{\fI}, d_{\ov{\fI}}}$ for all choices of $\ov{\fI}, d_{\ov{\fI}}$ and $\til{T}_{\ov{\fP}_{\tau}, d_{\ov{\fP}_{\tau}}}^\mo$ with $d_{\ov{\fP}_{\tau}} = (\#{\ov{I}})_{\ov{I}\in \ov{\fP}_\tau}$.
%    \begin{align*}
%        \cH(\rmG,\sig^\circ(\tau)) &\risom \cO[q^{-n(\fI)} x_{\fI}, q^{n(\fP_\tau)}x_{\fP_\tau}^\mo]_{\fI\subset \fP_\tau}^{\fS_{\ov{\fP}_\tau}} \\
%        T_{\fI} &\mapsto q^{-n(\fI)} x_{\fI}.
%    \end{align*}
\end{prop}

\begin{proof}
     Let $\mu \in X_*(Z(M))_-$ be a miniscule cocharacter. We can write $\mu = \sum_{I}-\veps_I$ where the sum runs over some subset of $\fP_{\tau}$. %we can rewrite $\mu$ as a sum of $-\veps_{I}$ for $I\in \fP_\tau$ (for convenience, one may assume that $\mu$ is a sum of $-\veps_{I}$ with non-negative coefficients). 
     From this and using Lemma \ref{lem:tame-type-formulas}, we can find a pair $(\ov{\fI},d_{\ov{\fI}})$ as above such that $\phi_{\mu} = \til{T}_{\ov{\fI}, d_{\ov{\fI}}}$. More precisely, $\ov{\fI}$ is the set of $\ov{I}$ whose elements appear in the expression of $\mu$ and $d_{\ov{I}}$ is the number of elements of $\ov{I}$ appearing in the expression of $\mu$. Moreover, $\mu = -\veps_{I'}$ for some $I'\subset \bfn$ of size $n(\ov{\fI}, d_{\ov{\fI}})$. Then $\phi_{\mu} = \til{T}_{\ov{\fI}, d_{\ov{\fI}}}$ follows from Lemma \ref{lem:tame-type-formulas}(4). The non-miniscule case follows from the miniscule case using Lemma \ref{lem:tame-type-formulas} again and the fact that $\til{T}_{\ov{\fP}_{\tau}, d_{\ov{\fP}_{\tau}}} = T_{-(1,1,\dots, 1)}$. %(4) and the above discussion. Note that $\til{T}_{\ov{\fP}_{\tau}, d_{\ov{\fP}_{\tau}}} = T_{-(1,1,\dots, 1)}$.
\end{proof}

\begin{rmk}\label{rmk:regular-case}
    Suppose that $\tau$ is regular. Then $W_{\tau_M}$ is trivial and thus $\#\ov{I}=1$ for all $\ov{I}\in \ov{\fI}$. In this case, we consider a subset $\fI\subset \fP_{\tau}$ instead of a pair $(\ov{\fI},d_{\ov{\fI}})$ as above. Given $\fI\subset \fP_{\tau}$, we can take $\ov{\fI}$ to be the image of $\fI$ in $\ov{\fP}_\tau$ and $d_{\ov{I}}$ is forced to be 1. Following this notation, we write $\til{T}_{\fI}$ for $\til{T}_{\ov{\fI}, (1)_{\ov{I}\in \ov{\fI}}}$. If we write $\veps_{\fI} := \sum_{I\in \fI} \veps_I$, then $\til{T}_{\fI}$ is just $T_{\veps_{\fI}}$. We also write $n_{\fI}$ for $n_{\ov{\fI}, (1)_{\ov{I}\in \ov{\fI}}}$ which in this case is equal to $\sum_{I\in \fI}\#I$.
\end{rmk}

%When $\tau$ is regular, then $\#\ov{I}=1$ for all $\ov{I}\in \ov{\fI}$ and $d_{\ov{I}}$ is forced to be 1. In this case, for any $\fI\subset \fP_\tau$, we write $f_{\ov{I},d_{\ov{I}}} = f_{\ov{I}}$, $F_{\fI} = F_{\ov{\fI},(1)_{\ov{I}\in \ov{\fI}}}$ and $\til{F}_{\fI} = \til{F}_{\ov{\fI},(1)_{\ov{I}\in \ov{\fI}}}$, similar to Remark \ref{rmk:hecke-regular-notation}.

%\begin{rmk}
%    Following the same argument, we can obtain a presentation of ``$l\neq p$'' integral Hecke algebras of tame types. Namely, we can define $\sig(\tau)$ and $\sig^\circ(\tau)$ in the same way except by replacing the coefficient rings $E$ and $\cO$ by a sufficiently large finite extension of $\Q_l$ and its ring of integers for $l\neq p$. In this case, $\cH(\rmG,\sig^\circ(\tau))$ is simply a symmetric Laurent polynomial ring $\cO[T_{\veps_I}^\pm]_{I\in \fP_\tau}^{\fS_{\ov{\fP}_\tau}}$.
%\end{rmk}

\begin{rmk}\label{rmk:indep-lattice}
    Recall that our choice of the lattice $\sig^\circ(\tau)\subset \sig(\tau)$ depends on choices of $\sig^\circ(\tau_I) \subset \sig(\tau_I)$. % and a certain ordering of the set $\ov{\fP}_\tau$. 
    Later, by Theorem \ref{thm:hecke-global-func}, it will be clear that $\cO$-subalgebra $\cH(\rmG,\sig^\circ(\tau))\subset \cH(\rmG,\sig(\tau))$ is independent of such choices. We expect that one can deduce this independence from a purely representation-theoretic method.
\end{rmk}

\begin{example}\label{ex:hecke-pres}
    We describe $\cH(\rmG,\sig^\circ(\tau))$ explicitly for $\tau$ as in Example \ref{ex:main}.
    \begin{itemize}
        \item [(PS)] For simplicity, we write $T_i=T_{-\veps_i}$. Note that $T_i'=T_i$. Then $\cH(\rmG,\sig^\circ(\tau))$ is generated by $q^{-\#I(\#I-1)/2}\prod_{i\in I} T_i$ for all subsets $I\subset \bfn$ and $q^{n(n-1)/2}T_1^\mo\cdots T_n^\mo$. 

        \item [(C)] In this case, we have $\fP_{\tau}= \{\bfn\}$ and $\ov{\fP}_{\tau} = \{\{\bfn\}\}$. The only possible choice of $(\ov{\fI}, d_{\ov{\fI}})$ is $(\{\{\bfn\}\},1)$ and $\cH(\rmG,\sig^\circ(\tau))\simeq \cO[T_{(1,1,\dots,1)}^\pm]$.

        \item [(Irr)] In this case, $\fP_{\tau} = \{I_1=\{1,2\},I_2=\{3,4\}\}$ and $\ov{\fP}_{\tau}=\{ \{I_1,I_2\}\}$. Thus, possible choice of $(\ov{\fI}, d_{\ov{\fI}})$ are $(\ov{\fP}_{\tau},1)$ and $(\ov{\fP}_{\tau},2)$ with corresponding Hecke operators
        \begin{align*}
            p^\mo(T'_{-\veps_{I_1}} + T'_{-\veps_{I_2}}) &= T_{-\veps_{I_1}} + T_{-\veps_{I_2}}
            \\
            p^{-6}T'_{-\veps_{I_1}}T'_{-\veps_{I_2}} &= T_{-(\veps_{I_1}+\veps_{I_2})}.
        \end{align*}
    \end{itemize}
\end{example}

\subsection{Locally algebraic types}\label{subsec:loc-alg-hecke}
We continue using the notation from the previous section. %In particular, $\sig(\tau)$ denotes the fixed tame principal series type and $\sig^\circ(\tau)$ denotes the chosen $\cO$-lattice. 

Let $\lam \in X^*(\uT)_+$ be a dominant weight. Let $W(\lam)_{/\cO}$ be the unique up to isomorphism irreducible algebraic 
$G_{/\cO}$-representation of highest weight $\lam$.  Let $V(\lam)$ be the the restriction of $W(\lam)_{/\cO}$ to $G(\cO_K)$. Note that since $V(\lam)$ is an algebraic representation, there is a natural $\rmG$-action on $V(\lam)[1/p]$. We define the locally algebraic type $\sig(\lam,\tau):= V(\lam)\otimes_{\cO}\sig(\tau)$ and an $\cO$-lattice in it $\sig^\circ(\lam,\tau):= V(\lam)\otimes_{\cO}\sig^\circ(\tau)$. 

\begin{lem}[{\cite[Lem.~1.4]{ST06}}]\label{lem:ST06}
    The map $\iota: \cH(\rmG, \sig(\tau)) \ra \cH(\rmG, \sig(\lam,\tau))$ sending $\phi: \rmG \ra \End_E(\sig(\tau))$ to $\iota(\phi) : g \in \rmG \mapsto g\otimes \phi(g) \in  \End_E(\sig(\lam,\tau))$ is an isomorphism.
\end{lem}

\begin{rmk}\label{rmk:loc-alg-hecke-mod}
    Let $\pi$ be a smooth admissible $E[\rmG]$-module. Under isomorphism $\iota$ above, the $\cH(\rmG, \sig(\tau))$-module  $\Hom_{\rmK}(\sig(\tau),\pi|_{\rmK})$ and the $\cH(\rmG, \sig(\lam,\tau))$-module $\Hom_{\rmK}(\sig(\lam,\tau), \pi\otimes_{\cO} V(\lam)|_{\rmK})$ are isomorphic.
\end{rmk}

We introduce some notations to describe $\cH(\rmG,\sig^\circ(\lam,\tau))$. Consider a pair $(\ov{\fI}, d_{\ov{\fI}})$ as in the previous subsection. We define $n_{\lam}(\ov{\fI}, d_{\ov{\fI}}) := \sum_{\tilj\in \tilcJ} \RG{\lam_{\tilj}, \omega_{n(\ov{\fI}, d_{\ov{\fI}})}^\vee}.$
%\begin{align*}
%    n_{\lam}(\ov{\fI}, d_{\ov{\fI}}) := \sum_{\tilj\in \tilcJ} \RG{w_0(\lam_{\tilj}), \omega_{n(\ov{\fI}, d_{\ov{\fI}})}^\vee}.
%\end{align*}

\begin{lem}
    For $(\ov{\fI}, d_{\ov{\fI}})$ as above and $m\in \Z$, $\pi_K^{m}\iota(\til{T}_{\ov{\fI}, d_{\ov{\fI}}}) \in \cH(\rmG,\sig^\circ(\lam,\tau))$ if and only if $m\ge n_{\lam}(\ov{\fI}, d_{\ov{\fI}})$. 
\end{lem}
\begin{proof}
    As explained in the proof of Proposition \ref{prop:int-hecke-tame-types}, $\til{T}_{\ov{\fI}, d_{\ov{\fI}}}$ is supported on the double coset $\rmK \mu(\pi_K) \rmK$ for some miniscule cocharacter $\mu\in X_*(Z(M))$. Thus, it suffices to show that $\iota(\til{T}_{\ov{\fI}, d_{\ov{\fI}}})(\mu(\pi_K))$ is in $\pi_K^{-n_{\lam}(\ov{\fI}, d_{\ov{\fI}})}\End_{E}(\sig^\circ(\lam,\tau))$. For $\lam'\le \lam$, $\mu(\pi_K)$ acts on the $\lam'$-weight space in $V(\lam)$ by $\prod_{\tilj \in \tilcJ} \lam'_{\tilj}(\sig_{\tilj}(\mu(\pi_K)))$. Let $\mu'$ be a dominant element in $W\mu$. Then the $\pi_K$-adic valuation of $\lam'_{\tilj}(\sig_{\tilj}(\mu(\pi_K)))$ is contained in $[\RG{w_0(\lam_{\tilj}), \mu'},\RG{\lam_{\tilj},\mu'}]$. Then the claim follows by noting that $\mu'$ is in fact equal to $-w_0(\omega_{n(\ov{\fI}, d_{\ov{\fI}})}^\vee)$. 
\end{proof}

%This shows that $\pi_K^{-\RG{w_0(\sum_{\tilj\in\tilcJ}\lam_{\tilj}), \mu'}}\iota(T_{\mu}) \in \cH(\sig^\circ(\lam,\tau))$. 

%Then, the previous paragraph explains that $\pi_K^{-n_{\lam}(\ov{\fI}, d_{\ov{\fI}})}\iota(\til{T}_{\ov{\fI}, d_{\ov{\fI}}}) \in \cH(\rmG,\sig^\circ(\lam,\tau))$, which we abusively denote by $\til{T}_{\ov{\fI}, d_{\ov{\fI}}}\in \cH(\rmG,\sig^\circ(\lam,\tau))$ again. Note that $\pi_K^{-m}\til{T}_{\ov{\fI}, d_{\ov{\fI}}}\notin \cH(\rmG,\sig^\circ(\lam,\tau))$ for any $m>0$.

By abuse of notation, we denote the normalized Hecke operator $\pi_K^{n_{\lam}(\ov{\fI}, d_{\ov{\fI}})}\iota(\til{T}_{\ov{\fI}, d_{\ov{\fI}}}) \in \cH(\rmG,\sig^\circ(\lam,\tau))$ by  $\til{T}_{\ov{\fI}, d_{\ov{\fI}}}\in \cH(\rmG,\sig^\circ(\lam,\tau))$ again. %As a generalization of Proposition \ref{prop:int-hecke-tame-types}, we make the following definition.

\begin{prop}\label{prop:int-hecke-loc-alg}
    Following the above notations, the Hecke algebra $\cH(G,\sig^\circ(\lam,\tau))$ is generated as an $\cO$-subalgebra of $\cH(\rmG,\sig(\lam,\tau))$ by $\til{T}_{\ov{\fI}, d_{\ov{\fI}}}$ for all choices of $\ov{\fI}, d_{\ov{\fI}}$ and $\til{T}_{\ov{\fP}_{\tau}, d_{\ov{\fP}_{\tau}}}^\mo$ with $d_{\ov{\fP}_{\tau}} = (\#{\ov{I}})_{\ov{I}\in \ov{\fP}_\tau}$. %Moreover, 
%    we have the following isomorphism
%    \begin{align*}
%        \cH(G,\sig^\circ(\lam,\tau)) &\risom \cO[\pi_K^{-n_\lam(\fI)}q^{-n(\fI)} x_{\fI}, \pi_K^{n_\lam(\fP_\tau)}q^{n(\fI)} x_{\fP_\tau}^\mo]_{\fI\subset \fP_\tau}^{\fS_{\ov{\fP}_\tau}} \\
%    \end{align*}
\end{prop}
\begin{proof}
    This follows from Proposition \ref{prop:int-hecke-tame-types} and the above discussion.
\end{proof}

\subsection{Mod $p$ reduction of integral Hecke algebras}\label{subsec:hecke-modp}
In this subsection, we investigate morphisms between Hecke algebras $\cH(\rmG,\sig^\circ(\tau)) \ra \cH(\rmG, \sig)$ for some Jordan--H\"older factors $\sig$ of $\osig(\tau)$. For simplicity, we assume that $\tau$ is regular throughout this subsection. Recall that we have an $P(k)$-representation $\lam^\circ = \otimes_{I\in \fP_\tau}\sig^\circ(\tau_I)$. Since $\tau$ is regular, $\sig^\circ(\tau)$ is equal to $\Ind_{P(k)}^{G(k)}\lam^\circ$. %We further assume that $\tau$ is regular, which implies that $\sig^\circ(\tau)=\Ind_{P(k)}^{G(k)}\lam^\circ_M$ and $\fS_{\ov{\fP}_\tau}$ is trivial.

\begin{lem}\label{lem:cuspidal-SW}
    For each $I\in \fP_\tau$, let $\sig_I$ be an irreducible constituent of $\osig(\tau_I)$. If $\sig$ is a Serre weight such that $\sig^{N(k)}$ is isomorphic to $\otimes_{I\in \fP_\tau} \sig_I$, then $\sig$ is an irreducible constituent of $\osig(\tau)$.
\end{lem}

\begin{rmk}\label{rmk:cuspidal-SW}
    Note that by Frobenius reciprocity, $\sig^{N(k)} \simeq \otimes_{I\in \fP_\tau} \sig_I$ if and only if there is a quotient map $\Ind_{P(k)}^{G(k)} \otimes_{I\in \fP_\tau} \sig_I \epi \sig$. %(called ``the quotient case''). Similarly,  $\sig_{N(k)} \simeq \otimes_{I\in \fP_\tau} \sig_I$ if and only if there is an injection $\sig \mono \Ind_{P(k)}^{G(k)} \otimes_{I\in \fP_\tau}\sig_I$ (called ``the subrepresentation case'').
\end{rmk}

\begin{proof}
    There is a subrepresentation $V\subset \otimes_{I\in \fP_\tau}\osig(\tau_I)$ with a surjection $V\epi \otimes_{I\in \fP_\tau}\sig_I$. Then $\Ind_{P(k)}^{G(k)}V$ is a subrepresentation of $\osig(\tau)$. At the same time, since $\otimes_{I\in \fP_\tau}\sig_I \mono \sig|_{P(k)}$, by Frobenius reciprocity, we obtain a surjection from  $\Ind_{P(k)}^{G(k)}V$ to $\sig$. This proves the claim. %The subrepresentation case can be proven by a similar argument.
\end{proof}

%\begin{rmk}\label{rmk:para-SW}
%    Fix $I\in \fP_\tau$ and write $m:= \sum_{J\le_\tau I} \#J$. Let $P'\subset \GL_n$ be the upper triangular parabolic subgroup with Levi factor $M'= \GL_m\times \GL_{n-m}$. We view $\GL_m,\GL_{n-m}$ as Levi blocks of $M'$. Let $\sig_{\le_\tau I}$ be an irreducible constituent of $\Ind_{P(k)\cap \GL_m(k)}^{\GL_m(k)}\otimes_{J\le_\tau I}\osig(\tau_J)$ and  $\sig_{>_\tau I}$ be an irreducible constituent of $\Ind_{P(k)\cap \GL_{n-m}(k)}^{\GL_{n-m}(k)}\otimes_{J>_\tau I}\osig(\tau_J)$. If $\sig$ is a Serre weight of $\GL_n(k)$ such that $\sig^{M'(k)}\simeq \sig_{\le_\tau I} \otimes \sig_{>_\tau I}$, then by following the proof of Lemma \ref{lem:cuspidal-SW}, we can show that $\sig$ is an irreducible constituent of $\osig(\tau)$.
%\end{rmk}

\begin{defn}\label{defn:cuspidal-SW}
    We denote by %$\JH_I(\ov{\sig^\circ}(\tau))$ and 
    $\JH_c(\osig(\tau)) $ the subset of $\sig\in\JH(\osig(\tau))$ %as in Remark \ref{rmk:para-SW} and 
    as in Lemma \ref{lem:cuspidal-SW}. %, respectively.% (The subscript $c$ stands for cuspidal, as $\otimes_{I\in \fP_\tau}\sig(\tau_I)$ is the ``cuspidal support'' of $\sig(\tau)$.)
\end{defn}
%Note that the definition of $\JH_c(\ov{\sig^\circ}(\tau))$ 

Now we state the main result of this subsection.

\begin{thm}\label{thm:hecke-modp-red}
Let $\sig$ be a $M$-regular weight in $\JH_c(\osig(\tau))$. There is a natural morphism
\begin{align*}
    \cR^{\sig^\circ(\tau)}_\sig : \cH(\rmG,\sig^\circ(\tau)) \ra \cH(\rmG, \sig).
\end{align*}
For each $\fI\subset \fP_\tau$, $\cR^{\sig^\circ(\tau)}_\sig({\til{T}}_{\fI})$ is zero unless $-\veps_{\fI}$ is antidominant, in which case $\cR^{\sig^\circ(\tau)}_\sig({\til{T}}_{\fI})$ is equal to $\ov{T}_{-\veps_{\fI}}$.
\end{thm}

\begin{rmk}
    A version of the above result with $\sig^\circ(\tau)$ replaced by an algebraic representation was proven by Herzig \cite[Prop.~2.10]{HerzigSatake}.
\end{rmk}

For the proof, we make use of the following lemma.

\begin{lem}\label{lem:coinv-comp}
    Let $\sig= F_M(\lam)$ be a Serre weight of $M(k)$ for $\lam \in X_f^*(T)$. We have the following isomorphism of $T(k)$-representations
    \begin{align*}
        (\Ind_{P(k)}^{G(k)}\sig)_{\ov{U}(k)} \simeq \oplus_{w_M \in W^M} \lam\ix{w_M}.
    \end{align*}
\end{lem}

\begin{proof}
    By the generalized Bruhat decomposition, we have
    \begin{align*}
       \Ind_{P(k)}^{G(k)}\sig|_{\ov{B}(k)} &= \oplus_{w_M\in W^M} \Ind_{P(k)}^{P(k)w_M \ov{U}(k)}\sig \\
        (\Ind_{P(k)}^{G(k)}\sig)_{\ov{U}(k)} &= \oplus_{w_M\in W^M} ( \Ind_{P(k)}^{P(k)w_M \ov{U}(k)}\sig)_{\ov{U}(k)}.
    \end{align*}
    We claim that $( \Ind_{P(k)}^{P(k)w_M \ov{U}(k)}\sig)_{\ov{U}(k)} \simeq \lam\ix{w_M}$ as $T(k)$-representations. To see this, first note that the subset
    \begin{align*}
        H_{w_M}(\sig) := \CB{f\in  \Ind_{P(k)}^{P(k)w_M \ov{U}(k)}\sig \mid \supp f \subset P(k)w_M}
    \end{align*}
    is mapped surjectively onto $ (\Ind_{P(k)}^{P(k)w_M \ov{U}(k)}\sig)_{\ov{U}(k)}$. We have a $T(k)$-action and a $P(k)$-action on $H_{w_M}(\sig)$ given by the right and left translations, respectively. The evaluation map $H_{w_M}(\sig) \xra{\ev_{w_M}} \sig$ is a $(P(k),T(k))$-biequivariant isomorphism where $T(k)$ acts on $\sig$ by $(\sig|_{T(k)})\ix{w_M}$. It follows from the property of $w_M$ that $w_M^\mo (P(k) \cap \ov{U}(k)) w_M \subset \ov{U}(k)$. Thus, the quotient map $H_{w_M}\epi (\Ind_{P(k)}^{P(k)w_M \ov{U}(k)}\sig)_{\ov{U}(k)}$ factors through $(H_{w_M})_{P(k) \cap \ov{U}(k)}$ where the coinvariant is taken with respect to the left translation action. Then $(H_{w_M})_{P(k) \cap \ov{U}(k)}$ is isomorphic to $(\sig_{P(k) \cap \ov{U}(k)})\ix{w_M} \simeq \lam\ix{w_M}$ as $T(k)$-representations. This shows the claimed isomorphism.
\end{proof}

\begin{proof}[Proof of Theorem \ref{thm:hecke-modp-red}]
    %We only prove the quotient case. The subrepresentation case can be proven using a similar argument. 
    
    By definition, there is a Jordan--H\"older constituent $\sig_I$ of $\osig(\tau_I)$ for each $I\in \fP_\tau$ such that $\sig^{N(k)}\simeq \otimes_{I}\sig_I$. We first show that there is a morphism
    \begin{align*}
        \cH(\rmG,\sig^\circ(\tau)) \ra \cH(\rmG, \Ind_{P(k)}^{G(k)}\otimes_I \sig_I)
    \end{align*}
    which maps $T_\mu\in \cH(\rmG,\sig^\circ(\tau))$ for $\mu\in X_*(Z(M))$ to $T_\mu\in \cH(\rmG, \Ind_{P(k)}^{G(k)}\otimes_I \sig_I)$. Recall that $\cH(\rmG,\sig^\circ(\tau)) = \cH(\rmG, \otimes_{I}\sig^\circ(\tau_I))$. Since $\otimes_I \sig_I$ is a subquotient of $\otimes_{I}\sig^\circ(\tau_I)/\varpi$, it suffices to show that for any $M(k)$-stable $\cO$-sublattice $ V\subset \otimes_{I}\sig^\circ(\tau_I)$ and any $\phi\in \cH(\rmG, \otimes_{I}\sig^\circ(\tau_I))$, $\phi$ maps $\ind_{\cP}^{\rmG} V$ into itself. This is equivalent to $\phi(g)(V) \subset V$ for any $g\in \rmG$. Since $\phi$ is supported on the union of double cosets $\cP \mu(\pi_K)\cP$ for $\mu\in X_*(Z(M))$, $\phi(g)$ is always $M(k)$-equivariant by Lemma \ref{lem:hecke-decomp-para}. Moreover, since $V[1/p]$ is an irreducible $M(k)$-representation, $\phi(g)$ is given by  a scalar. Therefore, it maps $V$ to itself. The claim that $T_\mu$ is mapped to $T_\mu$ follows immediately.
    
    %To prove this, we show that for any $M(k)$-subrepresentation $V\subset \otimes_I \osig(\tau_I)$, any $\rmG$-endomorphism $\phi$ of $\ind_{\cP}^{\rmG} \otimes_I \osig(\tau_I)$ preserves the $\rmG$-subrepresentation $\ind_\cP^\rmG V \subset \ind_{\cP}^{\rmG} \otimes_I \osig(\tau_I)$. In turn, it suffices to show that $\phi$ maps $V \subset \ind_\cP^\rmG V$ to $\ind_\cP^\rmG V$. In other words, if we view $\phi: \rmG \ra \End(\otimes_I \osig(\tau_I))$ as a $\cP$-biequivariant function, then $\phi(g)(V) \subset V$ for any $g\in \rmG$. Since $\phi$ is supported on the union of double cosets $\cP \mu(\pi_K)\cP$ for $\mu\in X_*(Z(M))$, by Lemma \ref{lem:hecke-decomp-para}, $\phi(g)$ is always $M(k)$-equivariant. Therefore, it maps $V$ to itself. The claim that $T_\mu$ is mapped to $T_\mu$ follows immediately.

    Now we construct a morphism 
    \begin{align*}
         \cH(\rmG, \Ind_{P(k)}^{G(k)}\otimes_I \sig_I) \ra \cH(\rmG, \sig).
    \end{align*}
    As mentioned in Remark \ref{rmk:cuspidal-SW}, $\sig$ can be realized as a quotient of $\Ind_{P(k)}^{G(k)}\otimes_I \sig_I$. Let $f: \Ind_{P(k)}^{G(k)}\otimes_I \sig_I \epi \sig$ be the quotient map. %Let $V$ be the kernel of this quotient map. Similar to the previous paragraph, 
    We show that for any $\phi \in \cH(\rmG, \Ind_{P(k)}^{G(k)}\otimes_I \sig_I)$, $f\circ \phi$ factors through $f$. % maps $\ind_\rmK^\rmG V$ to itself. 
    By Lemma \ref{lem:hecke-decomp}, $\phi(\mu(\pi_K))$ for $\mu\in X_*(T)_-$ factors through $(\Ind_{P(k)}^{G(k)}\otimes_I \sig_I)_{\ov{N}_\mu(k)}$ and lands into $(\Ind_{P(k)}^{G(k)}\otimes_I \sig_I)^{N_\mu(k)}$. %We need to show that the induced map from $V_{\ov{N}_\mu(k)}$ to $\sig^{N_\mu(k)}$ vanishes.
    Since $\sig^{N_\mu(k)}\simeq \sig_{\ov{N}_\mu(k)}$, it suffices to show that $\Hom_{M_{\mu}(k)}((\Ind_{P(k)}^{G(k)}\otimes_I \sig_I)_{\ov{N}_\mu(k)}, \sig_{\ov{N}_\mu(k)})$ is one dimensional over $\F$. This dimension is bounded by the dimension of $\Hom_{T(k)}((\Ind_{P(k)}^{G(k)}\otimes_I \sig_I)_{\ov{U}(k)}, \sig_{\ov{U}(k)})$. Then the claim follows from Lemma \ref{lem:coinv-comp} and the assumption that $\sig$ is $M$-regular.
   
    It remains to prove the description of $\cR^{\sig^\circ(\tau)}_\sig$. Let $\mu\in X_*(Z(M))$ be a cocharacter and $T_\mu\in \cH(G,\Ind_{P(k)}^{G(k)} \otimes_{I\in\fP_\tau}\sig_I)$. We want to show that $\cR^{\sig^\circ(\tau)}_\sig({T}_\mu)=0$ unless $\mu$ is antidominant, in which case it is supported on $\rmK \mu(\pi_K)\rmK$ and maps $\mu(\pi_K)$ to the composition
    \begin{align*}
        \sig \epi \sig_{\ov{N}_{\mu}(k)} \stackrel{\text{(Lem.~\ref{lem:inv-coinv-SW})}}{\simeq} \sig^{N_{\mu}(k)} \mono \sig.
    \end{align*}
    For $\nu \in X_*(T)_-$, we compute $\cR^{\sig^\circ(\tau)}_\sig({T}_\mu)(\nu(\pi_K)) \in \End_\F(\sig)$. Since $\Hom_{M_{\nu}(k)}(\sig_{\ov{N}_{\nu}(k)},\sig^{N_{\nu}(k)})$ is one-dimensional, this map is determined by its action on the highest weight vector $v\in \sig^{U(k)}$. By the isomorphism $\otimes_{I\in\fP_\tau}\sig_I \simeq \sig^{N(k)}$, we can view $v$ as the highest weight vector in $\otimes_{I\in\fP_\tau}\sig_I$. The image of $v$ under the natural injection $\otimes_{I\in\fP_\tau}\sig_I \mono \Ind_{P(k)}^{G(k)}\otimes_{I\in\fP_\tau}\sig_I$ is $[1,v]$ and is mapped to $v$ along the quotient map $\Ind_{P(k)}^{G(k)}\otimes_{I\in\fP_\tau}\sig_I  \epi \sig$. To compute $\cR^{\sig^\circ(\tau)}_\sig({T}_\mu)(\nu(\pi_K))(v)$, we only need to understand the part of $T_\mu(\nu(\pi_K))([1,v])$ on which $T(k)$ acts by the highest weight $\lam$. This is equivalent to understanding the image of $T_\mu(\nu(\pi_K))([1,v])$ under the evaluation map $\ev_{1}:\Ind_{P(k)}^{G(k)}\otimes_{I\in\fP_\tau}\sig_I \ra \otimes_{I\in\fP_\tau}\sig_I$. In short, we have the following commutative diagram
    \[
    \begin{tikzcd}
        \Ind_{P(k)}^{G(k)}\otimes_{I\in\fP_\tau}\sig_I \arrow[r, "T_{\mu}"] & \ind_{\rmK}^{\rmG} \Ind_{P(k)}^{G(k)}\otimes_{I\in\fP_\tau}\sig_I \arrow[r, "\ev_{\nu(\pi_K)}"] & \Ind_{P(k)}^{G(k)}\otimes_{I\in\fP_\tau}\sig_I \arrow[d, "\ev_{1}"] \\
        \otimes_{I\in\fP_\tau}\sig_I \arrow[r, "T_{\mu}"] \arrow[u, hook] & \ind_{\cP}^\rmG \otimes_{I\in\fP_\tau}\sig_I \arrow[r, "\ev_{\nu(\pi_K)}"] \arrow[u, equal] & \otimes_{I\in\fP_\tau}\sig_I.
    \end{tikzcd}
    \]
    Here, the left vertical arrow is the natural injection, and the middle vertical equality and the commutativity of all squares follow from Lemma \ref{lem:cind-transitive}. Since $T_\mu$ is supported on $\cP \mu(\pi_K)\cP$, ${T}_\mu(\nu(\pi_K))([1,v])(1)$ is zero unless $\mu=\nu$ (thus is zero unless $\mu$ antidominant), and if $\mu=\nu$, it is equal to $v$ by the definition of $T_\mu$. This proves our claim.
    %Let $\mu\in X_*(T)_-$ be an antidominant cocharacter. For $T_\mu \in \cH(G,\sig)$, $T_\mu(\mu(\pi))$ defines an isomorphism $\sig_{\ov{N}_\mu(k)} \risom \sig^{N_\mu(k)}$ and is determined by the image of the highest weight vector. 
    %The case that $\sig$ is a subrepresentation of $\Ind_{P(k)}^{G(k)}\otimes_{I\in\fP_\tau}\sig_I$ can be proven similarly. In this case, it is more convenient to analyze the action of $\ov{T}_\mu(\nu(\pi))$ on the lowest weight vector $w_0 v \in \sig^{\ov{U}(k)}$. This can be computed by modifying the above commutative diagram as follows: compose $w_0$ with the left vertical injection, replace the bottom $\ev_{\nu(\pi)}$ by $\ev_{w_0(\nu)(\pi)}$, and replace the right vertical arrow by $\ev_{w_0}$. Then left and right squares are not commutative but the largest rectangle is still commutative.
\end{proof}

\begin{rmk}
    Later, we compare $\cR^{\sig^\circ(\tau)}_\sig$ with an analogous map on the Galois side (Corollary \ref{cor:spec-hecke-modp}). Using the connection to the Galois side, we can obtain a morphism $\cH(\rmG, \sig^\circ(\tau)) \ra \cH(\rmG,\sig)$ for \textit{any} $\sig \in \JH(\osig(\tau))$. However, it is not clear how to obtain such a morphism in a purely representation-theoretic method.  
\end{rmk}

\begin{example}\label{ex:hecke-modp}
    We consider $\tau$ as in Example \ref{ex:main}(PS). In this case, $\sig^\circ(\tau)=\Ind_{B(k)}^{G(k)}[\mu]$ where $\mu:T(k) \ra \F^\times$ corresponds to $(\chi_1,\dots, \chi_n)$ under the local Artin map. %   i.e.~$\tau\simeq \oplus_{i=1}^{n}\chi_i$ where $\chi_i = \om_{K,\sig_0}^{\bfa_i}$ for $n$ distinct integers $\bfa_i \in [0,p^f-1)$. If we define a character $\mu=(\mu_1,\dots,\mu_n) : T(k) \ra \cO^\times$ by $\mu_i = \chi_i \circ \Art_K|_{\cO_K^\times}$, then $\sig^\circ(\tau)$ is equal to $\Ind_{B(k)}^{G(k)}\mu$. 
    As in Example \ref{ex:hecke-pres}, % By Proposition \ref{prop:int-hecke-tame-types}, 
    $\cH^\circ(\rmG,\sig^\circ(\tau))$ is isomorphic to the subalgebra of $\cO[T_1^\pm,\dots,T_n^\pm]$ generated by $q^{-\#I(\#I-1)/2}\prod_{i\in I}T_i$ for all subsets $I\subset \bfn$ and $q^{n(n-1)/2}T_1^\mo \cdots T_n^\mo$. In this case, $\JH_c(\osig(\tau))$ consists of a unique element $F(\mu)$. By Theorem \ref{thm:hecke-modp-red}, we have $\cR^{\sig^\circ(\tau)}_{F(\mu)}$ from $\cH(\rmG,\sig^\circ(\tau))$ to $\cH(\rmG,F(\mu))$ which maps $q^{-\#I(\#I-1)/2}\prod_{i\in I}T_i$ for $I=\{1,2,\dots,m\}$ to $\ov{T}_m$.

    On the nose, Theorem \ref{thm:hecke-modp-red} only applies to $\cH(\rmG,F(\mu))$ being the target. However, we can change the ordering of $\chi_i$ and thus change $\sig^\circ(\tau)$ to the one % an $\cO$-lattice $\sig^\circ(\tau)\subset \sig(\tau)$ 
    whose cosocle is isomorphic to $F(\mu')$ for any $\mu'\in \mu$ while keeping $\cH(\rmG,\sig^\circ(\tau))$ the same by Remark \ref{rmk:indep-lattice}. 
    Thus, Theorem \ref{thm:hecke-modp-red} shows that there is a surjective morphism $\cH(\rmG,\sig^\circ(\tau)) \ra \cH(\rmG,F(\mu'))$ for any $\mu' \in W\mu$.

    We make this explicit for later use. Let $\sig$ be a Serre weight with a lowest alcove presentation $(\tilw,\ud{\om})$. We write $\tilw = t_\nu w$, $\ud{\mu}=\pi^\mo(\tilw)\cdot (\ud{\om}-\eta)$, and $\mu= \sum_{j=1}^f \mu_j p^{f-j}$. Consider $\tau = \tau(e,\ud{\mu})$. Then it has a lowest alcove presentation $\tau \simeq \tau(\pi^\mo(w)^\mo w, \ud{\om}+ \pi^\mo(w)^\mo(\nu-\eta))$. By following the recipe in \S\ref{subsec:more-not}, we have $\sig(\tau) \simeq \Ind_{B(k)}^{G(k)}[w_{f-1}^\mo(\mu)]$. Then Theorem \ref{thm:hecke-modp-red} induces
    \begin{align*}
    \cR_{\sig}^{\sig^\circ(\tau)} : \cH(\rmG,\sig^\circ(\tau)) \ra \cH(\rmG,\sig)    
    \end{align*}
    that maps $\til{T}_{I}$ for $I= w_{f-1}^\mo(\{1,2,\dots, m\})$ to $\ov{T}_m$.
\end{example}

\section{Moduli stack of Breuil--Kisin modules}\label{sec4}
In \S\ref{subsec:BKmodules}, we discuss the moduli stack of Breuil--Kisin modules with tame descent data and prove its normality by following \cite{CL-ENS-2018-Kisinmodule-MR3764041} closely. In \S\ref{subsec:global-func}, we construct certain global functions on it, which are modelled by Frobenius eigenvalues of Weil--Deligne representations associated to Breuil--Kisin modules. In \S\ref{subsec:WD}, we relate the ring of global functions with integral Hecke algebras discussed in \S\ref{sec3}.

Throughout this section, we fix a tame inertial type $\tau$. Recall that we have $s_\tau \in W$ associated with $\tau$. We write $r:= \abs{s_\tau}$.

\subsection{Breuil--Kisin modules}\label{subsec:BKmodules}

We start by recalling the definition of Breuil--Kisin modules with tame descent data as in \cite[\S3.1]{LLLM-extreme} (which is based on \cite[\S5.1]{LLLMlocalmodel}). We refer the reader to \loccit~for omitted details. 

Let $K'$ be the subfield of $\ov{K}$ which is the unramified extension of degree $r$ over $K$. Let $K_0'$ be the maximal unramified subextension of $K'$ over $\Qp$ and $k'$ be the residue field of $K'$. We set $\cJ':= \CB{K_0'\mono E}$, $\tilcJ':=\CB{K' \mono E}$, $f':=fr$, and $e'= p^{f'}-1$. We fix an $e'$-root $\pi_{K'}$ of $\pi_K$ and define a field extension $L':= K'(\pi_{K'})/K$. We write $\Del':= \Gal(L'/K') \subset \Del := \Gal(L'/K)$. We denote by $\sig^f\in \Gal(L'/K)$ the lifting of the $q$-power Frobenius on $W(k')$ which fixes $\pi_{K'}$. 

For a $p$-adically complete $\cO$-algebra $R$, we define $\fS_{L',R}:= (W(k')\otimes_{\Zp} R)\DB{u'}$. It has a usual Frobenius $\varphi : \fS_{L',R} \ra \fS_{L',R}$ acting on $W(k')$ as Frobenius, on $R$ trivally, and sends $u'$ to $(u')^p$. It also has an action of $\Del$ such that $\fS_{L',R}^{\Del} = (W(k)\otimes_{\Zp} R)\DB{v}$ where $v:= (u')^{ p^{f'}-1}$.

Let $E(v)$ be the minimal polynomial for $\pi_K$ over $K_0$ of degree $e$. For each $\sig_j \in \cJ$, we write $E_j(v) := \sig_{j}(E(v)) \in \cO[v]$. Note that $E(v) = E((u')^{p^{f'}-1})$ is the minimal polynomial for $\pi_{K'}$ over $K_0$. %We write $E_j(v):= \sig_j(E(v)) \in \cO[v]$.

For any module $\fM$ over $\fS_{L',R}$, we have a $R\DB{u'}$-linear decomposition $\fM \simeq \oplus_{j'\in \cJ'} \fM\ix{j'}$ induced by the maps $W(k') \otimes_{\Zp} R \ra R$ defined by $x\otimes r \mapsto \sig_{j'} (x) r$ for $j'\in \cJ'$. 

\begin{defn}[{\cite[Def.~5.1.3]{LLLMlocalmodel}}]\label{defn:BKmod}
    Let $h$ be a positive integer. We let $Y^{[0,h],\tau}(R)$ be the groupoid of \textit{Breuil--Kisin modules of rank $n$ over $\fS_{L,R}$ and height in $[0,h]$ equipped with a tame descent data $\tau$}. An object in $Y^{[0,h]}(R)$ is a pair $(\fM,\phi_\fM)$ together with an $\Del$-action on $\fM$ where
    \begin{enumerate}
        \item $\fM$ is a finitely generated projective $\fS_{L',R}$-module locally free of rank $n$;
        \item $\phi_\fM: \varphi^*(\fM) \ra \fM$ is an injective $\fS_{L',R}$-linear map whose cokernel is annihilated by $E(v)^h$; and
        \item the $\Del$-action is semilinear and commutes with $\phi_\fM$, and for each $j'\in\cJ'$, 
        \begin{align*}
            \fM\ix{j'}\mod u' \simeq \tau^\vee \otimes R
        \end{align*}
        as $\Del'$-representations.
    \end{enumerate}
    In particular, the semilinear $\Del$-action induces an isomorphism $(\sig^f)^*\fM \simeq \fM$ as objects of $Y^{[0,h],\tau}(R)$ (see \cite[\S6.1]{LLLM1}).
\end{defn}

It is known that $Y^{[0,h],\tau}$ is a $p$-adic formal algebraic stack of finite type over $\Spf \cO$ \cite[Prop.~3.1.3 and 3.3.5]{cegsB}.

%\begin{rmk}
%    Note that in item (3) of the above definition, we use $\tau$ instead of $\tau^\vee$ used in \cite[Def.~3.1.1]{LLLM-extreme}. Our $Y^{[0,h],\tau}$ is equal to $Y^{[0,h],\tau^\vee}$ in \loccit.
%\end{rmk}

%\begin{rmk}
%    \begin{enumerate}
%        \item The Frobenius $\phi_\fM$ induces morphisms $\phi_\fM\ix{j} : \varphi^*(\fM\ix{j-1}) \ra \fM\ix{j}$.
%        \item It is known that $Y^{[0,h]}$ is a $p$-adic formal algebraic stack over $\Spf \cO$ \cite[Prop.~3.1.3 and %3.3.5]{cegsB}.
%    \end{enumerate}
%\end{rmk}

\begin{defn}
    For each $j'\in \cJ'$, we denote by $\phi_\fM\ix{j'}: \varphi^*(\fM\ix{j'-1}) \ra \fM\ix{j'}$ the morphism induced by  $\phi_\fM$. 
\end{defn}

\begin{defn}[{\cite[Def.~5.1.6]{LLLMlocalmodel}}]\label{defn:eigenbasis}
    An \textit{eigenbasis} of $\fM$ is a collection of ordered bases $\be\ix{j'} = (f_1\ix{j'},f_2\ix{j'},\dots, f_n\ix{j'})$ for each $\fM\ix{j'}$ for $j'\in\cJ'$ such that for each $i=1,\dots,n$,
    \begin{itemize}
        \item $\Del'$ acts on $f_i\ix{j'}$ by $\chi_i^\mo$; and
        \item $\sig^f$ maps $f_i\ix{j'}$ to $f_{s_\tau(i)}\ix{j'+f}$.
    \end{itemize}
\end{defn}

Let $\chi:\Del' \ra \cO^\times$ be a character. For any $R\DB{u'}$-module $M$ with semilinear $\Del'$-action, we denote by $M_{\chi}\subset M$ the $R\DB{v}$-submodule given by the $\chi$-eigenspace. For $\fM \in Y^{[0,h],\tau}(R)$ and $j'\in \cJ'$, the $R\DB{u'}$-linear partial Frobenii $\phi_{\fM}\ix{j'} : \varphi^*(\fM\ix{j'-1}) \ra \fM\ix{j'}$ restrict to $R\DB{v}$-linear morphism $\phi_{\fM,\chi}\ix{j'} : \varphi^*(\fM\ix{j'-1})_{\chi} \ra \fM\ix{j'}_{\chi}$. Given eigenbasis $\be= \CB{f_1\ix{j'},\dots, f_n\ix{j'}}_{j'\in\cJ'}$ for $\fM$, we define
\begin{align*}
            \ga\ix{j'} := \CB{ u^{\bfa\ix{j'}_{s_{\rmor,j'}^\mo(i)} - \bfa\ix{j'}_{s_{\rmor,j'}^\mo(n)}}f_{s_{\rmor,j'}^\mo(i)}\ix{j'}}_{i=1}^n.
        \end{align*}
which is a basis for  $\fM\ix{j'}_{ \chi^\mo_{s_{\rmor,j'}^\mo(n)}}$, and we also define
\begin{align*}
    \prescript{\varphi}{}{\ga}\ix{j'-1} := \CB{ u^{\bfa\ix{j'}_{s_{\rmor,j'}^\mo(i)} - \bfa\ix{j'}_{s_{\rmor,j'}^\mo(n)}} \otimes f_{s_{\rmor,j'}^\mo(i)}\ix{j'-1}}_{i=1}^n
\end{align*}
which is a basis for $\varphi^{*}(\fM\ix{j'-1})_{ \chi^\mo_{s_{\rmor,j'}^\mo(n)}}$. 

\begin{defn}\label{defn:partial-frob}
    For $\fM$, $\be$, and $j'\in \cJ'$ as above, we define $A_{\fM,\be}\ix{j'} \in \Mat_{n}(R\DB{v})$ the \textit{matrix of the $j'$-th partial Frobenius of $\fM$ with respect to $\be$} is the matrix of $\phi_{\fM, \chi_{s_{\rmor,j'}(n)}^\mo}\ix{j'}: \varphi^{*}(\fM\ix{j'-1})_{ \chi_{s_{\rmor,j'}^\mo(n)}} \ra \fM\ix{j'}_{ \chi_{s_{\rmor,j'}^\mo(n)}}$ with respect to the bases 
$\prescript{\varphi}{}{\ga}\ix{j'-1}$ and $\ga\ix{j'}$.
\end{defn}

\begin{rmk}
    By item (2) in Definition \ref{defn:eigenbasis}, the matrix $A\ix{j'}_{\fM,\be}$ only depends on $j'\mod f$.
\end{rmk}

For each $j\in \cJ$, let $P_{j}\subset \GL_n$ be the upper trangular parabolic subgroup with Levi factor $M_{j}$ defined by the condition that $\al$-entry vanishes for all roots $\al$ such that $\RG{s_{\rmor,j'}^\mo(\bfa\ix{j'}),\al^\vee}<0$ (this condition only depends on $j'\mod f$). We let $\cP_{j}$ be the pro-algebraic group scheme over $\cO$ defined by
\begin{align*}
    \cP_{j}(R) =\CB{g\in \varprojlim_{r}\GL_n(R[v]/E_j(v)^r) \mid g\mod v \in P_{j}(R)}
\end{align*}
for $\cO$-algebra $R$. Then $\cP_j$ is formally smooth over $\cO$ (see \cite[Prop.~2.21]{CL-ENS-2018-Kisinmodule-MR3764041} and the preceding paragraph).

Let $Y^{[0,h],\tau,\be}$ be a category fibered in groupoids over $\Spf \cO$ whose groupoid of $R$-points, for a $p$-adically complete $\cO$-algebra $R$, is the groupoid of pairs $(\fM,\be)$ where $\fM \in Y^{[0,h],\tau}(R)$ and $\be$ is an eigenbasis for $\fM$. %Then $Y^{[0,h],\tau,\be} \ra Y^{[0,h],\tau}$ is a torsor for $\prod_{\jj}\cP_j$ \cite[Prop.~4.6]{CL-ENS-2018-Kisinmodule-MR3764041}. 

\begin{prop}\label{prop:eigenbasis-torsor}
    The map $Y^{[0,h],\tau,\be} \ra Y^{[0,h],\tau}$ is a torsor for $\prod_{\jj} \cP_j$.
\end{prop}
\begin{proof}
    This is \cite[Prop.~4.6]{CL-ENS-2018-Kisinmodule-MR3764041}, except that they only consider a principal series $\tau$. The same argument works in our setting.
\end{proof}

\begin{rmk}
    Let $C_{\fM,\be}\ix{j'} \in \Mat_{n}(R\DB{u})$ denote the matrix of $\phi_{\fM}\ix{j'}: \varphi^{*}(\fM\ix{j'-1}) \ra \fM\ix{j'}$ with respect to the bases $\varphi^*(\be\ix{j'-1})$ and $\be\ix{j'}$. Then we have (\cite[(5.4)]{LLLMlocalmodel}) 
\begin{align}\label{eqn:mat-A-C-change}
    A_{\fM,\be}\ix{j'} = \Ad( s_{\rmor,j'}^\mo u^{-\bfa\ix{j'}}) (C_{\fM,\be}\ix{j'} ).
\end{align}
In particular, this shows that $A_{\fM,\be}\ix{j'} \mod v$ is contained in $\Lie P_{j}(R)$ for $j\equiv j' \mod f$.
\end{rmk}

Let $\be_1$ and $\be_2$ be eigenbases of $\fM \in Y^{[0,h],\tau}(R)$. For each $\be_i$, we have defined $\ga_i$  before Definition \ref{defn:partial-frob}. For $j'\in\cJ'$, let $D\ix{j'} \in \GL_n(R\DB{u})$ be the matrix such that $ \be_1\ix{j'} = \be_2\ix{j'} D\ix{j'}$. We define $P\ix{j'} := \Ad( s_{\rmor,j'}^\mo u^{-\bfa\ix{j'}})(D\ix{j'})$. Since $s_{\rmor,j'}^\mo (\bfa\ix{j'})$ is dominant, $P\ix{j'}$ is valued in $\cP_{j'}(R)$ and $\ga_1\ix{j'}= \ga_2\ix{j'} P\ix{j'}$. Conversely, if $\ga_1\ix{j'}= \ga_2\ix{j'} P\ix{j'}$ for some $P\ix{j'}\in\cP_{j'}(R)$, then $D\ix{j'}:= \Ad(  u^{\bfa\ix{j'}}s_{\rmor,j'})(P\ix{j'})$ is in $\GL_n(R\DB{u})$ and $ \be_1\ix{j'} = \be_2\ix{j'} D\ix{j'}$.

\begin{prop}\label{prop:change-of-basis-formula}
    Let $\fM,\be_1,$ and $\be_2$ be as above. For $j'\in\cJ'$, we have the change of basis formula
    \begin{align*}
        A_{\fM,\be_2}\ix{j'} = P\ix{j'}  A_{\fM,\be_1}\ix{j'} (\Ad(s_{j'}^\mo v^{\mu_{j'} +\eta_{j'}} ) (\varphi(P\ix{j'-1})^\mo))
    \end{align*}
    where $(\Ad(s_{j'}^\mo v^{\mu_{j'} +\eta_{j'}} ) (\varphi(P\ix{j'-1})^\mo)$ is also valued in $\cP_{j'}(R)$.
\end{prop}
\begin{proof}
    The formula follows from a direct computation unravelling the definition of $A_{\fM,\be_i}$. When $\tau$ is regular, this is explained in \cite[Prop.~3.2.9]{LLLelim}, and it generalizes to non-regular $\tau$ straightforwardly. The fact that $(\Ad(s_{j'}^\mo v^{\mu_{j'} +\eta_{j'}} ) (\varphi(P\ix{j'-1})^\mo)$ is valued in $\cP_{j'}(R)$ follows from
    \begin{align*}
        \Ad(s_{j'}^\mo v^{\mu_{j'} +\eta_{j'}} ) (\varphi(P\ix{j'-1})^\mo) = \Ad( s_{\rmor,j'}^\mo u^{-\bfa\ix{j'}})(\varphi(D\ix{j'-1})^\mo)
    \end{align*}
    where $D\ix{j'-1}=\Ad(u^{\bfa\ix{j'-1}}s_{\rmor,j'-1})(P\ix{j'-1})$. %Then $(\Ad(s_j^\mo v^{\mu_j +\eta_j} ) (\varphi(P\ix{j-1})^\mo) = \Ad( s_{\rmor,j}^\mo u^{-\bfa\ix{j}})(\varphi(D\ix{j-1})^\mo)$ from which we can see that it is valued in $\cP_j(R)$.
\end{proof}

For parabolic subgroups $P_j\subset \GL_n$ and $E_j(v)\in  \cO[v]$ with $\jj$, we recall the functor $\Fl^{E(v)}_{P_\cJ}:= \prod_{\jj}\Fl_{P_j}^{E_j(v)}$ over $\Spf \cO$ defined in \cite[\S2]{CL-ENS-2018-Kisinmodule-MR3764041}. For $p$-adically complete $\cO$-algebra $R$, $\Fl_{P_j}^{E_j(v)}(R)$ is defined to be the groupoid of triples $(L\ix{j},\al\ix{j},\veps\ix{j})$ where
\begin{enumerate}
    \item $L\ix{j}$ is a finitely generated projective module over $R\DB{v}$ locally of rank $n$;
    \item $\al\ix{j}: L\ix{j}[1/E_j(v)] \simeq R\DB{v}[1/E_j(v)]^n$ is a trivialization;
    \item $\veps\ix{j} = \CB{\Fil^i L\ix{j}/vL\ix{j}}$ is an increasing filtration on $L\ix{j}/vL\ix{j}$ such that $\Fil^0 L\ix{j}/vL\ix{j} = 0$ and $\Fil^n L\ix{j}/vL\ix{j} = L\ix{j}/vL\ix{j}$ and Zariski locally on $R$ whose stabilizer is $P_j(R) \subset \GL_n(R)$
\end{enumerate}
Note that $\Fl^{E(v)}_{P_\cJ}$ can be viewed as a limit $\varinjlim_{r} \Fl^{E(v)}_{P_\cJ}\times_{\cO}\Spec \cO/\varpi^r$ and $\Fl^{E(v)}_{P_\cJ}\times_{\cO}\Spec \cO/\varpi^r$ is known to be represented by an ind-scheme of ind-finite type over $\cO/\varpi^r$ (\cite[Prop.~4.1.4]{Levin-localmodel}).

For an integer $h\ge 0$, we define a subfunctor $\Fl^{E(v),[0,h]}_{P_\cJ} \subset \Fl^{E(v)}_{P_\cJ}$ by imposing the condition $E_j(v)^{-h} R\DB{v}^n \subset \be\ix{j}(L\ix{j}) \subset R\DB{v}^n$ for each $\jj$. Then for each $r\ge 1$, $\Fl^{E(v),[0,h]}_{P_\cJ} \times_{\cO} \Spec \cO/\varpi^r$ is represented by a projective scheme over $\cO/\varpi^r$ \cite[Prop.~2.9]{CL-ENS-2018-Kisinmodule-MR3764041} and $\Fl^{E(v),[0,h]}_{P_\cJ}$ is a $p$-adic formal scheme over $\Spf \cO$.

We have a morphism $Y^{[0,h],\tau,\be} \ra \Fl^{E(v),[0,h]}_{P_\cJ}$ which maps $(\fM,\be)$ to $(L\ix{j},\al\ix{j},\veps\ix{j})$ given by
\begin{itemize}
    \item $L\ix{j}  = \varphi^*( \fM\ix{j'-1})_{\chi^\mo_{s_{\rmor,j'}^\mo(n)}}$;
    \item $\al\ix{j}: L\ix{j} \simeq R\DB{v}^n$ given by the basis $\prescript{\varphi}{}{\ga}\ix{j'-1}$;
    \item  $\veps\ix{j} = \CB{\Fil^i L\ix{j}/vL\ix{j}}$ where for $i=1,\dots,n$,
    \begin{align*}
        \Fil^i  L\ix{j}/vL\ix{j} := u^{\bfa\ix{j'}_{s_{\rmor,j'}^\mo(i)}- \bfa\ix{j'}_{s_{\rmor,j'}^\mo(n)}} \varphi^*( \fM\ix{j'-1})_{\chi^\mo_{s_{\rmor,j'}^\mo(i)}} / v \varphi^*( \fM\ix{j'-1})_{\chi^\mo_{s_{\rmor,j'}^\mo(n)}}
    \end{align*}
\end{itemize}
for any  $j' \equiv j \mod f$. Note that this is well-defined by the properties $(\sig^f)^*\fM \simeq \fM$ and $\sig^f (\be\ix{j'}) = \be\ix{j'+f}$.

\begin{thm}\label{thm:BK-localmodel}
    We have the following local model diagram
    \[
    \begin{tikzcd}
        & Y^{[0,h],\tau,\be} \arrow[ld] \arrow[rd] \\
        Y^{[0,h],\tau} & & \Fl^{E(v),[0,h]}_{P_\cJ}
    \end{tikzcd}
    \]
    where both diagonal arrows are formally smooth.
\end{thm}
\begin{proof}
    This follows from Proposition \ref{prop:eigenbasis-torsor} and \cite[Thm.~4.9]{CL-ENS-2018-Kisinmodule-MR3764041}. Note that \loccit~only considers a principal series $\tau$ but its proof carries over verbatim to our setting.
\end{proof}

Let $\lam \in X_*(T)^{\tilcJ}$ be a dominant cocharacter such that $\lam_{\tilj}+\eta_{\tilj} \in [0,h]^n$. For each $\jj$ which we view as an equivalence class in $\tilcJ$, we denote by $\lam_j + \eta_j = (\lam_{\tilj}+\eta_{\tilj})_{\tilj \in j} \in X_*(T)^{j}$. Let $M_\cJ(\le\lam+\eta) := \prod_{\jj}M_j(\le\lam_j + \eta_j)$ where $M_j(\le\lam_j+\eta_j)$ is the $p$-adic completion of the Pappas--Zhu local model for $\Res_{\cO_K\otimes_{W(k),\sig_j}\cO/\cO}\GL_n$ defined in \cite{Levin-localmodel} as the Zariski closure of the open Schubert cell for the cocharacter $\lam_j+\eta_j$. It is a $\cO$-flat closed formal subscheme inside $\Fl^{E(v),[0,h]}_{P_\cJ}$. We define $Y^{\le\lam+\eta,\tau,\be}$ to be the pullback $Y^{[0,h],\tau,\be} \times_{\Fl^{E(v),[0,h]}_{P_\cJ}} M_\cJ(\le\lam+\eta)$. By \cite[Prop.~5.2]{CL-ENS-2018-Kisinmodule-MR3764041}, it descends to a closed substack $Y^{\le\lam+\eta,\tau} \subset Y^{[0,h],\tau}$. It is uniquely characterized by the following properties (\cite[Thm.~5.13]{CL-ENS-2018-Kisinmodule-MR3764041}; also see the discussion before Thm.~5.3.3 in \cite{LLLMlocalmodel}):
\begin{enumerate}
    \item $Y^{\le\lam+\eta,\tau}$ is $\cO$-flat closed substack of $Y^{[0,h],\tau}$; and
    \item for any finite extension $E'/E$ with ring of integers $\cO'$, $Y^{\le\lam+\eta,\tau} (\cO') \subset Y^{[0,h],\tau}(\cO')$ is exactly the subgroupoid of objects $\fM$ such that for any eigenbasis $\be$, $A_{\fM,\be}\ix{j'}$ viewed as an element in $\Mat_n(E'\DB{v})$ has elementary divisors bounded by $\prod_{\tilj \in j}(v-\sig_{\tilj}(\pi_K))^{(\lam+\eta)_{\tilj}}$.
\end{enumerate}
As an immediate consequence of Theorem \ref{thm:BK-localmodel}, we obtain the following result.
\begin{thm}[cf.~{\cite[Thm.~5.3]{CL-ENS-2018-Kisinmodule-MR3764041}}]\label{thm:BK-hodge-localmodel}
    We have the following local model diagram
    \[
    \begin{tikzcd}
        & Y^{\le\lam+\eta,\tau,\be} \arrow[ld] \arrow[rd] \\
        Y^{\le\lam+\eta,\tau} & & M_\cJ(\le\lam+\eta)
    \end{tikzcd}
    \]
    where both diagonal arrows are formally smooth.
\end{thm}

The following is the main result of this subsection. Note that since $Y^{\le\lam+\eta,\tau}$ is finite type over $\Spf\cO$, it is residually Jacobson in the sense of \cite[Def.~8.8]{Emerton-formal-alg-stk}. For a formal algebraic stack $X$ that is residually Jacobson, we say that $X$ is \textit{analytically normal} if it can be smoothly covered by formal schemes of the form $\Spf A$ such that the completion of $A$ at each maximal ideal of $A$ is a normal local domain (see Def.~8.13.1, 8.22 and Rmk.~8.23 in \cite{Emerton-formal-alg-stk}).
\begin{cor}\label{cor:normal}
    The $p$-adic formal algebraic stack $Y^{\le\lam+\eta,\tau}$ is analytically normal.% in the sense of \cite[Def.~8.22]{Emerton-formal-alg-stk} (also see Rmk.~8.23 in \loccit).
\end{cor}
\begin{proof}
    This follows from Theorem \ref{thm:BK-hodge-localmodel} and \cite[Thm.~1.0.1]{Levin-localmodel} (which is built on \cite[Thm.~1.1]{PZ13-Inv-local_model-MR3103258}).
\end{proof}

\subsection{Global functions}\label{subsec:global-func}

We start by introducing some notations. Let $P\subset \GL_n$ be a parabolic subgroup with Levi subgroup $M$. Let $R$ be an $\cO$-algebra. For $A\in \Lie P (R\DB{v})$, we denote by $\ov{A}\in \Lie M (R\DB{v})$ the image of $A \mod v$ under the quotient map $\Lie P \epi \Lie M$. This applies to $A=A_{\fM,\be}\ix{j} \in \Lie P_j(R\DB{v})$ for $(\fM,\be) \in Y^{[0,h],\tau,\be}(R)$.

We want to understand the affect of changing $\be$ on $\ov{A}_{\fM,\be}\ix{j}$. We need the following lemma.

\begin{lem}\label{lem:M_j}
    For $j\in \cJ$, we have $M_j = s_j^\mo M_{j-1} s_j$.
\end{lem}
\begin{proof}
    Note that $s_j^\mo = s_{\rmor, j}^\mo s_{\rmor, j-1}$. By the definition of $M_{j-1}$, $U_{\al}\subset M_{j-1}$ for $\al \in \Phi$ is equivalent to
    \begin{align*}
        \RG{s_{\rmor, j-1}^\mo (\bfa\ix{j-1}) , \al^\vee} = 0 & \iff \RG{s_{\rmor, j}^\mo (\bfa\ix{j-1}) , s_j^\mo (\al^\vee)} = 0 \\
        & \iff  \RG{s_{\rmor, j}^\mo (\bfa\ix{j}) , s_j^\mo (\al^\vee)} = 0
        \\
        & \iff U_{s_j^\mo (\al)} \subset M_{j}. \qedhere
    \end{align*}
    %where the second equivalence follows from 
    \end{proof}
Recall that $W_{\tau_M}$ acts on $\bfn$ and its orbits are precisely fiber of the map $i\in \bfn \mapsto \chi_i$. For $i \in \bfn$, we write $\ov{i}:=W_{\tau_M}(i)$. Then $s_{\rmor,j}^\mo(\ov{i})$ corresponds to a block in $M_j$. For $A\in M_j(R)$, we denote by $A_{s_{\rmor,j}^\mo(\ov{i})}$ the block of $A$ corresponding to $s_{\rmor,j}^\mo(\ov{i})$. 

%Recall that we have defined two partitions $\fP_\tau$ and $\ov{\bfn}_\tau$ of $\bfn$ in \S\ref{subsec:ILLC}. For each $\ov{i}\in \ov{\bfn}_\tau$ and $j'\in \cJ'$, $s_{\rmor,j'}^\mo(\ov{i})$ corresponds to a block of the Levi factor of $P_j$. For any $A\in M_j(R)$, we denote by $A_{s_{\rmor,j'}^\mo(\ov{i})}$ the block of $A$ corresponding to $s_{\rmor,j'}^\mo(\ov{i})$.

%For $(\fM,\be) \in Y^{[0,h],\tau,\be}(R)$, we denote by  $\ov{A}_{\fM,\be}\ix{j} \in M_j(R)$ the image of  $A_{\fM,\be}\ix{j} \mod v$ under $P_j \epi M_j$. We want to understand how changing $\be$ affects $\ov{A}_{\fM,\be}\ix{j}$. For $(P\ix{j})_{\jj} \in \prod_{\jj} \cP_j(R)$, we write $\ov{P}\ix{j} \in M_j(R)$ for the Levi factor of $P\ix{j} \mod v$. Let $\be_1$ and $\be_2$ be eigenbases for $\fM$ such that $\be_1 = \be_2 D\ix{j}$ where $D\ix{j} = \Ad(  u^{\bfa\ix{j}}s_{\rmor,j})(P\ix{j})$ as usual. By Proposition \ref{prop:change-of-basis-formula}, we have
%\begin{align*}
%    \ov{A}\ix{j}_{\fM,\be_2}  = \ov{P}\ix{j} \ov{A}\ix{j}_{\fM,\be_1} \Ad(s_j^\mo)(\ov{P}\ix{j-1})^\mo.
%\end{align*}

\begin{lem}\label{lem:func-change-basis}
    Let $\fM \in Y^{[0,h],\tau}(R)$ and $\be_1$ and $\be_2$ be eigenbases for $\fM$ such that $\be_1\ix{j} = \be_2\ix{j} D\ix{j}$ where $D\ix{j} = \Ad(  u^{\bfa\ix{j}}s_{\rmor,j})(P\ix{j})$ as usual. Then, we have
\begin{align*}
    \ov{A}\ix{j}_{\fM,\be_2}  = \ov{P}\ix{j} \ov{A}\ix{j}_{\fM,\be_1} \Ad(s_j^\mo)(\ov{P}\ix{j-1})^\mo.
\end{align*}
\end{lem}
\begin{proof}
    This essentially follows from Proposition \ref{prop:change-of-basis-formula}. One has to check that for
    \begin{align*}
        Q:=\Ad(s_{j}^\mo v^{\mu_{j}+\eta_{j}})(\varphi(P\ix{j-1})^\mo),
    \end{align*}
    $\ov{Q}$ is equal to $\Ad(s_j^\mo)(\ov{P}\ix{j-1})^\mo$. This follows from
    \begin{align*}
          \Ad(s_{j}^\mo v^{\mu_{j} +\eta_{j}} ) (\varphi(P\ix{j-1})^\mo) = \Ad(v^{s_{j}^\mo(\mu_{j} +\eta_{j})}) (\Ad(s_j^\mo(\varphi(P\ix{j-1})^\mo))  %&= \Ad( s_{\rmor,j'}^\mo u^{-\bfa\ix{j'}})(\varphi(D\ix{j'-1})^\mo)
          %\\
          %&=  \Ad(  u^{-s_{\rmor,j'}^\mo(\bfa\ix{j'})})( s_{\rmor,j'}^\mo \varphi(D\ix{j'-1})^\mo s_{\rmor,j'}),
    \end{align*}
    and the fact that $v^{s_{j}^\mo(\mu_{j} +\eta_{j})} \in Z(M_j)(R(v))$.
%    and the fact that $u^{-s_{\rmor,j'}^\mo(\bfa\ix{j'})} \in Z(M_j)(R(v))$, $\ov{Q}$ is equal to the image of $s_{\rmor,j'}^\mo \varphi(D\ix{j'-1})^\mo s_{\rmor,j'} \mod u$ under the natural quotient map $\Lie \GL_n \epi \Lie M_j$. Now using
%    \begin{align*}
%        s_{\rmor,j'}^\mo \varphi(D\ix{j'-1})^\mo s_{\rmor,j'} &= \Ad(s_j^\mo u^{ s_{\rmor,j'-1}^\mo(p\bfa\ix{j'-1})}) (\varphi (P\ix{j'-1})^\mo) 
%        \\
%        &= \Ad(s_j^\mo u^{ s_{\rmor,j'-1}^\mo(p\bfa\ix{j'-1})} s_j) (\Ad(s_j^\mo) (\varphi (P\ix{j'-1})^\mo)) 
%    \end{align*}
%    and the fact that $s_j^\mo u^{ s_{\rmor,j'-1}^\mo(p\bfa\ix{j'-1})} s_j \in Z(M_j)(R(v))$, $\ov{Q}$ is equal to $\Ad(s_j^\mo)(\ov{P}\ix{j-1})^\mo$.
\end{proof}

\begin{lem}\label{lem:func-independent}
    Let $\fM \in Y^{[0,h],\tau}(R)$ with an eigenbasis $\be$.  For $\ov{I}\in \ov{\fP}_\tau$ with a representative $I$ and $i\in I$, the matrix
    \begin{align*}
       \ov{A}_{\fM,\be}\ix{\ov{I}} :=  \prod_{k=n_{\ov{I}}-1}^{0}\PR{\prod_{j=f-1}^0 \Ad (s_\tau^{k} s_{\rmor,j})(\ov{A}\ix{j}_{\fM,\be})}_{\ov{i}}.
    \end{align*}
    is independent of the choice of $\be$, $I$, and $i$ up to conjugation.
\end{lem}
\begin{proof}
    Let $\be_1, \be_2$ be  eigenbases such that $\be_1\ix{j}=\be_2\ix{j} D\ix{j}$ with $D\ix{j} = \Ad(  u^{\bfa\ix{j}}s_{\rmor,j})(P\ix{j})$ as usual. By Lemma \ref{lem:func-change-basis}, we have
    \begin{align*}
    \prod_{j=f-1}^0 \Ad (s_{\tau}^k s_{\rmor,j})(\ov{A}\ix{j}_{\fM,\be_2}) = \Ad (s_\tau^{k+1})(\ov{P}\ix{f-1})\PR{ \prod_{j=f-1}^0 \Ad (s_{\tau}^k s_{\rmor,j})(\ov{A}\ix{j}_{\fM,\be_1}) } \Ad (s_\tau^{k})(\ov{P}\ix{f-1})^\mo
\end{align*}
for $k\in \Z$ (note that we have chosen $s_{\rmor,f'-1}=1$ and $s_{\rmor, f-1}=s_\tau$).  Using the fact that $s_{\tau}^{n_{\ov{I}}}(\ov{i})=\ov{i}$, we have
\begin{align*}
    \ov{A}_{\fM,\be_2}\ix{\ov{I}} = \ov{P}\ix{f-1}_{ \ov{i}} \ov{A}_{\fM,\be_1}\ix{\ov{I}} (\ov{P}\ix{f-1}_{ \ov{i}})^\mo.
\end{align*}
For different choice of $i\in I$, note that $\ov{A}_{\fM,\be_1}\ix{\ov{I}}$ is only conjugated by a power of $s_\tau$.
\end{proof}

We fix a dominant weight $\lam \in X^*(\uT)_+$ until the end of this section. %For a finite type $p$-adic formal algebraic stack $\cX$ over $\cO$, we denote by $\cX^{\rig}$ its generic fiber, which is an adic stack over $E$ \cite[Def.~4.31]{Oh-Shimizu} (see also, \cite[Prop.~5.1.13]{EGH22} and the preceding discussion).

\begin{defn}
     For $\fM$, $\ov{I}$ as above and an integer $1\le d \le r$ where $r$ is the size of $\ov{I}$ (equivalently, $\ov{A}_{\fM,\be}\ix{\ov{I}}$ is an $r\times r$ matrix),   
     we define  $f_{\ov{I},d}\in \cO(Y^{\le\lam+\eta,\tau})[p^\mo]$ to be the function taking $\fM\in Y^{\le\lam+\eta,\tau}(\cO')$ for any finite $\cO$-flat algebra $\cO'$, to the coefficient of degree $(r-d)$-th term, multiplied by $(-1)^d$, of the characteristic polynomial of $(\ov{A}_{\fM,\be}\ix{\ov{I}})^\mo$.
\end{defn}

\begin{rmk}
\begin{enumerate}
    \item Note that $f_{\ov{I},d}(\fM)$ does not depend on the choices of $\be$, $I$, and $i$ by Lemma \ref{lem:func-independent} and Proposition \ref{prop:eigenbasis-torsor}.
    \item Since $Y^{\le\lam+\eta,\tau}$ is $\cO$-flat, $\cO(Y^{\le\lam+\eta,\tau})$ is an $\cO$-flat algebra and mapped injectively into $\cO(Y^{\le\lam+\eta,\tau})[p^\mo]$. 
    
    \item We explain why $f_{\ov{I},d}$ defines an element in $\cO(Y^{\le\lam+\eta,\tau})[p^\mo]$. Note $g_{\ov{I},d}$ defined similarly to $f_{\ov{I},d}$ using $\ov{A}_{\fM,\be}\ix{\ov{I}}$ instead of $(\ov{A}_{\fM,\be}\ix{\ov{I}})^\mo$ is a global function on $Y^{\le\lam+\eta,\tau}$. Moreover, $g_{\ov{\fP}_\tau, d_{\ov{\fP}_\tau}}$ where $d_{\ov{\fP}_\tau} = (\#\ov{I})_{\ov{I}\in \ov{\fP}_\tau}$ takes $\fM\in Y^{\le\lam+\eta,\tau}(\cO')$ to $\pi_K^m (\cO')^\times$ where $m= \sum_{\tiljj}\sum_{i=1}^n \lam_{\tilj,i}+n-i.$ By Corollary \ref{cor:normal} and \cite[Thm.~7.4.1]{deJong95}, $\pi_K^{-m}g_{\ov{\fP}_\tau, d_{\ov{\fP}_\tau}}$ defines an invertible element in $\cO(Y^{\le\lam+\eta,\tau})$. Then one can see that $\pi_K^{-m}g_{\ov{\fP}_\tau, d_{\ov{\fP}_\tau}}f_{\ov{I},d}$ is given by $g_{\ov{I}',d'}$ for some $(\ov{I}',d')$, justifying our claim.%   Moreover, since $Y^{\le\lam+\eta,\tau}$ is residually Jacobson, its global functions are determined by their values at $\cO'$-points for finite $\cO$-flat algebras $\cO'$. This implies that
    %\begin{align*}
    %    \pi_K^{\sum_{\tiljj}\RG{\omega_{d\#I}, (\lam+\eta)_{\tilj}}}f_{\ov{I},d} \in \cO(Y^{\le\lam+\eta,\tau})
    %\end{align*}
    %because the elementary divisor of $\ov{A}_{\fM,\be}\ix{j}$ is bounded by $\pi_K^{\sum_{\tilj\in\tilcJ, \tilj|_{K_0}=j}(\lam+\eta)_{\tilj}}$ where the sum runs over $\tilj$ whose image in $\cJ$ is $j$.
\end{enumerate}
\end{rmk}

%\begin{lem}\label{lem:global-functions}
%    For $\ov{I}$ and $d$ as above, $f_{\ov{I},d}$ descends to a function on $Y^{[0,h],\tau}$.
%\end{lem}

As in \S\ref{subsec:int-hecke-alg} and \ref{subsec:loc-alg-hecke}, we consider a pair $(\ov{\fI}, d_{\ov{\fI}})$ where $\ov{\fI}\subset \ov{\fP}_\tau$ and $d_{\ov{\fI}}=(d_{\ov{I}})_{\ov{I}\in \ov{\fI}}$ such that $1 \le d_{\ov{I}}\le \#\ov{I}$. Recall that $n(\ov{\fI}, d_{\ov{\fI}}) = \sum_{\ov{I}\in \ov{\fI}}d_{\ov{I}}n_{\ov{I}}$.% and $n_{\lam}(\ov{\fI}, d_{\ov{\fI}})$ for dominant weight $\lam \in X^*(\uT)_+$.

\begin{defn}\label{defn:glob-func}
    For a pair $(\ov{\fI}, d_{\ov{\fI}})$ as above, we define
    \begin{align*}
        F_{\ov{\fI},d_{\ov{\fI}}}& :=\prod_{\ov{I}\in \ov{\fI}} f_{\ov{I},d_{\ov{I}}} \in \cO(Y^{\le\lam+\eta,\tau})[\tfrac{1}{p}] \\
        \til{F}_{\ov{\fI},d_{\ov{\fI}}} &:=\pi_K^{\sum_{\tiljj} \RG{\om_{n(\ov{\fI}, d_{\ov{\fI}})}, \lam_{\tilj}}}q^{ \RG{\om_{n(\ov{\fI}, d_{\ov{\fI}})}, \eta}} F_{\ov{\fI},d_{\ov{\fI}}} \in \cO(Y^{\le\lam+\eta,\tau})[\tfrac{1}{p}].
    \end{align*}
    %If $\ov{\fI}$ consists of a single element $\ov{I}$, we simply write $F_{\ov{I},d_{\ov{I}}}$ for $F_{\ov{\fI},(d_{\ov{I}})}$ and similarly for the normalized ones.
\end{defn}

\begin{rmk}\label{rmk:regular-case-func}
    When $\tau$ is regular, then $\#\ov{I}=1$ for all $\ov{I}\in \ov{\fI}$ and $d_{\ov{I}}$ is forced to be 1. In this case, for any $\fI\subset \fP_\tau$, we write $f_{\fI} = f_{\ov{\fI},(1)_{\ov{I}\in \ov{\fI}}}$ and similarly for $F_{\fI}$ and $\til{F}_{\fI}$ where $\ov{\fI}$ is the image of $\fI$ in $\ov{\fP}_\tau$ (cf.~Remark \ref{rmk:regular-case}).
\end{rmk}

%\begin{example}\label{ex:functions}
%    We describe the functions $\til{F}_{\ov{\fI}, (d_{\ov{I}})}$ for $\tau$ as in Example \ref{ex:main} (also, see Example \ref{ex:hecke-pres}).
%    \begin{itemize}
%        \item [(PS)] 
%    \end{itemize}
%\end{example}

\begin{lem}\label{lem:global-functions/p}
    For a pair $(\ov{\fI}, d_{\ov{\fI}})$ as above, the function $\til{F}_{\ov{\fI},d_{\ov{\fI}}}$ is contained in $\cO(Y^{\le\lam+\eta,\tau})$.
\end{lem}
\begin{proof}
    By Corollary \ref{cor:normal} and \cite[Thm.~7.4.1]{deJong95}, it suffices to show that for any finite extension $E'/E$ with ring of integers $\cO'$ and $\fM\in Y^{\le\lam+\eta,\tau}(\cO')$, $\til{F}_{\ov{\fI},d_{\ov{\fI}}}(\fM)$ is in $\cO'$. Let $\be$ be an eigenbasis for $\fM$. For each $\ov{I}\in \ov{\fI}$, $f_{\ov{I},d_{\ov{I}}}(\fM)$ is a sum of $d_{\ov{I}} \times d_{\ov{I}}$ minor of the matrix $(\ov{A}\ix{\ov{I}}_{\fM,\be})^\mo$. Since $\PR{\prod_{j=f-1}^0 \Ad (s_{\tau}^ks_{\rmor,j})(\ov{A}\ix{j}_{\fM,\be})}_{\ov{i}}$ for $0 \le k \le n_{\ov{I}} -1$ are distinct Levi blocks in $\prod_{j=f-1}^0 \Ad (s_{\rmor,j})(\ov{A}\ix{j}_{\fM,\be})$, Lemma \ref{cauchy-binet} shows that $f_{\ov{I},d_{\ov{I}}}(\fM)$ is equal to a sum of $d_{\ov{I}}n_{\ov{I}} \times d_{\ov{I}}n_{\ov{I}}$ minor of the matrix $\prod_{j=f-1}^0 \Ad (s_{\rmor,j})(\ov{A}\ix{j}_{\fM,\be})^\mo$. Finally, by taking product over all $\ov{I}\in \ov{\fI}$ and noting that blocks in $\prod_{j=f-1}^0 \Ad (s_{\rmor,j})(\ov{A}\ix{j}_{\fM,\be})$ for different $\ov{I}$ are all distinct, we can show that $F_{\ov{\fI},d_{\ov{\fI}}}(\fM)$ is given by a sum of $n(\ov{\fI}, d_{\ov{\fI}}) \times n(\ov{\fI}, d_{\ov{\fI}})$ minor of $\prod_{j=f-1}^0 \Ad (s_{\rmor,j})(\ov{A}\ix{j}_{\fM,\be})^\mo$. Then the claim follows from the fact that $\prod_{j=f-1}^0 \Ad (s_{\rmor,j})({A}\ix{j}_{\fM,\be})$ has elementary divisor bounded by $\prod_{\tiljj}(v-\sig_{\tilj}(\pi_K))^{\lam_{\tilj}+\eta_{\tilj}}$.
    %Then $\prod_{j=f-1}^0 \Ad (s_{\rmor,j})({A}\ix{j}_{\fM,\be})$ has elementary divisor bounded by $\prod_{\tiljj}(v-\sig_{\tilj}(\pi_K))^{\lam_{\tilj}+\eta_{\tilj}}$. Thus, any $i\times i$ minor of it is divisible by $\prod_{\tiljj}(v-\sig_{\tilj}(\pi_K))^{\RG{(w_0\omega_{i,\tilj})_{\tiljj}, \lam+\eta}}$ in $\cO'\DB{v}$. By Lemma \ref{cauchy-binet}, $F_{\ov{\fI},(i_{\ov{I}})}$ is a sum of $d_{\ov{\fI},(d_{\ov{I}})}\times d_{\ov{\fI},(d_{\ov{I}})}$ minors of  $\prod_{j=f-1}^0 \Ad (s_{\rmor,j})({A}\ix{j}_{\fM,\be}) \mod v$. This proves the claim.
\end{proof}

\begin{lem}\label{cauchy-binet}
    Let $R$ be a commutative ring and $A,B\in \GL_n(R)$. For any $1 \le k \le n$, any $k\times k$ minor of $AB$ is a sum of $2k \times 2k$ minors of $\begin{psmallmatrix}
        A & \\ & B
    \end{psmallmatrix} \in \GL_{2n}(R)$.
\end{lem}
\begin{proof}
    This follows from the generalized Cauchy--Binet formula (see \cite[Thm.~1]{CauchyBinet}).  %The following argument was suggested by the user Dabouliplop in the Mathematics Stack Exchange post \cite{MSE}. Let $I,J\subset \bfn$ be subsets of size $k$. Let $V=R^n$ be the free $R$-module with standard basis $\CB{e_1,\dots,e_n}$. Then the $I\times J$ minor of $AB$ is equal to the $\wedge_{i\in I}e_i \in \wedge^k$-coefficient of $(\wedge^k AB)(\wedge_{j\in J}e_i)$. Since $\wedge^kAB = \wedge^kA \wedge^kB$, this is equal to the sum of the products of $I\times K$ minor of $A$ and $K\times J$ minor of $B$ where the sum runs over all subsets $K\subset \bfn$ of size $k$.   
\end{proof}

\subsection{Connection to Weil--Deligne representations and Hecke algebras}\label{subsec:WD}
%Now we replace $\tau$ by its dual $\tau^\vee$. Note that if $\tau = \oplus_{i=1}^n \chi_i$, then $\tau^\vee = \oplus_{i=1}^n \chi_i^\mo$. Thus, $s_{\tau}$, $\fP_{\tau}$, $\ov{\fP}_{\tau}$ and other related data remain the same when $\tau$ is replaced by $\tau^\vee$.  

For $\fM\in Y^{\le\lam+\eta,\tau}(\cO)$, let $D(\fM)$ be the Frobenius-semisimplification of the Weil--Deligne representation associated to $\fM$ (see  \cite[\S9.3]{LLMPQ-FL} for a detailed construction). Then $D(\fM)$ satisfies the followings:
\begin{enumerate}
    \item $D(\fM)|_{I_K} \simeq \tau^\vee$
    \item Let $\be$ be an eigenbasis for $\fM$. Then $\CB{f\ix{f'-1}_1 ,\dots, f\ix{f'-1}_n}$ is a basis for $D(\fM)$.
    \item The matrix presentation of the $\Frob_K$-action on $D(\fM)$ with respect to the above basis is given by the semisimplification of
    \begin{align*}
    s_\tau^\mo C\ix{f-1}_{\fM,\be}\cdots C\ix{1}_{\fM,\be}C\ix{0}_{\fM,\be} \mod u.
\end{align*}
    Note that we have $C\ix{j}_{\fM,\be}\mod u = \Ad(s_{\rmor,j})(\ov{A}\ix{j}_{\fM,\be})$ by equation \eqref{eqn:mat-A-C-change}.
\end{enumerate}

Let $\ov{I}\in \ov{\fP}_\tau$. We write $\tau_{\ov{I}}:= \tau_I$ where $I\in \ov{I}$ (it is independent of the choice of $I$). We denote by $D(\fM)_{\tau_{\ov{I}}}$ the $\tau_{\ov{I}}^\vee$-eigenspace in $D(\fM)$. By Frobenius-semisimplicity, we can further decompose $D(\fM)_{\tau_{\ov{I}}} \simeq \oplus_{I\in \ov{I}} \phi_I$ where $\phi_I$ is an irreducible Weil--Deligne representation such that $\phi_I|_{I_K}\simeq \tau_I^\vee$.

For the following result, recall that for integers $d,m\ge 1$, $\sym_d(x_1,\dots,x_m)$ denotes the elementary symmetric polynomial of degree $d$ in variables $x_1,\dots,x_m$.

\begin{prop}\label{prop:glob-func-WD}
    For $\ov{I}\in \ov{\fP}_\tau$ and $1 \le d \le \#\ov{I}$, $f_{\ov{I},d}(\fM)$ is equal to $\sym_d(\det(\phi_I(\Frob_K^\mo)))_{I\in \ov{I}}$. 
\end{prop}
\begin{proof}
    By item (3) above, $\Frob_K^{n_{\ov{I}}}$-action on $D(\fM)_{\tau_I}$ is given by the matrix
    \begin{align*}
        \PR{s_\tau^\mo \prod_{j=f-1}^0 \Ad(s_{\rmor,j})(\ov{A}\ix{j}_{\fM,\be})}^{n_{\ov{I}}} = \PR{\prod_{k=-1}^{-n_{\ov{I}}}  \prod_{j=f-1}^0 \Ad(s_\tau^k s_{\rmor,j})(\ov{A}\ix{j}_{\fM,\be})} s_\tau^{-n_{\ov{I}}},
    \end{align*}
    whose block at $\ov{i}$ for some $I\in \ov{I}$ and $i\in I$ was used to define $f_{\ov{I},d}$. 
    
    Let $F/K$ be the unramified extension of degree $\#I$. Then $D(\fM)_{\tau_{\ov{I}}}|_{W_F}$ decomposes into a sum of characters
    \begin{align*}
        \oplus_{I\in \ov{I}}\oplus_{i\in I} \om_{F,\sig_0}^{-\bfa_i'}\ur_{x_I}
    \end{align*}
    where $\bfa_i'$ is the integer defined in \S\ref{subsec:tame-inertial-types} and $x_I = \mathrm{det} (\phi_{I}(\Frob_K))$. The $\ov{i}$-block of the above matrix corresponds to the sum of characters $\oplus_{i\in \ov{i}}\om_{F,\sig_0}^{-\bfa_i'}\ur_{x_I}$ where for each $i\in \ov{i}$, $I\in \ov{I}$ is the one containing $i$. Then the claim follows from the definition of $f_{\ov{I},d}$.
\end{proof}

We finish this section by relating these global functions to Hecke operators.

\begin{thm}\label{thm:hecke-global-func}
    Let $\lam \in X^*(T)^{\tilcJ}$ be a dominant character. Then there is an injective morphism
    \begin{align*}
        \Psi^{\lam,\tau}: \cH(\sig^\circ(\lam,\tau)) \mono \cO(Y^{\le\lam+\eta,\tau})
    \end{align*}
    such that for any $(\ov{\fI},d_{\ov{\fI}})$ as above, $\til{T}_{\ov{\fI},d_{\ov{\fI}}}$ is mapped to $\til{F}_{\ov{\fI},d_{\ov{\fI}}}$. %$\prod_{\ov{I}\in \ov{\fI}}\sym_{i_{\ov{I}}}(T_{-\veps_I})_{I\in \ov{I}}$ is mapped to $\pi_K^{-n_\lam(\ov{\fI},(i_{\ov{I}}))}f_{\ov{\fI},(i_{\ov{I}})}$. 
    Moreover, $\Psi^{\lam,\tau}$ interpolates the local Langlands correspondence in the following sense: for any finite extension $\cO'/\cO$, if $\fM\in Y^{\le\lam+\eta,\tau }(\cO')$, then the natural $\sig^\circ(\lam,\tau)$-action on $\Hom_{\rmK}(\sig^\circ(\lam,\tau),r_p^\mo( D(\fM)^\vee) \otimes V(\lam)\mid_\rmK)$ is given by the character
    \begin{align*}
        \cH(\sig^\circ(\lam,\tau)) \xra{\Psi^{\lam,\tau}} \cO(Y^{\le\lam+\eta,\tau}) \xra{\fM} \cO'.
    \end{align*}
    %such that for any $I\in \fP_\tau$ and $i \in \CB{1,2,\dots,\#I}$, $\sum_{J} T_{J}$ is mapped to $f_{I,i}$ where the sum runs over the set of subsets of $I$ of size $i$.
\end{thm}

\begin{proof}
    The existence of the morphism $\Psi^{\lam,\tau}$ follows from Proposition \ref{prop:int-hecke-loc-alg} and Lemma \ref{lem:global-functions/p}. For the injectivity, it suffices to prove the injectivity of $\Psi^{\lam,\tau}$ after inverting $p$. Recall that $\cH(\sig(\lam,\tau))$ is generated by $T'_{\ov{I},d}$ where
    \begin{align*}
        T'_{\ov{I},d} = \sym_d(T'_{-\veps_{I_1}}, \dots, T'_{-\veps_{I_r}})
    \end{align*}
    for $\ov{I} = \{I_1,\dots, I_r\} \in \ov{\fP}_{\tau}$ and $1 \le d \le r$. Then $\Psi^{\lam,\tau}$ maps $T'_{\ov{I},d}$ to $F_{\ov{I},d}$. 
    %that there is no relation between functions $\til{F}_{\ov{\fI},(d_{\ov{I}})}$ for different $(\ov{\fI},(d_{\ov{I}}))$. 
    
    For each $I\in \fP_\tau$, define a dominant cocharacter $\lam_I := (\lam_{i,\tilj})_{i\in I,\tilj\in \tilcJ} \in X_*(T_{\#I})^{\tilcJ}$ and similarly $\eta_I$. Then we have a morphism induced by taking direct products
    \begin{align*}
        \prod_{I\in \fP_\tau}Y^{\le \lam_I +\eta_I, \tau_I} \ra Y^{\le\lam+\eta,\tau},
    \end{align*}
    and by taking pullback, we get
    \begin{align*}
        \cH(\sig(\lam,\tau)) \xra{\Psi^{\lam,\tau}} \cO(Y^{\le\lam+\eta,\tau})[p^\mo] \ra \otimes_{I\in \fP_\tau}\cO(Y^{\le \lam_I +\eta_I, \tau_I})[p^\mo].
    \end{align*}
    To prove that $\Psi^{\lam,\tau}$ is injective, we show that the above composition is injective.

    Note that for $I\in \fP_{\tau}$, $\tau_I$ is cuspidal. As in Example \ref{ex:hecke-pres}(C), the only possible pair $(\ov{\fI},d_{\ov{\fI}})$ associated to $\tau_I$ is $(\{\{ I\}\},1)$, and we get $F_{I}:= F_{\ov{\fI},d_{\ov{\fI}}}\in \cO(Y^{\le \lam_I +\eta_I, \tau_I})$. We claim that the polynomial ring $E[F_{I}^\pm]$ embeds into $\cO(Y^{\le \lam_I +\eta_I, \tau_I})[p^\mo]$ for each $I\in \fP_\tau$. To see this, it suffices to show that $F_I$ takes infinitely many values when evaluated at all $\fM\in Y^{\le \lam_I +\eta_I, \tau_I}(\cO)$. Using unramified twists of Breuil--Kisin modules (cf.~\cite[Definition 3.3.2]{cegsC}), we can show that if $F_I(\fM)= c$, there is $\fM_{t}$ such that $F_I(\fM)=ct^r$ for any $t\in \cO^\times$ where $r$ is the rank of $\fM$. Then the claim follows from the fact that $F_I$ always takes nonzero value.
    
    Let $G_{\ov{I},d}\in \otimes_{I\in \fP_\tau}\cO(Y^{\le \lam_I +\eta_I, \tau_I})[p^\mo]$ be the image of $F_{\ov{I},d}$. By its construction, we have $G_{\ov{I},d} = \sym_d(F_{I_1},\dots, F_{I_r})$. Since $\cH(\sig(\lam,\tau))$ is generated by $T'_{\ov{I},d}$ for $\ov{I}$ and $d$ as above, this shows that the image of $\cH(\sig(\lam,\tau))$ under the composition is contained in $E[F_{I}^\pm]_{I\in \fP_{\tau}}$ and that there is no relation between $G_{\ov{I},d}$ for different choices of $(\ov{I},d)$. This shows that $\Psi^{\lam,\tau}$ is injective. %Note that for $I\in \fP_{\tau}$, $\tau_I$ is cuspidal. As in Example \ref{ex:hecke-pres}(C), the only possible pair $(\ov{\fI},d_{\ov{\fI}})$ associated to $\tau_I$ is $(\{\{ I\}\},1)$, and we get $F_{I}:= F_{\ov{\fI},d_{\ov{\fI}}}\in \cO(Y^{\le \lam_I +\eta_I, \tau_I})$. By its construction, we have $G_{\ov{I},d} = \sym_d(F_{I_1},\dots, F_{I_r})$. This shows that there is no relation between $G_{\ov{I},d}$ for different choice of $(\ov{I},d)$. Therefore, $\Psi^{\lam,\tau}$ is injective.
    
    Now we prove that $\Psi^{\lam,\tau}$ interpolates the local Langlands correspondence by following the argument in \cite[Prop.~10.1.5]{LLMPQ-FL}. We first consider the case $\lam=0$. Recall that $D(\fM) \simeq \oplus_{I\in \fP_\tau}\phi_I$. By \cite[Prop.~3.10]{6author}, after permuting $I$'s in each $\ov{I}\in \ov{\fP}_\tau$ if necessary, $\rec_p^\mo(D(\fM)^\vee)$ is the unique irreducible quotient of
    \begin{align*}
         i_{P(K)}^{G(K)} \otimes_{I\in \fP_\tau}\rec_p^\mo(\phi_I^\vee)
    \end{align*}
    We have isomorphisms between $\cH(\rmM,\lam_M)^{W_{[M,\omega]}} \simeq \cH(\rmG,\sig(\tau))$-modules
    \begin{align*}
        &\Hom_{\rmG}(\ind_\rmK^\rmG \sig(\tau), \rec_p^\mo(D(\fM)^\vee)  \\
        \simeq  & \Hom_{\rmG}( \ind_{\cP}^\rmG \lam, i_{P(K)}^{G(K)} \otimes_{I\in \fP_\tau}\rec_p^\mo(\phi_I)^\vee) \\
        \simeq & \Hom_{\rmM}(\ind_{M(\cO_K)}^{\rmM}\lam_M, \otimes_{I\in \fP_\tau}\rec_p^\mo(\phi_I)^\vee)
    \end{align*}
    where $\cH(\rmG,\sig(\tau))$ acts on the second module by restriction of the $\cH(\rmG,\lam)$-action. The first isomorphism follows from \cite[Thm.~4.1]{Dat-caracteres}, and the second isomorphism follows from the property of $t_P$ \cite[pg.~55]{BKssGLn}.
    %(diagram (5) in \loccit; note that the parabolic induction using lower triangular parabolic can be replaced by the upper triangular parabolic because of Prop.~2.1). 
    Noting that $\rec_p^\mo(-) = r_p^\mo(-)\otimes \abs{\det}^{\tfrac{n-1}{2}}$, we have an isomorphism between $\cH(\rmM,\lam_M)^{W_{[M,\omega]}} \simeq \cH(\rmG,\sig(\tau))$-modules
    \begin{align}
        &\Hom_{\rmG}(\ind_\rmK^\rmG \sig(\lam,\tau), r_p^\mo(D(\fM))^\vee|_{\rmK}) \label{eqn:hecke-module} \\
        \simeq & \Hom_{\rmM}(\ind_{M(\cO_K)}^{\rmM}\lam_M, \otimes_{I\in \fP_\tau}(\rec_p^\mo(\phi_I^\vee)|_{M(\cO_K)}\otimes \abs{\mathrm{det}_{\#I}}^{\tfrac{1-n}{2}})) \nonumber.
    \end{align}
    Then $t_{-\veps_I}$ acts on the latter by the scalar given by which $\pi_K$ acts on $\rec_p^\mo(\phi_I^\vee)|_{M(\cO_K)}\otimes \abs{\mathrm{det}_{\#I}}^{\tfrac{1-n}{2}}$. This equals to $\det(\phi_I(\Frob_K^\mo)) q^{\tfrac{(n-1)\#I}{2}}$. Thus, for $\ov{I}\in \ov{\fP}_\tau$ and $1\le d\le \#\ov{I}$, $\sym_d(t_{-\veps_I})_{I\in \ov{I}}$ acts on \eqref{eqn:hecke-module} by
    \begin{align*}
        q^{\tfrac{d\#I(n-1)}{2}}\sym_d(\det(\phi_I(\Frob_K^\mo)))_{I\in \ov{I}}.
    \end{align*}
    By Lemma \ref{lem:conv-formula},
    \begin{align*}
        t_P(\sym_d(t_{-\veps_I})_{I\in \ov{I}}) = q^{-\tfrac{d\#I(n-d\#I)}{2}}\til{T}_{\ov{I},d}.
    \end{align*}
    Thus, $\til{T}_{\ov{I},d}$ acts on \eqref{eqn:hecke-module} by
    \begin{align*}
        q^{\tfrac{d\#I(2n-d\#I-1)}{2}}\sym_d(\det(\phi_I(\Frob_K^\mo)))_{I\in \ov{I}}% = q^{\tfrac{d\#I(2n-d\#I-1)}{2}}\sym_d(\det(\phi_I(\Frob_K)))_{I\in \ov{I}}.
    \end{align*}
    By Proposition \ref{prop:glob-func-WD}, this exactly matches with $\til{F}_{\ov{I},d}$ evaluated at $\fM$. % the image of $T_{\ov{I},d}$ under $\ev_{\fM}\circ \Psi^{\tau}$.
    Using Lemma \ref{lem:tame-type-formulas}, this completes the case $\lam=0$. The general case easily follows from the above argument and Remark \ref{rmk:loc-alg-hecke-mod}. 
\end{proof}
\section{Spectral mod $p$ Satake isomorphism}\label{sec5}

In \S\ref{subsec:EGstack}, we first recall the Emerton--Gee stack $\cX_n$ and its basic properties. In particular, we discuss the bijection between Serre weights of $\GL_n(k)$ and irreducible components in $\cX_{n,\red}$, and the relationship between potentially crystalline stack and the moduli stack of Breuil--Kisin modules. In \S\ref{subsec:GLn}, we deduce our main result which says that for a sufficiently generic Serre weight $\sig$, the Hecke algebra $\cH(\sig)$ is naturally isomorphic to $\cO(\cC_\sig)$ the ring of global functions on the irreducible component corresponding to $\sig$ (Theorem \ref{thm:spec-satake}). As a preliminary, we prove its version for tori in \S\ref{subsec:tori} using local class field theory. Then we show that $\cO(\cC_\sig)$ injects into $\cH(\sig)$ for non-Steinberg $\sig$. This essentially follows from the case of $\GL_2$ and is treated in \S\ref{subsec:GL2}. The surjectivity follows from the result of the previous section and requires a genericity condition on $\sig$.

\subsection{Recollections on the Emerton--Gee stack}\label{subsec:EGstack}

In \cite{EGstack}, Emerton and Gee constructed $\cX_{n}$ the moduli stack of projective rank $n$ \'etale $\pgma$-modules. It is a Noetherian formal algebraic stack over $\Spf \cO$. For a dominant cocharacter $\lam\in X_*(T)^{\tilcJ}$ and a tame inertial type $\tau$, there is a closed substack $\cX_n^{\lam,\tau} \subset \cX_n$ which is a $p$-adic formal algebraic stack. It is characterized as a unique $\cO$-flat closed substack such that for any finite extension $\cO'/\cO$, its $\cO'$-points are exactly the subgroupoid of $\cX_n(\cO')$ consisting of potentially crystalline $G_K$-representations of Hodge type $\lam$ and inertial type $\tau$. We denote by $\cX^{\le\lam,\tau}$ the union $\cup_{\lam'\le \lam}\cX^{\lam',\tau}$.  %The stack $\cX^{\le \lam+\eta,\tau}_n$ is closely related to $Y^{\le\lam+\eta,\tau^\vee}$, which we explain now. 

%\begin{rmk}\label{rmk:HT-convention}
%    Recall that our convention on Hodge--Tate weights is opposite to that of \cite{LLLMlocalmodel}. In \loccit, $\cX^{\le\lam,\tau}$ denotes $\cX^{\le -w_0(\lam),\tau}$ in our notation. The automorphism of $\cX_n$ induced by taking dual of \'etale $\pgma$-modules (see \cite[\S3.8]{EGstack}) induces an isomorphism between $\cX^{\le\lam,\tau}$ in our notation and $\cX^{\le\lam,\tau^\vee}$ in the notation of \loccit. We often use this isomorphism to apply results in \loccit~to our setting. 
%\end{rmk}

There is a morphism $\veps_\infty: \cX_n \ra \Phi\dash\mathrm{Mod}_K^{\et,n}$ where $\Phi\dash\mathrm{Mod}_K^{\et,n}$ is the moduli stack of projective rank $n$ \'etale $\varphi$-modules (see \cite[Def.~3.7.1]{EGstack} and also \cite[\S5.4]{LLLMlocalmodel}), and the morphism is given by, at the level of $G_K$-representations, restricting the $G_K$-action to $G_{K_\infty}$ \cite[Prop.~3.7.2]{EGstack}. On the other hand, there is a morphism $\veps_{\tau}: Y^{\le\lam+\eta,\tau} \ra \Phi\dash\mathrm{Mod}_K^{\et,n}$ \cite[\S5.4.1]{LLLMlocalmodel}. We define $\cK^{\le\lam+\eta,\tau}$ by the following cartesian square
\[
\begin{tikzcd}
    \cK^{\le\lam+\eta,\tau} \arrow[r] \arrow[d] & Y^{\le\lam+\eta,\tau} \arrow[d, "\veps_{\tau}"] \\
    \cX_n^{\le\lam+\eta,\tau} \arrow[r, "\veps_\infty"] &   \Phi\dash\mathrm{Mod}_K^{\et,n}.
\end{tikzcd}
\]
As in \cite{LLLMlocalmodel}, we use the contravariant functor between \'etale $\varphi$-modules and $G_{K_\infty}$-representations. The next result allows us to consider $\cX_n^{\le\lam+\eta,\tau}$ as a closed substack of $Y^{\le\lam+\eta,\tau}$ for sufficiently generic $\tau$.

\begin{prop}\label{prop:BK-EG}
    If $\tau$ is $(eh_{\lam+\eta}+2)$-generic, then $\cK^{\le\lam+\eta,\tau} \simeq \cX_n^{\le\lam+\eta,\tau}$, and the morphism $\cK^{\le\lam+\eta,\tau} \ra Y^{\le\lam+\eta,\tau}$ is a closed immersion.
\end{prop}
\begin{proof}
    This can be proven as \cite[Prop.~7.2.3]{LLLMlocalmodel}. 
\end{proof}

 %By abuse of notation, we denote the pullback of $f\in \cO(Y^{\le\lam+\eta,\tau})$ along this closed embedding by $f$ again. 
Next, we prove an integral version of the globalization of \cite[Thm.~4.1]{6author}. For any de Rham representation $\rho: G_K \ra \GL_n(\cO)$, we denote by $\WD(\rho)$ the Frobenius-semisimplification of the Weil--Deligne representation associated to $\rho$ as in \cite[App.~B]{CDT} (in particular, $\rho \mapsto \WD(\rho)$ is covariant). If $\rho$ is potentially crystalline of Hodge type $\lam+\eta$ and tame inertial type $\tau$, we define the Breuil--Schneider's locally algebraic representation associated to $\rho$
\begin{align*}
    \mathrm{BS}(\rho) := r_p^\mo(\WD(\rho)) \otimes_\cO V(\lam)
\end{align*}
as in \cite[\S5]{6author}. %(note that this differs from the one constructed in \cite[\S4]{BS} by a unitary character). 
If $\fM$ is the image of $\rho$ in $Y^{\le\lam+\eta,\tau}(\cO)$, then $\WD(\rho) \simeq D(\fM)^\vee$.

\begin{cor}\label{cor:hecke-global-func}
    If $\tau$ is $(eh_{\lam+\eta}+2)$-generic, then there is an injective morphism
    \begin{align*}
        \Psi^{\lam,\tau}: \sig^\circ(\lam,\tau) \mono \cO(\cX_n^{\lam+\eta,\tau})
    \end{align*}
    which interpolates the local Langlands correspondence in the following sense: for a finite extension $\cO'/\cO$, if $\rho\in \cX_n^{\lam+\eta,\tau}(\cO')$, then the natural $\sig^\circ(\lam,\tau)$-action on $\Hom_{\rmK}(\sig^\circ(\lam,\tau), \mathrm{BS}(\rho))$ is given by the character
    \begin{align*}
        \sig^\circ(\lam,\tau) \xra{\Psi^{\lam,\tau}} \cO(\cX_n^{\lam+\eta,\tau}) \xra{\ev_{\rho}} \cO'.
    \end{align*}
\end{cor}
\begin{proof}
    This follows from Theorem \ref{thm:hecke-global-func} and Proposition \ref{prop:BK-EG} by taking pullback of global functions on $Y^{\le\lam+\eta,\tau}$ to $\cX_n^{\lam+\eta,\tau}$. The injectivity of $\Psi^{\lam,\tau}$ can be checked as in the proof of Theorem \ref{thm:hecke-global-func}. The key point is that %  using $\prod_{I\in\fP_\tau} \cX_{\#I}^{\le\lam_I+\eta_I,\tau_I} \ra \cX^{\le\lam+\eta,\tau}_n$ using the argument for Theorem \ref{thm:hecke-global-func}. The same argument works because 
    the composition
    \begin{align*}
        E[F_I^\pm]_{I\in \fP_\tau} \mono \otimes_{I\in \fP_\tau}\cO(Y^{\le\lam_I+\eta_I,\tau_I})[p^\mo] \ra \otimes_{I\in \fP_\tau}\cO(\cX_{\# I}^{\le\lam_I+\eta_I,\tau_I})[p^\mo]  
    \end{align*}
    remains injective. This can be checked by using unramified twists for Galois representations as in \loccit. % the value of Frobenius of the determinant of Galois representations corresponding to points of $\cX_{\# I}^{\le\lam_I+\eta_I,\tau_I}$.   %Here, the first map is defined in the proof of Theorem \ref{thm:hecke-global-func}. 
\end{proof}

Now we turn our attention to the reduced substack $\cX_{n,\red}\subset \cX_n$. It is a Noetherian algebraic stack over $\F$ equidimensional of dimension $\tfrac{n(n-1)}{2}[K:\Qp]$, and there is a bijection between its irreducible components and Serre weights of $\GL_n(k)$ \cite[Thm.~5.5.12 and 6.5.1]{EGstack}. For a Serre weight $\sigma$ of $\GL_n(k)$, we denote by $\cC_\sig$ the irreducible component of $\cX_{n,\red}$ corresponding to $\sig$ following the normalization of \cite[\S7.4]{LLLMlocalmodel}.

Let $\mu\in X_1^*(\uT)$ be a $p$-restricted cocharacter such that $\sig=F(\mu)$. The closed substack $\cC_\sig\subset \cX_{n,\red}$ is uniquely characterized as the closure of an open substack $\cU_\sig\subset \cX_{n,\red}$ whose $\Fpbar$-points are given by $G_K$-representations of the form
\begin{align}\label{eqn:ord-rhobar}
    \rhobar = \pma{\chi_1 & * & \cdots & * \\ 
     0 & \chi_2 & \cdots & * \\ 
     \text{\rotatebox{90}{$\cdots$}} &  & \text{\rotatebox{135}{$\cdots$}} &  \text{\rotatebox{90}{$\cdots$}} \\ 
     0 & \cdots & 0 & \chi_n}
\end{align}
where
\begin{enumerate}
    \item $\rhobar$ is maximally non-split of niveau $1$, i.e.~it admits a unique $G_K$-stable flag;
    \item For $i=1,\dots,n$, $\chi_i|_{I_K} = \prod_{\jj}\oom_{K,\sig_j}^{(\mu_j+\eta_j)_i}$;% for some $t_i\in \Fpbar^\times$;
    \item If $\chi_i\chi_{i+1}^\mo |_{I_K} = \ov{\veps}$, then $\mu_{j,i} - \mu_{j,i+1} = p-1$ for all $\jj$ if and only if  $\chi_i\chi_{i+1}^\mo= \ov{\veps}$ and the element of $\Ext^1_{G_K}(\chi_i,\chi_{i+1})$ determined by $\rhobar$ is tr\`es ramifi\'ee. Otherwise, $\mu_{j,i} - \mu_{j,i+1} = 0$ for all $\jj$.
\end{enumerate}

%\begin{rmk}\label{rmk:HT-convention2}
%    Note that our definition of $\cC_\sig$ follows a different normalization from that of \cite{LLLMlocalmodel}. The duality automorphism of $\cX_n$ induces an isomorphism between $\cC_\sig$ in our notation and $\cC_{\sig^\vee}$ in the notation of \loccit~where $\sig^\vee$ is the $\F$-dual of $\sig$.
%\end{rmk}

\subsection{The case of tori}\label{subsec:tori} 
Let $T\subset \GL_n$ be the diagonal torus. We define $\cX_T:= \prod_{i=1}^n \cX_{1}$. A Serre weight of $T(k) \simeq (k^\times)^n$ is given by a character
\begin{align*}
    \mu: (k^\times)^n \ra k^\times \xhookrightarrow{\sig_0} \F^\times.
\end{align*}
Let $\cC_\mu$ be the corresponding irreducible component in $\cX_{T,\red}$. 
%With applications to the case of $\GL_n$ in mind, we renormalize the bijection between Serre weights of $T(k)$ and irreducible components in $(\cX_{1,\red})^n$ so that $\mu$ corresponds to $\cC_{\mu+\eta}$. 
Its $\Fpbar$-points are given by characters of the form
\begin{align*}
    \omega_{K,\sig_0}^{\mu} \otimes (\ur_{x_1},\ur_{x_2},\dots,\ur_{x_n}) : G_K \ra T(\Fpbar)
\end{align*}
for some $x_1,x_2,\dots,x_n\in \Fpbar^\times$. We have a presentation of $\cC_\mu$ as a quotient stack $[(\G_m)^n/(\G_m)^n]$ where the first $(\G_m)^n$ parameterizes the image of the geometric Frobenius, and the second $(\G_m)^n$ acts trivially. 

The Hecke algebra $\cH(\mu)$ is isomorphic to $\F[x_1^\pm,x_2^\pm,\dots,x_n^\pm]$ where $x_i$ corresponds to the double coset operator $[\veps_i(\pi_K^\mo)T(\cO_K)]$. Note that if $\chi:T(K) \ra \Fpbar^\times$ is a character such that $\chi|_{T(\cO_K)} = \mu$, then $[\veps_i(\pi_K^\mo)T(\cO_K)]$ acts on $\Hom_{T(K)}(\ind_{T(\cO_K)}^{T(K)}\mu, \chi)$ by $\chi(\veps_i(\pi_K))$. Thus, local class field theory induces a natural isomorphism
\begin{align*}
    \cH(\mu) \risom \cO(\cC_\mu) \simeq \cO((\G_m)^n)
\end{align*}
sending $[\eps_i(\pi_K^\mo)T(\cO_K)]$ to the coordinate of $i$-th $\G_m$.

\subsection{The case of $\GL_n$}\label{subsec:GLn} Let $\sig$ be a non-Steinberg Serre weight of $\GL_n(k)$ with ($q$-restricted) highest weight $\mu\in X^*(T)$. Since $\cC_{\mu+\eta}$ is irreducible, the image of the natural map $\cC_{\mu+\eta} \ra \cX_{n,\red}$ is contained in some irreducible component, and it is easy to see that the image is contained in $\cC_\sig$. %We denote by $S_\sig: \cO(\cC_\sig) \ra \cO(\cC_{\mu+\eta})$ the morphism given by restriction.

\begin{lem}\label{lem:upperbound-prelim}
    The restriction map $\cO(\cC_\sig) \ra \cO(\cC_{\mu+\eta})$ is injective.
\end{lem}
\begin{proof}
    Since $\cU_\sig$ is dense in $\cC_\sig$, the restriction map $\cO(\cC_\sig) \ra \cO(\cU_\sig)$ is injective. Moreover, since $\cU_\sig$ is reduced, any function $f\in \cO(\cC_\sig)$ is determined by its value at $\rhobar \in \cU_\sig(\Fpbar)$. By \cite[Thm.~6.6.3 (3)]{EGstack}, any finite type point in the closure of $\CB{\rhobar}$ inside $\abs{\cX_{n,\red}}$ is a partial semisimplification of $\rhobar$, and $\rhobar^\ss$ is the unique closed point in the closure. By continuity of $f$, we have $f(\rhobar) = f(\rhobar^\ss)$ for any $\rhobar \in \cU_\sig(\Fpbar)$. Since the image of $\cC_{\mu+\eta}$ in $\cC_\sig$ contains $\rhobar^\ss$ for all $\rhobar\in \cU_\sig(\Fpbar)$, $f$ is determined by its values at the images of $\rhobar \in \cC_{\mu+\eta}(\Fpbar)$.
\end{proof}

We have a morphism $(-)\otimes \ov{\veps}^{\eta}: \cC_\mu \ra \cC_{\mu+\eta}$ given by twisting $\ov{\veps}^{\eta}$. It is easily seen to be an isomorphism. Using this, we define the following Galois-theoretic analogue of the mod $p$ Satake transform $\cS_\sig: \cH(\sig) \ra \cH(\mu)$.

\begin{defn}\label{defn:spec-satake-transform}
    We define a morphism $S_\sig: \cO(\cC_{\sig}) \ra \cO(\cC_{\mu})$ to be the pullback along the composition of $(-)\otimes \ov{\veps}^{\eta}$ and the natural morphism $\cC_{\mu+\eta} \ra \cC_\sig$. 
\end{defn}
%Since $\cU_\sig$ is dense in $\cC_\sig$, the restriction map $\cO(\cC_\sig)\ra \cO(\cU_\sig)$ is injective. On the other hand, we have the pullback map $\ss^*: \cO(\cC_{\mu+\eta}) \ra \cO(\cU_\sig)$. If $\sig$ is non-Steinberg, then $\ss$ has a dense image. Otherwise, the image of $\ss$ has a positive codimension. Thus, $\ss^*$ is injective when $\sig$ is non-Steinberg, but it is never injective if $\sig$ is Steinberg. Moreover, from ..., $\ss^*$ is subjective.

%The above discussion also shows that if $\rhobar\in \cU_\sig(\Fpbar)$, then for any $f\in \cO(\cC_\sig)$, $f(\rhobar)$ is equal to $f(\rhobar^\ss)$. Indeed, when $\sig$ is non-Steinberg, the image of $\cC_{\mu+\eta} \ra \cX_{n,\red}$ is contained in $\cC_\sig$. Thus, we get the restriction map $\cO(\cC_\sig)\ra \cO(\cC_{\mu+\eta})$ which coincides with the composition of $(\ss^*)^\mo$ and the restriction map $\cO(\cC_\sig)\mono \cO(\cU_\sig)$.

\begin{prop}\label{prop:upperbound}
    For non-Steinberg $\sig$, $S_\sig: \cO(\cC_\sig) \ra \cO(\cC_{\mu})$ is an injection with image contained in $\ov{\Psi}_\mu\circ \cS_\sig(\cH(\sig))$. In particular, we have a morphism $\ov{\Phi}_\sig : \cO(\cC_\sig)\mono \cH(\sig)$ that fits in the following commutative diagram
    \[
    \begin{tikzcd}
        \cH(\sig) \arrow[r, hook, "\cS_\sig"] & \cH(\mu)  \\
        \cO(\cC_\sig) \arrow[r, hook, "S_\sig"]  \arrow[u, "\ov{\Phi}_\sig"]& \cO(\cC_{\mu}) \arrow[u, "\ov{\Psi}_{\mu}^\mo"].
    \end{tikzcd}
    \]
    %where $\ov{\Phi}$ is induced by the inverse of the bottom isomorphism.
\end{prop}

This essentially follows from the case of $\GL_2$, which requires some detailed discussions on \cite{cegsC}. So we postpone its proof to \S\ref{subsec:GL2} and take it for granted for now.

Our main result shows that $\ov{\Phi}_\sig$ is indeed an isomorphism for sufficiently generic $\sig$. Let $(\tilw,\om)$ be a lowest alcove presentation of $\sig$ and define $\ud{\mu}$, $\mu$, and $\tau=\tau(e,\mu)$ as in Example \ref{ex:hecke-modp}. By \cite[Lem.~6.2.8]{LLLMlocalmodel}, if $\sig$ is $(n-1)$-deep, $\cC_\sig$ is contained in $\cX^{\eta,\tau}_n$. We denote the restriction map $\cO(\cX^{\eta,\tau}_n) \ra \cO(\cC_\sig)$ by $R^\tau_\sig$. %For the following theorem, let $\tau:= \tau(e,\mu)$ be a principal series tame inertial type and $\sig^\circ(\tau)$ be the $\cO$-lattice in $\sig(\tau)$ as in \S\ref{subsec:int-hecke-alg}. By \cite[Lem.~6.2.8]{LLLMlocalmodel}, if $\sig$ is $(n-1)$-deep, $\cC_\sig$ is contained in $\cX^{\eta,\tau}_n$. We denote the restriction map $\cO(\cX^{\eta,\tau}_n) \ra \cO(\cC_\sig)$ by $R^\tau_\sig$.

\begin{thm}\label{thm:spec-satake}
    Suppose that $\sig$ is $((e+1)(n-1)+2)$-deep. Then the morphism $\ov{\Phi}_\sig: \cO(\cC_\sig) \mono \cH(\sig) $ is an isomorphism. Moreover, if we write $\ov{\Psi}_\sig = \ov{\Phi}_\sig^\mo$, it fits in the following commutative diagram
    \[
    \begin{tikzcd}
        \cH(\sig^\circ(\tau)) \arrow[r, two heads, "\cR^{\sig^\circ(\tau)}_\sig"] \arrow[d, "\Psi^{\tau}"] & \cH(\sig) \arrow[r, hook, "\cS_\sig"] \arrow[d, "\ov{\Psi}_\sig"] & \cH(\mu) \arrow[d, "\ov{\Psi}_{\mu}"] \\
        \cO(\cX^{\eta,\tau}) \arrow[r, two heads, "R^{\tau}_\sig"] & \cO(\cC_\sig) \arrow[r, hook, "S_\sig"] & \cO(\cC_{\mu}).
    \end{tikzcd}
    \]
    %where $R_{\tau}^\sig$ is given by restriction\comment{explain why $\cC_\sig \subset \cX^\tau$}, and the middle vertical arrow $\ov{\Psi}_\sig$ is the inverse of $\ov{\Phi}_\sig$.
\end{thm}

\begin{proof}
    Since $\sig$ is $((e+1)(n-1)+2)$-deep, $\tau$ is $(e(n-1)+2)$-generic. Therefore, we can apply Corollary \ref{cor:hecke-global-func} to obtain $\Psi^{\tau}$. Once we prove the commutativity of the largest rectangle, it implies the existence of $\ov{\Psi}_\sig$ making all squares commutative, which has to be an isomorphism by Proposition \ref{prop:upperbound}. Take
    \begin{align*}
        \rhobar' = \oplus_{i=1}^n \chi_i \in \cC_{\mu}(\Fpbar), \ \ \chi_i|_{I_K} = \prod_{\jj}\oom_{K,\sig_j}^{\mu_{j,i}}.
    \end{align*}
    We denote by $t_i$ the image $\chi_i(\Frob_K)\in \Fpbar^\times$. 
    Let $\fI\subset \fP_\tau$ be a subset of size $d$. To compute the image of $T_{\fI}\in \cH(\sig^\circ(\tau))$ in $\cO(\cC_{\mu})$, we can lift $\rhobar =\rhobar'\otimes \ov{\veps}^{\eta}$ to $\rho \in \cX^{\eta,\tau}_n(\Zpbar)$ and evaluate at $\til{F}_{\fI}$. We take $ \rho = \rho'\otimes \veps^{\eta}$ where
    \begin{align*}
        \rho' &= \oplus_{i=1}^n \til{\chi}_i, \ \ \til{\chi_i}|_{I_K} = \prod_{\jj}\om_{K,\sig_j}^{\mu_{j,i}}
    \end{align*}
    and let $\til{t}_i$ be $\til{\chi}_i(\Frob_K)\in \Zpbar^\times$. 
    By an elementary computation, we obtain
    \begin{align*}
        \WD(\rho) = \oplus_{i=1}^n \phi_i, \ \ \phi_i|_{I_K} = \om_{K,\sig_j}^{\mu_{j,i}}
    \end{align*}
    and $\phi_i(\Frob_K)=q^{i-n}\til{t}_i$. Noting that $\sig(\tau)\simeq \Ind_{B(k)}^{G(k)}[w_{f-1}^\mo(\mu)]$, we can deduce that
    \begin{align*}
        \til{F}_{w_{f-1}^\mo(\fI)}(\rho) =  q^{d(2n-d-1)/2} \prod_{i\in \fI}q^{i-n}\til{t}_i %= \prod_{i\in \fI}\til{t}_i
    \end{align*}
    which is zero modulo the maximal ideal of $\Zpbar$ unless $\veps_{-\fI}$ is antidominant, in which case is equal to $\prod_{i\in \fI} t_i$. Together with the discussion in \S\ref{subsec:tori}, this proves the commutativity of the largest square and thus completes the proof.
\end{proof}

\begin{defn}
    Suppose that $\sig$ is $((e+1)(n-1)+2)$-deep. For $i=1,\dots, n$, we define $\ov{f}_i\in \cO(\cC_\sig)$ to be the function $R^\tau_\sig(\til{F}_{\fI_i})$ where $\fI_i = w_{f-1}^\mo(\CB{1,2,\dots,i})$. Note that $\ov{f}_i$ is equal to $\ov{\Psi}_\sig(\ov{T}_{-\om_i})$. 
\end{defn}

\begin{example}\label{ex:fail}
    We remark that Theorem \ref{thm:spec-satake} may hold under a somewhat milder deepness condition on $\sig$, but it fails to hold in general. %Let $\mu$ be the highest weight of $\sig$. 
    \begin{enumerate}
        \item Suppose that $\sig$ is Steinberg and let $\bfn_0\subset \bfn$ be the subset of integers $i$ such that $\mu_{j,i} - \mu_{j,i+1} = p-1$ for all $\jj$. Following the proof of Lemma \ref{lem:upperbound-prelim} using the eigenvalue morphism (see \cite[Def.~5.5.11]{EGstack})
        \begin{align*}
            \nu: \cU_\sig \ra (\G_m)^{n-\#\bfn_0}, 
        \end{align*}
        we can show that $\cO(\cC_\sig)$ injects into $\F[x_1^\pm,\dots, x_{n-\#\bfn_0}^\pm]$ and thus cannot be isomorphic to $\cH(\sig)$. 
        \item Suppose that there is an integer $1\le i \le n-1$ such that $\mu_{j,i} - \mu_{j,i+1} = p-2$ for all $\jj$. Then for $\rhobar \in \cU_\sig(\Fpbar)$ of the form \eqref{eqn:ord-rhobar}, its $[i,i+1]$-diagonal block has the form  
    \begin{align*}
        \pma{\oom_{K,\sig_0}^a \ur_{t_i} & * \\ & \oom_{K,\sig_0}^a \ur_{t_{i+1}}}
    \end{align*}
    for some fixed integer $a$. If $\cH(\sig)\simeq \cO(\cC_\sig)$ as in Theorem \ref{thm:spec-satake}, then the function $\ov{f}_i$ maps $\rhobar$ to $t_1 \cdots t_i$ (here, we assume $\oom_{K,\sig_0}(\Frob_K)=1$ for notational convenience). We can also find $\rhobar' \in \cU_\sig(\Fpbar)$ whose diagonal characters are exactly the diagonal characters of $\rhobar$ except that $t_i$ and $t_{i+1}$ are interchanged. Then $\ov{f}_i(\rhobar') = \ov{f}_i(\rhobar)t_{i+1}t_i^\mo$. However, since $\rhobar^\ss = \rhobar^{\prime\ss}$ and any function in $\cO(\cC_\sig)$ does not change its value under semisimplification (this is because $\rhobar^\ss$ is in the closure of $\rhobar$ by \cite[Thm.~6.6.3 (3)]{EGstack}), we get a contradiction. We acknowledge that before finding this argument, we learned from Kalyani Kansal and Ben Savoie that $\cH(\sig)$ and $\cO(\cC_\sig)$ cannot be isomorphic for $\sig$ in this example with $n=2$ and $K=\Qp$ using a different argument.
    \end{enumerate}
\end{example}

\begin{rmk}\label{rmk:fail}
    Let $\sig=F(\mu)$ be a non-Steinberg Serre weight and $\tau=\tau(1,\mu)$ as in Theorem \ref{thm:spec-satake}. If $\cX_n^{\eta,\tau}$ is analytically normal, then we can obtain the morphism $\Psi^\tau : \cH(\sig^\circ(\tau)) \ra \cO(\cX_n^{\eta,\tau})$. Assuming that $\cC_\sig$ is contained in $\cX_n^{\eta,\tau}$, by the proof of Theorem \ref{thm:spec-satake}, $\Psi^\tau$ induces the isomorphism $\ov{\Psi}_\sig: \cH(\sig) \risom \cO(\cC_\sig)$. Therefore, for $\mu$ in Example \ref{ex:fail}, $\cX_n^{\eta,\tau}$ cannot be normal and the image of $\cH(\sig^\circ(\tau))$ in $\cO(\cX_n^{\eta,\tau})[1/p]$ is not contained in $\cO(\cX_n^{\eta,\tau})$. In a very recent paper \cite{LHMM}, Le Hung--M\'ezard--Morra show that when exactly $\cX_2^{\eta,\tau}$ fails to be normal under the assumption that $p\ge 7$ and $K/\Qp$ is unramified, or $K=\Q_5$. In particular, their main result confirms that when $\mu\ix{j}_1 - \mu\ix{j}_2 = p-2$ for all $\jj$, $\cX^{\eta,\tau}_2$ is not normal. %On the other hand, if we define $\sig'=F(\nu)$ with $\nu\ix{j}_1-\nu\ix{j}_2 = 1$ for all $\jj$, then the main result of \cite{GKKSW} shows that $\cO(\cC_{\sig'})\simeq \cH(\sig')$ and yet, $\cX_2^{\eta,\tau(1,\nu)}$ is \textit{not} normal because it is isomorphic to $\cX_2^{\eta,\tau(1,\mu)}$ by twisting by a character. 
\end{rmk}

For any irreducible component $\cC_\sig\subset \cX^{\lam+\eta,\tau}_n$, we denote by $R^{\lam,\tau}_\sig:\cO(\cX^{\lam+\eta,\tau}_n) \ra \cO(\cC_\sig)$ the morphism given by restriction (note that $R^{\lam+\tau}_{\sig}$ was defined earlier only for $\lam=0$ and specific $\tau$ and $\sig$). We finish this subsection by discussing the relationship between $\cR^{\sig^\circ(\tau)}_\sig$ and $R^\tau_\sig$ (the case $\lam=0$). We first state the following result which generalizes \cite[Prop.~4.5.2]{EGHweightcyc} to potentially crystalline representations. %Note that $\varphi^{[F_w:\Qp]}$ in the statement of \loccit~should be replaced by $\varphi^{-[F_w^0:\Qp]}$ (also see \cite{EGHerratum}).

\begin{prop}\label{prop:reducibility}
    Let $\rho\in \cX^{\lam+\eta,\tau}_n(\cO)$ be a potentially crystalline representation of Hodge type $\lam+\eta$ and tame inertial type $\tau$. Let $\WD(\rho) = \oplus_{I\in\fP_\tau} \phi_I$ where $\phi_I |_{I_K} \simeq \tau_I$.
    %and $\WD(\rho)(\Frob_K)$ has a characteristic polynomial 
    %\begin{align*}
    %    X^n + \cdots + (-1)^i q^{\sum_{\tilj\in \tilcJ} \RG{w_0(\omega_{i}\ix{\tilj}),(\lam+\eta)\ix{\tilj}}}t_i X^i +\cdots (-1)^n   q^{\sum_{\tilj\in \tilcJ} \RG{\ud{1},(\lam+\eta)\ix{\tilj}}}t_n.
    %\end{align*}
    Suppose that there exists a pair $(\ov{\fI}, d_{\ov{\fI}})$ of $\ov{\fI}\subset \ov{\fP}_\tau$ and $d_{\ov{\fI}}=(d_{\ov{I}})_{\ov{I}\in \ov{\fI}}$ with $d_{\ov{I}}\in [1,\#\ov{I}]$ such that 
    \begin{align*}
      \pi_K^{\sum_{\tiljj} \RG{\om_{n(\ov{\fI}, d_{\ov{\fI}})}, \lam_{\tilj}}}q^{ \RG{\om_{n(\ov{\fI}, d_{\ov{\fI}})}, \eta}}\prod_{\ov{I}\in \ov{\fI}} \sym_{d_{\ov{I}}} (\det (\phi_I)(\Frob_K))_{I\in \ov{I}} \in \cO^\times.
    \end{align*}
    Then $\rho$ contains a $m$-dimensional subrepresentation $\rho'\subset \rho$ such that $\rho' \in \cX_{m}^{\lam'+\eta_m+\ud{n-m},\tau'}(\cO)$ where
    \begin{align*}
        m := n(\ov{\fI},d_{\ov{\fI}}), \ \        \lam' := (\lam_{\tilj,1},\dots, \lam_{\tilj,m})_{\tilj\in\tilcJ}, \ \ \text{and } \ 
        \tau' := \oplus_{\ov{I}\in \ov{\fI}}(\tau_I)^{\oplus d_{\ov{I}}}.
    \end{align*}
    (In the definition of $\tau'$, for each $\ov{I}\in \ov{\fI}$, choose one $I\in \ov{I}$ to take the direct sum.)
\end{prop}

\begin{proof}
    We briefly explain how the proof of \cite[Prop.~4.5.2]{EGHweightcyc} carries over verbatim (see also, \cite{EGHerratum}). Since $\varphi$ and $\Del'$-action on $D:=\Dcris(\rho)$ commute, we can decompose $D$ into $D_{\tau_{\ov{I}}}$ its $\tau_{\ov{I}}$-isotypic parts for $\ov{I}\in \ov{\fP}_\tau$. For each $\ov{I}\in \ov{\fI}$, we can find a $\varphi$-stable $K_0\otimes_{\Qp}E$-subspace $D'_{\tau_I} \subset D_{\tau_I}$ of rank $d_{\ov{I}}n_{\ov{I}}$ spanned by $d_{\ov{I}}n_{\ov{I}}$ $\varphi^{f}$-eigenvectors of the $d_{\ov{I}}n_{\ov{I}}$ smallest slopes. Then the argument in \loccit~shows that $D':= \oplus_{\ov{I}\in \ov{\fI}}D'_{\tau_I}$ is weakly admissible, and $\rho'$ is given by $\Vcris(D')$.     
\end{proof}

\begin{rmk}
    If $\tau$ is $(eh_{\lam+\eta}+2)$-generic, then the condition
    \begin{align*}
      \pi_K^{\sum_{\tiljj} \RG{\om_{n(\ov{\fI}, d_{\ov{\fI}})}, \lam_{\tilj}}}q^{ \RG{\om_{n(\ov{\fI}, d_{\ov{\fI}})}, \eta}}\prod_{\ov{I}\in \ov{\fI}} \sym_{d_{\ov{I}}} (\det (\phi_I)(\Frob_K))_{I\in \ov{I}} \in \cO^\times.
    \end{align*}
    in Proposition \ref{prop:reducibility} is equivalent to $\til{F}_{\ov{\fI},d_{\ov{\fI}}}(\rho) \in \cO^\times$. We used the previous one so that we can state the proposition without genericity condition on $\tau$.
\end{rmk}

\begin{prop}\label{prop:glob-func-restriction}
    Let $\tau$ be $(e(n-1)+2)$-generic. For any $\cC_\sig \subset \cX^{\eta,\tau}_n$ and ${\fI}\subset {\fP}_\tau$, $R^\tau_\sig(\til{F}_{\fI}) = 0$ unless $\cC_\sig$ intersects with the image of $\cX^{\eta_m+\ud{n-m},\tau'}_m \times \cX^{\eta_{n-m},\tau''}_{n-m} \ra \cX_n^{\eta,\tau}$ where $m$, $\tau'$ are as in Proposition \ref{prop:reducibility} and $\tau''=\oplus_{I \notin \fI}\tau_I$, in which case $R_{\sig}^\tau (\til{F}_{\fI}) = \ov{f}_m$. % are defined similarly using $\fI'' = \bfn\backslash\fI$.% m := n(\ov{\fI},d_{\ov{\fI}})$, $\tau':= \oplus_{I\in \fI}\tau_I$, and $\tau'':= \oplus_{I\notin \fI}\tau_I$, in which case  $R^\tau_\sig(\til{F}_{\fI})$ equals $\ov{f}_{m}\in \cO(\cC_\sig)$.
\end{prop}

\begin{proof}
    Note that $\tau$ is regular by the genericity condition. If $R^{\tau}_\sig(\til{F}_{\fI}) \neq 0$, there is $\rhobar \in \cU_\sig(\F)$ at which $R^{\tau}_\sig(\til{F}_{\fI})$ takes a non-zero value. This implies that $\rhobar$ has a lift $\rho \in \cX^{\eta,\tau}_n(\cO)$ such that $\til{F}_{\fI}(\rho)\in \cO^\times$. By Proposition \ref{prop:reducibility}, $\rho$ has a subrepresentation $\rho' \in \cX^{\eta_m+\ud{n-m},\tau'}_m(\cO)$ and $\rho'':=\rho/\rho' \in \cX^{\eta_{n-m},\tau''}_{n-m}(\cO)$. Therefore, the unique $m$-dimensional subrepresentation $\rhobar'$ of $\rhobar$ is contained in $\cX^{\eta_m+\ud{n-m},\tau'}_m(\F)$ and similarly $\rhobar'':=\rhobar/\rhobar'$ is in $\cX^{\eta_{n-m},\tau''}_{n-m}(\F)$. This shows that $\cC_\sig$ intersects with the image of $\cX^{\eta_m+\ud{n-m},\tau'}_m \times \cX^{\eta_{n-m},\tau''}_{n-m}\ra \cX_n^{\eta,\tau}$. The claim that $R^{\tau}_\sig(\til{F}_{\fI})=\ov{f}_{m}$ follows from the proof of Theorem \ref{thm:spec-satake}.% shows that this is congruent to $\det(\rhobar')(\Frob_K)$ modulo $\varpi$.
\end{proof}

\begin{cor}\label{cor:spec-hecke-modp}
    Suppose that $\tau$ is $(e(n-1)+2)$-generic. Let $\sig \in \JH_c(\ov{\sig}(\tau))$ be a Serre weight such that $\sig^{N(k)}\simeq \oplus_{I\in \fP_\tau} \sig_I$ for some $\sig_I \in \JH(\osig(\tau_I))$. We suppose that $\cC_\sig \subset \cX^{\eta,\tau}_n$ and $\cC_{\sig_I}\subset \cX^{\eta,\tau_I}_{\#I}$. Then we have the following commutative diagram
    \[
    \begin{tikzcd}
        \cH(\sig^\circ(\tau)) \arrow[r, "\Psi^{\tau}"] \arrow[d, "\cR^{\sig^\circ(\tau)}_\sig"] & \cO(\cX^{\eta,\tau}) \arrow[d, "R^{\tau}_\sig"] \\
        \cH(\sig)  \arrow[r, "\ov{\Psi}_{\sig}"] & \cO(\cC_\sig).
    \end{tikzcd}
    \]
\end{cor}
\begin{proof}
    This follows from Proposition \ref{prop:glob-func-restriction} and Theorem \ref{thm:hecke-modp-red}.
\end{proof}

\begin{rmk}
    For $K/\Qp$ unramified and $\tau$ is $(5n-1)$-generic, then it is known that $\cC_\sig\subset \cX^{\eta,\tau}_n$ if and only if $\sig \in \JH(\osig(\tau))$ (\cite[Thm.~7.4.2]{LLLMlocalmodel}, also see Rmk.~7.4.3 (3)). Thus, the assumption on $\cC_{\sig}$ and $\cC_{\sig_I}$ above are natural.
\end{rmk}

\subsection{The case of $\GL_2$}\label{subsec:GL2} In this subsection, we provide the proof of Proposition \ref{prop:upperbound} and thus finishing the proof of Theorem \ref{thm:spec-satake}. 

Let $n=2$. In \cite{cegsC}, the authors construct $Y^{\eta,\tau}$ (denoted by $\cC^{\tau^\vee,\mathrm{BT}}$ in \loccit; the presence of dual is due to the discrepancy between Definition \ref{defn:BKmod} and \cite[Def.~2.3.2]{cegsC}) the moduli stack of rank two projective Breuil--Kisin modules with a tame descent data $\tau$ and a certain Kottwitz-type determinant condition. %We put $\tau^\vee$ instead of $\tau$ due to the discrepancy between Definition \ref{defn:BKmod} and \cite[Def.~2.3.2]{cegsC}. 
Our stack $Y^{\le\eta,\tau}$ contains $Y^{\eta,\tau}$ as a closed substack. They also construct the moduli stack of two dimensional tamely potentially Barsotti--Tate representations of $G_K$ with tame inertial type $\tau$. It is originally constructed as a certain closed substack of the moduli stack of \'etale $\varphi$-modules with descent data, and it is known to be isomorphic to $\cX^{\eta,\tau}_2$ by \cite[Thm.~1.4]{APAW}. By its construction, $\cX^{\eta,\tau}_2$ admits a partial resolution $Y^{\eta,\tau} \ra \cX^{\eta,\tau}_2$. %If $\tau$ is $(e+2)$-generic, this coincides with the isomorphism in Proposition \ref{prop:BK-EG}.

Let $\sig=F(\mu)$ be a non-Steinberg Serre weight of $\GL_2(k)$. We take $\tau= \tau(1,\mu)\simeq \chi_1 \oplus \chi_2$ and $\sig^\circ(\tau)$ as in Theorem \ref{thm:spec-satake}. We take $\tau^\vee = \chi_2^\mo \oplus \chi_1^\mo$ (note that the choice of ordering matters for explicit analysis). By taking $J=\emptyset$ in \cite[Thm.~1.1]{cegsC}, we obtain an irreducible component $Z_\sig\subset Y^{\eta,\tau}\times_{\cO}\F$ with a scheme-theoretically dominant morphism $Z_\sig\ra \cC_\sig$. (This is where we use the condition that $\sig$ is non-Steinberg.) Note that this implies $\cC_\sig \subset \cX^{\eta,\tau}$. Moreover, there is a dense open substack $U_\sig \subset Z_\sig$ such that the map  $Z_\sig\ra \cC_\sig$ restricts to an open immersion on $U_\sig$. Similar to $\cU_\sig \subset \cC_\sig$, $U_\sig$ has the following properties: any $\ov{\fM}\in U_\sig(\Fpbar)$ is an extension of $\ov{\fN}_2$ by $\ov{\fN}_1$ where for $\ov{\fN}_1$ and $\ov{\fN}_@$ are rank 1 Breuil--Kisin modules with tame descent data $\chi_2^\mo$ and $\chi_1^\mo$, respectively. Moreover, we can choose an eigenbasis $\ov{\be}_i= \CB{\ov{f}_i\ix{j}}_{\jj}$ of $\ov{\fN}_i$ for $i=1,2$ so that $\phi_{\ov{\fN}_i}$ maps $1\otimes \ov{f}_i\ix{j-1}$ to $\ov{a}_1\ix{j}\ov{f}_1\ix{j}$ if $i=1$ and to $\ov{a}_2\ix{j}u^{(p^f-1)e}\ov{f}_2\ix{j}$ if $i=2$ for some $\ov{a}_i\ix{j}\in \Fpbar^\times$. 

Recall from Lemma \ref{lem:global-functions/p}, we have global functions $f_{1}:= qf_{\{1\}}, f_2:= qf_{\{2\}}, f_3:=f_1f_2/q \in \cO(Y^{\eta,\tau})$ if $\chi_1 \not\simeq \chi_2$ and $f_2:=qf_{\{1,2\},1}, f_3:=qf_{\{1,2\},2} \in \cO(Y^{\eta,\tau})$ if $\chi_1\simeq \chi_2$.

\begin{lem}\label{lem:BK-func-ord}
    Following the notations above, if $\chi_1 \not\simeq \chi_2$, the restriction of functions $f_{1},f_2, f_3 \in \cO(Y^{\eta,\tau})$ to $U_\sig$ maps $\ov{\fM} \in U_\sig(\Fpbar)$ to $0$, $\prod_{\jj}(\ov{a}_2\ix{j})^\mo$, and $\prod_{\jj}(\ov{a}_1\ix{j}\ov{a}_2\ix{j})^\mo$, respectively. If $\chi_1 \simeq \chi_2$, the restriction of functions $f_2,f_3 \in \cO(Y^{\eta,\tau})$ to $U_\sig$ maps $\ov{\fM} \in U_\sig(\Fpbar)$ to $\prod_{\jj}(\ov{a}_2\ix{j})^\mo$ and $\prod_{\jj}(\ov{a}_1\ix{j}\ov{a}_2\ix{j})^\mo$, respectively.
\end{lem}
\begin{proof}
    For $i=1,2$, we can lift $\ov{\fN}_i$ to a rank 1 Breuil--Kisin module $\fN_i$ over $\cO$ with tame descent data $\chi_i$ and an eigenbasis $\be_i= \CB{f_i\ix{j}}_{\jj}$ such that $\phi_{\fN_i}$ maps $1\otimes f\ix{j-1}$ to $a_1\ix{j}f_1\ix{j}$ if $i=1$ and to $a_2\ix{j}E_j(u)f_2\ix{j}$ if $i=2$. By \cite[Cor.~3.1.7]{cegsC}, this allows as to lift $\ov{\fM}$ to $\fM\in Y^{\eta,\tau}(\cO)$. Then the claim follows from the definition of the global functions.  
\end{proof}

%For simplicity, we write $f'_1$ and $f'_2$ for $f_1$ and $f_1f_2/q$ if $\chi_1\not\simeq \chi_2$ and for $f_{\{1,2\},1}$ and $f_{\{1,2\},2}/q$ if $\chi_1\simeq \chi_2$. 
We denote by $\ov{f_2}$ and $\ov{f}_3$ the restriction of $f_2$ and $f_3$ to $Z_\sig$, respectively. %  If $\chi_1\simeq \chi_2$, we denote by $\ov{f}_1$ and $\ov{f}_2$ the restriction of $f_{\{1,2\},1}$ and $f_{\{1,2\},2}/q$ to $Z_\sig$, respectively. 

\begin{proof}[Proof of Proposition \ref{prop:upperbound}]
     There is nothing to prove if $n=1$. Suppose that $n=2$. Recall from \S\ref{subsec:tori} that $\cO(\cC_\mu)\simeq \cH(\mu) \simeq \F[x_1^\pm,x_2^\pm]$. Since $\F[x_1^\pm,x_2^\pm] = \F[(x_1x_2)^\pm][x_1^\pm]$, we can write any $g\in \F[x_1^\pm ,x_2^\pm]$ as a Laurent polynomial in variable $x_1$ and coefficients in $\F[(x_1x_2)^\pm]$. Let $g\in \cO(\cC_\sig)$ be a function whose restriction to $\cU_\sig$ is given by a polynomial $G(x_1x_2,x_1)\in \F[(x_1x_2)^\pm][x_1^\pm]$. We suppose that $G(x_1x_2,x_1)\notin \F[(x_1x_2)^\pm][x_1]$. Let $m$ be the smallest positive integer such that $x_1^m G(x_1x_2,x_1)\in \F[(x_1x_2)^\pm][x_1]$ and $F(x_1x_2)\in \F[(x_1x_2)^\pm]$ be its constant term. Let $g'$ be the pullback of $g$ to $Z_\sig$. By Lemma \ref{lem:BK-func-ord} and \cite[Lem.~4.1.4]{cegsC}, it is easy to see that $g'=G(\ov{f}_3,\ov{f}_2)$. Let $\til{G}(x_1x_2,x_1)\in \cO[(x_1x_2)^\pm][x_1^\pm]$ and $\til{F}(x_1x_2)\in \cO[(x_1x_2)^\pm]$ be the Teichm\"uller lifts of $G$ and $F$ respectively. 

     Let $\ov{\fM}\in Z_\sig(\F)$ be a Breuil--Kisin module whose image in $\cC_\sig(\F)$ is an irreducible $G_K$-representation. Then we can lift $\ov{\fM}$ to an $\cO$-point $\fM$. Note that $f_3$ is invertible and thus $f_3(\fM) \in \cO^\times$. Possibly after twisting $\ov{\fM}$ and $\fM$ by an unramified character (using \cite[Def.~3.3.2]{cegsC}), we can assume that $\til{F}(f_3)(\fM)\in \cO^\times$. We claim that $f_2(\fM)\equiv 0 \mod\varpi$. Suppose $f_2(\fM)\in \cO^\times$. %Since $f'_2$ is invertible, we must have $F_2(\fM)\in q\cO^\times$. Then $F_1(\fM) + F_2(\fM) \in \cO^\times$. 
     Let $\rho:G_K \ra \GL_2(\cO)$ be the image of $\fM$ in $\cX_2^{\eta,\tau}(\cO)$. By Proposition \ref{prop:glob-func-WD} and \ref{prop:reducibility}, $\rho$ is reducible. This contradicts the assumption that the image of $\ov{\fM}$ in $\cC_\sig(\F)$ is irreducible. Therefore, $f_2(\fM)\equiv 0 \mod\varpi$. This implies that
     \begin{align*}
         (f_2)^m\til{G}(f_3,f_2)(\fM) \equiv \til{F}(f_3)(\fM) \equiv 0 \mod \varpi.
     \end{align*}
    This contradicts $\til{F}(f_3)(\fM)\in \cO^\times$ and thus completes the case $n=2$.

     Now we prove the general case. For each $1\le i \le n-1$, we define $\sig_i := F((\mu_{i}+n-i-1,\mu_{i+1}+n-i-1))$ a non-Steinberg Serre weight of $\GL_2(k)$. Then there are natural morphisms
     \begin{align*}
         \cC_{\mu+\eta} \ra \cC_{\sig_i} \times \prod_{k \neq i,i+1} \cC_{\mu_{k}+n-k} \ra \cC_\sig
     \end{align*}
     with injective pullbacks
     \begin{align*}
         \cO(\cC_{\sig}) \mono \cO( \cC_{\sig_i} \times \prod_{k \neq i,i+1} \cC_{\mu_{k}+n-k}) \mono \cO(\cC_{\mu+\eta})\simeq \cO(\cC_\mu).
     \end{align*}
     By the $n=2$ case, the image of the second morphism is contained in
     \begin{align*}
         \F[x_1^\pm,\dots,x_{i-1}^\pm, x_i, (x_ix_{i+1})^\pm, x_{i+2}^\pm, \dots, x_n^\pm].
     \end{align*}
     Then the claim follows by taking the intersection of these images for all $i$.
\end{proof}

%
%\begin{rmk}
%    The eigenvalue morphism $e_\sig$ can be upgraded into the ``semisimplification'' morphism. Define a morphism
%\begin{align*}
%    (\G_m)^n &\ra \cC_{\mu+\eta}\simeq [(\G_m)^n/(\G_m)^n] \\
%    (t_1,\dots,t_n) &\mapsto (\ur_{t_1}\prod_{\jj}\oom_{K,\sig_j}^{(\mu\ix{j}+\eta\ix{j})_1},\dots,\ur_{t_n}\prod_{\jj}\oom_{K,\sig_j}^{(\mu\ix{j}+\eta\ix{j})_n}).
%\end{align*}
%The composition of this morphism and $e_\sig$ is the map $\ss: \cU_\sig \ra \cC_{\mu+\eta}$ which maps $\rhobar\in \cU_\sig(\Fpbar)$ to its semisimplification $\rhobar^\ss$. The image of $\ss$ is dense if $\sig$ is non-Steinberg, but otherwise has a positive codimension.
%\end{rmk}

\section{Parabolic loci}\label{sec6}
In this section, we use the theory of local models for tame potentially crystalline Emerton--Gee stacks to study certain open loci in $\cX^{\eta,\tau}_n$ which admits a natural ``parabolic structure'' (Theorem \ref{thm:reducibility-in-families}). More precisely, any point in such a locus admits a natural filtration corresponding to a parabolic subgroup 
 of $\GL_n$, and the filtrations vary continuously. When $\sig=F(\mu)$ and $\tau=\tau(1,\mu)$ is a principal series, by taking the intersection of $\cX^{\eta,\tau}$ and $\cC_{\sig}$, we get parabolic loci inside $\cC_{\sig}$ which generalize the ordinary locus $\cU_{\sig}$ (Theorem \ref{thm:modp-parabolic}). Using this, we define the notion of \textit{supersingular $G_K$-representation of weight $\sig$} (Definition \ref{defn:strata}).

\subsection{Ramified local models}
We briefly recall the theory of local models from \cite[\S3]{LLLM-extreme} (which is based on \cite{LLLMlocalmodel}). Throughout this subsection, we fix a tame inertial type $\tau$ with $1$-generic lowest alcove presentation $(s,\mu)$.

We define the following (ind-)group schemes over $\cO$: for $\jj$, $h\in \Z_{\ge 0}$ and any $p$-adically complete $\cO$-algebra $R$,
\begin{align*}
    L\cG\ix{j}(R) &:= \CB{A\in \GL_n(R\DB{v}[\tfrac{1}{E_j}])} \\
    L^+\cG\ix{j}(R) &:= \CB{A\in \GL_n(R\DB{v}) \mid A \mod v \in B(R)}
    \\
    L^{[0,h]}\cG\ix{j}(R) & := \CB{A\in L\cG\ix{j}(R) \mid A, E_j^hA^\mo\in \Mat_n(R\DB{v}) \text{ and upper triangular modulo $v$}}.
\end{align*}
Note that since $\tau$ is $1$-generic, $\cP_j$ defined in \S\ref{subsec:BKmodules} is equal to $L^+\cG\ix{j}$.  
We have the $(s,\mu)$-twisted $\varphi$-conjugation action of $\prod_{\jj}L^+\cG\ix{j}$ on $\prod_{\jj} L^{[0,h]}\cG\ix{j}$ defined by
\begin{align*}
    (I\ix{j}) \cdot (A\ix{j}) := I\ix{j}A\ix{j} (\Ad(s_j^\mo v^{\mu\ix{j}+\eta\ix{j}})(\varphi(I\ix{j-1})^\mo)).
\end{align*}
Recall that we have a torsor $Y^{[0,h],\tau, \be}\ra Y^{[0,h],\tau}$ (Proposition \ref{prop:eigenbasis-torsor}). There is a morphism $Y^{[0,h],\tau, \be} \ra \prod_{\jj} L^{[0,h]}\cG\ix{j}$ sending $(\fM,\phi_\fM,\be)$ to $(A_{\fM,\be}\ix{j})_{\jj}$. By Proposition \ref{prop:change-of-basis-formula}, we have the presentation
\begin{align*}
    Y^{[0,h],\tau} \simeq \BR{\prod_{\jj} L^{[0,h]}\cG\ix{j}/_{(s,\mu),\varphi} L^+\cG\ix{j}}.
\end{align*}
In the special fiber, we have a refined presentation. The scheme $\cI:= L^+\cG\ix{j}_\F$ is the usual Iwahori group scheme over $\F$, and we let $\cI_1\subset \cI$ be the pro-$p$ Iwahori group scheme. By \cite[Rmk.~3.1.5]{LLLM-extreme}), if $\tau$ is $(eh+1)$-generic, we have
\begin{align*}
   \pi_{(s,\mu)}: Y^{[0,h],\tau}_\F \risom [(\prod_{\jj}\cI_1\bss L^{[0,h]}\cG\ix{j}_\F)/_{(s,\mu),\varphi} T^{\vee,\cJ}_\F]. 
\end{align*}

For $\tilz \in \tilW^{\vee,\cJ}$, we recall the scheme $U^{[0,h]}(\tilz) = \prod_{\jj}U^{[0,h]}(\tilz_j)$ defined in \cite[\S4.2]{LLLM-extreme}. It has the property that the product $T^{\vee,\cJ}U^{[0,h]}(\tilz)$ maps into $[\prod_{\jj}\cI_1\bss L^{[0,h]}\cG\ix{j}_\F]$ as an open immersion, and moreover they form an open cover by varying $\tilz$.

\begin{defn}
    Let $\tilz\in \tilW^{\vee,\cJ}$ and $R$ be a $p$-adically complete Noetheiran $\cO$-algebra.
    \begin{enumerate}
    \item We define $Y^{[0,h],\tau}_\F(\tilz) \subset Y^{[0,h],\tau}_\F$ to be the open substack corresponding via $\pi_{(s,\mu)}$ to $[T^{\vee,\cJ}{U}^{[0,h]}(\tilz)_\F/_{(s,\mu),\varphi} T^{\vee,\cJ}_\F]$.
    \item We define $Y^{[0,h],\tau}(\tilz) \subset Y^{[0,h],\tau}$ to be the open substack induced by $Y^{[0,h],\tau}_\F(\tilz)$.
    \item We define $Y^{\le\eta,\tau}(\tilz) = Y^{\le\eta,\tau}\cap Y^{[0,h],\tau}(\tilz)$.
    \item We say that $\fM\in Y^{\le\eta,\tau}(R)$ \textit{admits a $\tilz$-gauge} if it is contained in $Y^{\le\eta,\tau}(\tilz)(R)$. 
\end{enumerate}
\end{defn}

Recall the Pappas--Zhu local model $M_\cJ(\le\eta) = \prod_{\jj}M_j(\le\eta_j)$. For $\tilz \in \Adm^\vee(\eta)$, we have a closed subscheme $U(\tilz,\le\eta)\subset U(\tilz)$ which identifies with an open neighborhood of $M_\cJ(\le\eta)$, and $\CB{U(\tilz,\le\eta)}_{\tilz \in \Adm^\vee(\eta)}$ is an open cover of $M_\cJ(\le\eta)$ (see Remark 4.2.3 and the preceding discussion in \cite{LLLM-extreme}).

The next theorem is a generalization of \cite[Thm.~5.3.3 and Cor.~5.3.4]{LLLMlocalmodel} to a ramified base. Its proof carries over verbatim (using \cite[Prop.~4.2.2]{LLLM-extreme} instead of \cite[Prop.~5.2.7]{LLLMlocalmodel}).
\begin{thm}\label{thm:BKstack-localmodel}
   Suppose that $\tau$ is $(e(n-1)+1)$-generic. The open substack $Y^{\le \eta,\tau}(\tilz) \subset Y^{\le\eta,\tau}$ is non-empty if and only if $\tilz \in \Adm^\vee(\eta)$. For each $\tilz \in \Adm^\vee(\eta)$, we have the following local model diagram
    \[
\begin{tikzcd}
& T^{\vee,\cJ}{U}(\tilz, \le \eta)^\pcp \arrow[ld, "T^{\vee,\cJ}_\cO"'] \arrow[rd, "T^{\vee,\cJ}_\cO"] & \\
Y^{\le \eta,\tau}(\tilz) = \BR{T^{\vee,\cJ}{U}(\tilz, \le \eta) /_{(s,\mu)} T^{\vee,\cJ}_\cO}^\pcp 
& & U(\tilz, \le \eta)^\pcp 
\end{tikzcd}
\]
where the superscript $\pcp$ stands for taking $p$-adic completion. The diagonal arrows are torsors for $(T^{\vee,\cJ})^\pcp$. The left diagonal arrow corresponds to the quotient by the $(s,\mu)$-twisted conjugation action, while the right diagonal arrow corresponds to the quotient by the left translation action.
\end{thm}

\subsection{Parabolic structures}

%We now replace the role of $\tau$ by its dual $\tau^\vee$. In particular, we have $\tau^\vee\simeq \oplus_{I\in \fP_{\tau}}\tau^\vee_I$ where $\tau^\vee_I$'s are pairwise non-isomorphic. 

For any $\fI \subset \fP_{\tau}$, we have a function $\til{F}_{\fI}\in \cO(Y^{\eta,\tau})$ defined in Definition \ref{defn:glob-func}. If $\tau$ is $(e(n-1)+2)$-generic, then $\til{F}_{\fI}$ defines a global function on $\cX^{\eta,\tau}$. We denote by $\cX^{\eta,\tau}(\til{F}_\fI)\subset \cX^{\eta,\tau}$ the open substack obtained by inverting $\til{F}_\fI$. 

Let $a :=n_{\fI}$ and $b:=n-n_{\fI}$. We write $\tau_a := \oplus_{I\in \fI}\tau_I$ and $\tau_b := \oplus_{I\notin \fI}\tau_I$. By Proposition \ref{prop:reducibility}, any $\rho\in \cX^{\eta,\tau}(\til{F}_\fI)(\cO)$ is an extension of $\rho_b$ by $\rho_a$ where $\rho_a \in \cX^{\eta_a+\ud{b},\tau_a}(\cO)$ and $\rho_b\in \cX^{\eta_b,\tau_b}$. The main result of this subsection shows the map $\rho \mapsto (\rhobar_a,\rhobar_b)$ extends continuously to families.

\begin{thm}\label{thm:reducibility-in-families}
    Suppose that $\tau$ is $((e+1)(n-1)+2)$-generic. There is a morphism $\cX^{\eta,\tau}(\til{F}_\fI) \ra \cX^{\eta_a+\ud{b},\tau_a} \times \cX^{\eta_{b},\tau_b}$ which makes $\cX^{\eta,\tau}(\til{F}_\fI)$ into a rank $[K:\Qp]ab$ vector bundle over the target. It maps $\rho \in \cX^{\le\eta,\tau}(\til{F}_\fI)(\cO)$ to $(\rho_a,\rho_b)$ where $\rho_a$ is the subrepresentation of $\rho$ given by Proposition \ref{prop:reducibility}, and $\rho_b = \rho/\rho_a$.
\end{thm}

To prove this, we use the parabolic structures on local models as discussed in \cite[\S4.2]{LLLM-extreme}. We introduce some notations first. 

%\begin{lem}\label{lem:aff-weyl-decomp}
%    Let $\tilw  = t_{\nu} w \in \tilW^+$. Let $w^M \in W^M$ be the unique element given by the condition $W_Mw = W_M w^M$ so that we have a factorization $w = w_M w^M$ where $w_M=(w_a,w_b)\in W_M$. If we write $\nu = (\nu_a,\nu_b)$, then $\tilw_a =  t_{\nu_a}w_a \in \tilW^+_{a}$ and $\tilw_b = t_{\nu_b}w_b \in \tilW^+_{b}$ satisfy  $\tilw\cdot C_0 \subset \tilw_a \cdot C_{a,0} \times \tilw_b \cdot C_{b,0}$.
%\end{lem}

\begin{setup}\label{setup:parabolic-shape} Let $a,b$ be positive integers such that $a+b=n$. We let $P\subset \GL_n$ be the standard parabolic subgroup given by the partition $n=a+b$ and $M\simeq \GL_a \times \GL_b$ be its Levi factor.  For $d \in \{a,b\}$, we choose tame inertial types $\rhobar_d, \tau_d: I_K \ra \GL_d(\cO)$ and define $\rhobar := \rhobar_a \oplus \rhobar_b$ and $\tau := \tau_a \oplus \tau_b$. We assume that $\tau$ is $((e+1)(n-1)+2)$-generic. We explain how to choose lowest alcove presentations of these tame inertial types so that 
\begin{align}\label{eqn:parabolic-shape}
    \tilw(\rhobar,\tau) = \pi(w^M)^\mo \pma{\tilw(\rhobar_a,\tau_a) & \\ & \tilw(\rhobar_b,\tau_b)}  \pi(w^M)
\end{align}
for a certain element $w^M\in \uW^M$. In the above notation, we view $\tilw(\rhobar_d,\tau_d)$ as a $\cJ$-tuple of $d\times d$ matrices and $\pi(w^M)$ as a $\cJ$-tuple of $n\times n$ elementary matrices.
    \begin{enumerate}
        \item For $d \in \{a,b\}$, let $(s_{\rhobar_d},\mu_{\rhobar_d})$ and $(s_d,\mu_d)$ be lowest alcove presentations of $\rhobar_d$ and $\tau_d$, respectively. We assume that 
\begin{align*}
    \tilw'_d = w'_dt_{\nu'_d} := \tilw(\rhobar_d,\tau_d)
\end{align*}
is in $\Adm(\eta_d+\del_{d=a}{\ud{b}})$. Note that this implies 
\begin{align*}
    (s_{\rhobar_d},\mu_{\rhobar_d}) &= (s_dw'_d, \mu_d+s_dw'_d(\nu'_d) + \del_{d=a}\ud{b}).
    %(s_{\rhobar_b},\mu_{\rhobar_b}) &= (s_bw'_b, \mu_b+s_bw'_b(\nu'_b)).
\end{align*}
        \item  
We define $s = (s_a,s_b)$ and $\mu=(\mu_a-\ud{b},\mu_b)$ so that $\tau \simeq \tau(s,\mu+\eta)$. We choose $\tilw = t_\nu w \in \utilW$ such that $\mu \in \tilw\cdot \uC_0$. Then we have a lowest alcove presentation of $\tau$ given by
\begin{align*}
    \prescript{\tilw^\mo}{}{(s,\mu)} =  (w^\mo s \pi(w) , \tilw^\mo\cdot \mu +w^\mo s \pi(\nu)).
\end{align*}
Similarly, we define $s_{\rhobar} := (s_{\rhobar_a},s_{\rhobar_b})$ and $\mu_{\rhobar} := (\mu_{\rhobar_a}-\ud{b}, \mu_{\rhobar_b})$ so that $\rhobar \simeq \tau(s_{\rhobar},\mu_{\rhobar}+\eta)$ and
\begin{align*}
     \prescript{\tilw^\mo}{}{(s_{\rhobar},\mu_{\rhobar})} =  (w^\mo s_{\rhobar} \pi(w) , \tilw^\mo\cdot \mu_{\rhobar} +w^\mo s_{\rhobar} \pi(\nu))
\end{align*}
is a lowest alcove presentation of $\rhobar$.

    \item Let $w^M \in \uW^M$ be the unique element given by the condition $\uW_Mw = \uW_M w^M$ so that we have a factorization $w = w_M w^M$ for some $w_M=(w_a,w_b)\in \uW_M$. Also, we write $\nu = (\nu_a,\nu_b) \in X^*(\uT_a)^\cJ \times X^*(\uT_b)^\cJ$. We define $\tilw_a :=  t_{\nu_a}w_a \in \utilW_{a}$ and $\tilw_b := t_{\nu_b}w_b \in \utilW_{b}$. It follows from the definition that  $\tilw\cdot \uC_0 \subset \tilw_a \cdot \uC_{a,0} \times \tilw_b \cdot \uC_{b,0}$. By the definition of $\tilw$, this implies that $\tilw_a \in \uOm_a$ and $\tilw_b \in \uOm_b$.
        
        \item %For $d\in \CB{a,b}$, let $\tilw_d = t_{\nu_d}w_d\in \tilW_{d,1}^+$ be the element provided by applying Lemma \ref{lem:aff-weyl-decomp} to $\tilw$ in (2). We also get a decomposition $w=w_Mw^M$. Note that by the definition of $\tilw$, $\tilw_d\in \Omega_d$. 
        We choose the following modified lowest alcove presentation of $\tau_d$ and $\rhobar_d$ given by
        \begin{align*}
            \prescript{\tilw_d^\mo}{}{(s_d,\mu_d)} &= (w_d^\mo s_d \pi(w_d) , \tilw_d^\mo \cdot \mu_d - w_d^\mo s_d\pi(\nu_d))) \\ 
            \prescript{\tilw_d^\mo}{}{(s_{\rhobar_d},\mu_{\rhobar_d})} &= (w_d^\mo s_{\rhobar_d} \pi(w_d) , \tilw_d^\mo \cdot \mu_{\rhobar_d} - w_d^\mo s_{\rhobar_d}\pi(\nu_d)))
        \end{align*}
        Then a direct computation shows that \eqref{eqn:parabolic-shape} holds.
    \end{enumerate}
\end{setup}
We remark that $((e+1)(n-1)+2)$-genericity of $\tau$ implies that any lowest alcove presentation of $\tau$ is $((e+1)(n-1)+1)$-generic (this follows from the proof of \cite[Lem.~2.3.2]{LLLMlocalmodel}) and $\tau_a$ and $\tau_b$ are $(e(n-1)+2)$-generic.

\begin{prop}[Proposition 4.2.5 in \cite{LLLM-extreme}]
We follow the notations in Setup \ref{setup:parabolic-shape}. Let $R^\univ = \cO(T^\vee U( \tilw(\rhobar,\tau), \le \eta))$.  The universal matrix $A^\univ \in T^\vee U( \tilw(\rhobar,\tau), \le \eta)(R^\univ)$ factors as
\begin{align*}
    A^\univ = \pi(w^M)^\mo \pma{D_a^\univ & \\ & D_b^\univ} \pma{M_a^\univ & \\ X^\univ & M_b^\univ} \pi(w^M) 
\end{align*}
where for $d\in\CB{a,b}$, $D_d^\univ \in T^\vee_d(R^\univ)$ and $M_d^\univ \in U(\tilw(\rhobar_d,\tau_d),\le\eta_d+\del_{d=a}\ud{b})(R^\univ)$. This induces a morphism
\begin{align}\label{eqn:parabolic-morphism-pre}
    T^\vee U( \tilw(\rhobar,\tau), \le \eta) \ra T^\vee_a U (\tilw(\rhobar_a,\tau_a),\le\eta_a +\ud{b}) \times T^\vee_b U(\tilw(\rhobar_b,\tau_b),\le\eta_b)
\end{align}
which exhibits $T^\vee U( \tilw(\rhobar,\tau), \le \eta)$ as an affine space over $T^\vee_a U (\tilw(\rhobar_a,\tau_a),\le\eta_a +\ud{b}) \times T^\vee_b U(\tilw(\rhobar_b,\tau_b),\le\eta_b)$ of relative dimension $ab[K:\Qp]$.
\end{prop}

Recall that $\cX^{\le\eta,\tau}$ can be identified with a closed substack of $Y^{\le\eta,\tau}$. For $\tilz\in \Adm^\vee(\eta)$, we write $\cX^{\le\eta,\tau}(\tilz) = \cX^{\le\eta,\tau} \cap Y^{\le\eta,\tau}(\tilz)$. Note that if we take $\fI\subset \fP_\tau$ such that $\tau_a \simeq \oplus_{I\notin \fI}\tau^\vee_I$, then $\cX^{\le\eta,\tau}(\tilz) \subset \cX^{\le\eta,\tau}(\til{F}_{\fI})$.

\begin{lem}\label{lem:parabolic-morphism-pre-gal}
The morphism \eqref{eqn:parabolic-morphism-pre} induces a morphism     
\begin{align*}
     \cX^{\le\eta,\tau}(\tilw(\rhobar,\tau)) \ra \cX^{\le\eta_a+\ud{b}, \tau_a}(\tilw(\rhobar_a,\tau_a)) \times \cX^{\le\eta_b,\tau_b}(\tilw(\rhobar_b,\tau_b)).
     %[U( \tilw(\rhobar,\tau), \le \eta)/_{(\varphi,\tau)}T ] \ra [U (\tilw(\rhobar_a^{\tilw_a},\tau_a^{\tilw_a}),\le\eta_a +\ud{b}) /_{(\varphi,\tau_a^{\tilw_b})} T_a ] \times [U(\tilw(\rhobar_b^{\tilw_b},\tau_b^{\tilw_b}),\le\eta_b)/_{(\varphi,\tau_b^{\tilw_b})}T_b].
\end{align*}
For any finite extension $\cO'/\cO$, it maps $\rho \in \cX^{\le\eta,\tau}(\tilw(\rhobar,\tau))(\cO')$ to $(\rho_a,\rho_b)$ where $\rho_a$ is the subrepresentation of $\rho$ given by Proposition \ref{prop:reducibility}, and $\rho_b = \rho/\rho_a$.
\end{lem}

\begin{proof}
    It follows from a direct computation using Theorem \ref{thm:BKstack-localmodel} and Proposition \ref{prop:change-of-basis-formula} that \eqref{eqn:parabolic-morphism-pre} induces a morphism
    \begin{align*}
     Y^{\le\eta,\tau}(\tilw(\rhobar,\tau)) \ra Y^{\le\eta_a+\ud{b}, \tau_a}(\tilw(\rhobar_a,\tau_a)) \times Y^{\le\eta_b,\tau_b}(\tilw(\rhobar_b,\tau_b)).
\end{align*}
    It is clear from the definition this maps $\fM\in Y^{\le\eta,\tau}(\tilw(\rhobar,\tau))(\cO)$ to $(\fM_a,\fM_b)$ where $\fM_b\subset \fM$ is the unique subobject of Hodge type $\le\eta_b$ and tame inertial type $\tau_b$, and $\fM_a = \fM/\fM_b$. 
    
    To finish the proof, we need to show that the image of $\cX^{\le\eta,\tau}(\tilw(\rhobar,\tau))$ under the above morphism is contained in $\cX^{\le\eta_a+\ud{b}, \tau_a}(\tilw(\rhobar_a,\tau_a)) \times \cX^{\le\eta_b,\tau_b}(\tilw(\rhobar_b,\tau_b))$. This follows from the observation that $\veps_{\tau}(\fM_a)$ is a sub-$G_{K_\infty}$-representation of $\veps_{\tau}(\fM)$ whose $G_{K_\infty}$-action extends to $G_K$ and thus provides a sub-$G_K$-representation $\rho_a\subset \rho$. 
\end{proof}

\begin{proof}[Proof of Theorem \ref{thm:reducibility-in-families}]
In addition to given $\tau_a,\tau_b$, we can choose $\rhobar_a,\rhobar_b$, and $\rhobar=\rhobar_a\oplus\rhobar_b$ as in Setup \ref{setup:parabolic-shape}. Then $\cX^{\eta,\tau}(\til{F}_{\fI})$ contains $\cX^{\eta,\tau}(\tilw(\rhobar,\tau))$ as an open substack. Indeed, we have
\begin{align*}
    \cX^{\eta,\tau}(\til{F}_{\fI}) = \cup_{\rhobar}\cX^{\eta,\tau}(\tilw(\rhobar,\tau))
\end{align*}
where the union runs over all choices of $\rhobar=\rhobar_a\oplus\rhobar_b$ such that for $d\in \CB{a,b}$, $\tilw(\rhobar_d,\tau_d)\in \Adm^\vee(\eta_d +\del_{d=a}\ud{b})$. To see this, first note that the right hand side contains the image of the natural morphism given by taking direct sum
\begin{align*}
    \cX^{\le\eta_a+\ud{b}, \tau_a} \times \cX^{\le\eta_b,\tau_b} \ra \cX^{\le\eta,\tau}.
\end{align*}
Suppose that $\rho \in \cX^{\eta,\tau}(\til{F}_{\fI})(\cO)$ is not contained in $\cup_{\rhobar}\cX^{\eta,\tau}(\tilw(\rhobar,\tau))(\cO)$. Then any point in the closure of $\rho$ is not contained in the latter. If we write $\rho_a$ for the subrepresentation of $\rho$ given by Proposition \ref{prop:reducibility} and $\rho_b = \rho/\rho_a$, by \cite[Thm.~6.6.3 (3)]{EGstack}, the closure of $\rho$ contains $\rho_a \oplus \rho_b \mod \varpi$. Since $\rho_a \oplus \rho_b \mod \varpi$ is contained in $\cup_{\rhobar}\cX^{\eta,\tau}(\tilw(\rhobar,\tau))(\cO)$, we get a contradiction.

We show that the morphism in Lemma \ref{lem:parabolic-morphism-pre-gal} for different choices of $\rhobar_a,\rhobar_b$ glue to the morphism in the statement. Since both the source and the target are residually Jacobosn, it suffices to show that two morphisms coincide at any $\cO'$-points in the intersection of the domains. This follows from our characterization of the map at the level of $\cO'$-points.

Finally, we explain why $\cX^{\eta,\tau}(\til{F}_{\fI})$ is a vector bundle over the target. Equivalently, we show that the morphism in Lemma \ref{lem:parabolic-morphism-pre-gal} is a vector bundle. Using Proposition \ref{prop:BK-EG} and Theorem \ref{thm:BK-localmodel}, we have a pullback $\til{\cX}^{\le\eta,\tau}(\tilw(\rhobar,\tau)) \subset T^{\vee,\cJ}U(\tilz,\le\eta)^\pcp$ of ${\cX}^{\le\eta,\tau}(\tilw(\rhobar,\tau))$ and the natural morphism $\til{\cX}^{\le\eta,\tau}(\tilw(\rhobar,\tau)) \ra {\cX}^{\le\eta,\tau}(\tilw(\rhobar,\tau))$ is a $T^{\vee,\cJ}$-torsor (similarly for the the ones with subscribt $a,b$). The closed formal subscheme $\til{\cX}^{\le\eta,\tau}(\tilw(\rhobar,\tau)) \subset T^{\vee,\cJ}U(\tilz,\le\eta)^\pcp$ is given by the ``monodromy condition'' as in \cite[\S4.4]{LLLM-extreme}. Then Lemma 4.4.3 in \loccit computes a $p$-adic approximation of the monodromy condition, and when combined with Proposition 4.4.1, Lemma 4.4.2, and Remark 4.4.5 in \loccit, this shows that 
\begin{align*}
    \til{\cX}^{\le\eta,\tau}(\tilw(\rhobar,\tau)) \ra \til{\cX}^{\le\eta_a + \ud{b},\tau_a}(\tilw(\rhobar_a,\tau_a)) \times \til{\cX}^{\le\eta_b,\tau_b}(\tilw(\rhobar_b,\tau_b))
\end{align*}
is a vector bundle of rank $[K:\Qp]ab$.
%where $\til{\cX}^{\eta,\tau}(\til{F}_\fI) \subset T^{\vee,\cJ}U(\tilz,\le\eta)^\pcp$  of the morphism in Lemma \ref{lem:parabolic-morphism-pre-gal} as a closed $p$-adic formal subscheme of \eqref{eqn:parabolic-morphism-pre}.  Since $\cX^{\le\eta,\tau}(\tilw(\rhobar,\tau^\vee))$ can be viewed as a substack of $Y^{\le\eta,\tau^\vee}(\tilw(\rhobar,\tau^\vee))$ (and similarly for the the ones with subscribt $a,b$) and $T^\vee U(\tilw(\rhobar,\tau^\vee),\le\eta)$ is a $T^\vee$-torsor over  This essentially follows from \cite[Lemma 4.4.3]{LLLM-extreme}. The proof of \loccit~shows that the morphism \eqref{eqn:parabolic-morphism-pre}  
\end{proof}

Now we discuss an application to the special fiber. Let $(\tilw,\om)$ be a lowest alcove presentation of $\sig$ and define $\ud{\mu}$, $\mu$, and $\tau=\tau(e,\mu)$ as in Example \ref{ex:hecke-modp}. Suppose that $\sig$ is a $((e+2)(n-1)+2)$-deep. Then  $\tau$ is $((e+1)(n-1)+2)$-generic.

\begin{defn}\label{defn:strata}
Let $I=\CB{i_1<i_2<\dots<i_r}\subset \CB{1,2,\dots,n-1}$ be a subset.
    \begin{enumerate}
        \item We define $\cU_{\sig,I}:= \cap_{i\in I} \cC_\sig(\ov{f}_i)$. Note that $I \mapsto \cU_{\sig,I}$ is inclusion-reversing. This defines a stratification on $\cC_\sig$ with strata given by $\cC_{\sig,I}:= \cU_{\sig,I} \bss \cup_{I\subset I'}\cU_{\sig,I'}$. When $I=\emptyset$, we define $\cC_{\sig}^\ss:=\cC_{\sig,\emptyset}$ the \textit{supersingular locus} of $\cC_\sig$. We say that a continuous $G_K$-representation $\rhobar: G_K \ra \GL_n(\Fpbar)$ is \textit{supersingular of weight $\sig$} if $\rhobar\in\cC^\ss_\sig(\Fpbar)$.
        \item Let $P_I\subset \GL_n$ be the standard upper triangular parabolic subgroup with Levi factor $M_I \simeq \prod_{k=1}^{r+1} \GL_{i_{k}-i_{k-1}}$ where we write $i_0=0$ and $i_{r+1}=n$. We also define $\cX_{M_I,\red}:= \prod_{k=1}^{r+1} \cX_{i_{k}-i_{k-1},\red}$.
        \item We define $\sig_I = F_{M_I}(\mu+\eta-\eta_M)$, where $\eta_M$ defined analogously to $\eta$ for $M$. Then $\cC_{\sig_{I}}\subset \cX_{M_I,\red}$ is the irreducible component which is uniquely characterized as a closure of the image of the map $\cU_\sig(\Fpbar)\ra \cX_{M_I,\red}(\Fpbar)$ given by the quotient map $P_I\epi M_I$.
    \end{enumerate}
\end{defn}

We refer to Remark \ref{rmk:LGC}(4) how the notion of being supersingular of weight $\sigma$ is related to the representation side.

%\begin{rmk}
%    The notion of $\rhobar$ being supersingular of weight $\sigma$ is supposed to correspond to the property of a mod $p$ smooth representation $\pi(\rhobar)$ having $\sigma$ as a supersingular weight where $\rhobar \mapsto \pi(\rhobar)$ is a conjectural mod $p$ local Langlands correspondence (see Remark \ref{rmk:LGC}). We remark that this is a premature definition, and the correct notion should be defined in terms of the conjectural Breuil--M\'ezard cycle attached to $\sig$ instead of $\cC_\sig$.
%\end{rmk}

\begin{thm}\label{thm:modp-parabolic}
    Suppose that $\sig$ is $((e+2)(n-1)+2)$-deep. For each $I\subset \CB{1,2,\dots,n-1}$, there is a natural map 
    \begin{align*}
        \cU_{\sig,I} \ra \cC_{\sig_I}
    \end{align*}
    whose restriction to $\cC_{\sig,I}$ lands into $\cC_{\sig_I}^\ss$.
\end{thm}
\begin{proof}
    This follows from Theorem \ref{thm:reducibility-in-families} and \ref{thm:spec-satake}.
\end{proof}

\begin{rmk}
    Already in the case $I=\CB{1,2,\dots,n-1}$, our $\cU_{\sig,I}$ slightly refines the ordinary locus $\cU_\sig$ in the following sense. It is not necessarily true that $\cU_\sig$ contains all $\rhobar$ of the form \eqref{eqn:ord-rhobar}. Indeed, what we denote by $\cU_\sig$ is really \textit{some} open dense substack of $\cC_\sig$ satisfying the properties therein. One reason behind this is that the construction of the eigenvalue morphism \cite[Prop.~5.3.8]{EGstack} relies on a certain birational morphism (see the proof of Lem.~5.3.7 in \loccit) and thus requires replacing $\cU_\sig$ by its open dense substack. However, at least for $((e+2)(n-1)+2)$-deep $\sig$, the above result shows that $\cU_{\sig,I}$ exactly contains all $\rhobar$ of the form \eqref{eqn:ord-rhobar} and their partial semisimplifications. 
\end{rmk}

\section{Relationship with the $p$-adic local Langlands program}\label{sec7}

\subsection{Mod $p$ local-global compatibility}\label{subsec:LGC} We briefly recall the global setup from \cite[\S10.3]{LLMPQ-FL}. We refer the reader to \loccit~for details. Let ${F}$ be an imaginary CM field with maximal totally real subfield $F^+$. We assume that $[F^+:\Q]$ is even and that all places of $F^+$ above $p$ split in $F$. We let $G_{/F^+}$ be a reductive group which is an outer form of $\GL_n$ splitting over $F$ such that $G(F_v^+)\simeq U_n(\R)$ for all $v\mid \infty$. We let $\cG$ be its reductive model over $\cO_{F^+}[1/N]$ for some $N\in \Z_{>0}$ coprime to $p$ as in \loccit.

%For a compact open subgroup $U\le G(\A^\infty_{F^+}$, we define the space of mod $p$ algebraic automorphic forms
%\begin{align*}
%    S(U,\F) := \CB{f: G(F^+)\bss G(\A^\infty_{F^+})/U \ra \F}.
%\end{align*}
%Let $\cP_U$ be the set of finite places in $F$ such that $v:=w|_{F^+}$ splits in $F$, $w\nmid pN$, and $U$ is unramified at $v$. For a subset $\cP\subset \cP_U$ of finite complement and closed with respect to complex conjugation, we define $\T^\cP := \cO[T_w\ix{i} \mid w\in \cP, i\in \bfn]$ as in \loccit. We have a maximal ideal $\fm_{\rbar}\le \T^\cP$ defined in \loccit.

Let $\rbar: G_F \ra \GL_n(\F)$ be an absolutely irreducible continuous representation and let $\rhobar:= \rbar|_{G_{F_{\til{v}}}}$. Following \cite{LLMPQ-FL}, let $S(U,\F)$ be the space of algebraic automorphic forms on $G$ of level $U= U^vU_v$ and coefficients $\F$. We have a set $\cP_U$ of finite places $w$ of $F$ such that $w|_{F^+}$ splits in $F$, $w\nmid pN$, and $U$ is unramified at $w|_{F^+}$. For a subset $\cP\subset \cP_U$ of finite complement, let $\T^{\cP}$ be the abstract Hecke algebra generated by the usual double coset operators at places in $\cP$ with a maximal ideal $\fm_{\rbar}$ determined by $\rbar$.

We choose a place $\til{v}\mid p$ of $F$ and let $v:= \til{v}|_{F^+}$. We denote by $K$ the finite extension $F_{\til{v}} \simeq F^+_v$ of $\Qp$. We define an admissible smooth $\cG(F^+_v)$-representation
\begin{align*}
    \pi(\rbar) := \varinjlim_{U_v \le \cG(\cO_{F_v^+})} S(U^vU_v,\F)[\fm_{\rbar}].
\end{align*}
It is a folklore conjecture that $\pi(\rbar)$ is \textit{purely local}, i.e.~only depends on $\rhobar$ (up to multiplicity). This is only proven in the case $\GL_2(\Qp)$ \cite{6authorGL2}, although numerous important progress has been made for $\GL_2(\Qpf)$ (see the introduction in \cite{BHHMS2}). In a different direction, it is natural to ask whether we can recover $\rhobar$ from $\pi(\rbar)$. This question was first studied by Breuil--Diamond \cite{BD} and further studied by several authors (\cite{HLM,LMP,PQ,EL} and most importantly \cite{LLMPQ-FL}), all under the assumption that $K/\Qp$ is unramified using Fontaine--Laffaille theory. We refer to the introduction of \cite{LLMPQ-FL} for a more detailed account of this problem. Also, we remark that this is the Galois-to-automorphic direction of the $p$-adic local Langalnds program, opposite to the categorical conjecture by Emerton--Gee--Hellmann discussed in the introduction. See \cite[\S7.8]{EGH22} for the discussion on the relationship between the two conjectural correspondences.

We expect that our Theorem \ref{thm:spec-satake} will be useful to study these problems. As a first application, we prove the following result.

\begin{thm}\label{thm:modpLGC}
    Let $\sig$ be a $((e+1)(n-1)+2)$-deep Serre weight of $\GL_n(k)$ and $\rhobar \in \cC_\sig(\F)$. Suppose that
    \begin{enumerate}
        \item $\cG(\cO_{F^+_v})U^v$ is sufficiently small;
        \item the image of $\rbar|_{G_{F(\zeta_p)}}$ is adequate (in the sense of \cite[Def.~2.3]{Thorne-smallresidual}); and
        \item $\Hom_{\rmK}(\sig, \pi(\rbar)|_{\rmK})\neq 0$.
    \end{enumerate}
    Then the natural action of $\cH(\sig)$ on $\Hom_{\rmK}(\sig, \pi(\rbar)|_{\rmK})$ is given by the character
    \begin{align*}
        \cH(\sig) \xra{\ov{\Psi}_\sig} \cO(\cC_\sig) \xra{\ev_{\rhobar}} \F.
    \end{align*}
\end{thm}

\begin{rmk}\label{rmk:LGC}
    \begin{enumerate}
        \item The assumption that $\Hom_{\rmK}(\sig, \pi(\rbar)|_{\rmK})\neq 0$ is equivalent to saying that $\sig$ is \textit{modular} Serre weight of $\rbar$. In turn, this is conjectured to be equivalent to the condition that $\rhobar$ is contained in the conjectural Breuil--M\'ezard cycle $\cZ_\sig$ attached to $\sig$ (\cite[Conj.~9.1.5]{LLLMlocalmodel}). It is expected that $\cC_\sig$ is contained in $\cZ_\sig$, so the assumption here is natural.  
    
    \item This result is implicitly used in \cite{EGHweightcyc}. More precisely, although the Emerton--Gee stack was not available at that time, they interpret $\ov{f}_i(\rhobar)$ for each $i=1,\dots,n$ as the mod $\varpi$ reduction of a certain normalized coefficient of the characteristic polynomial of the Frobenius action on $\Dcris(\rho)$ for a crystalline lift $\rho$ of $\rhobar$. Then they prove that the $\cH(\sig)$-action as above (using a variant of $\pi(\rbar)$) is given by these normalized coefficients modulo $\varpi$. As the reader can easily recognize, our Theorem \ref{thm:modpLGC} is inspired by \cite{EGHweightcyc}.

    \item Suppose that $\sig$ is $((e+2)(n-1)+2)$-deep. Suppose that $\rhobar \in \cC_{\sig,I}(\F)$ for some $I\subset \CB{1,2,\dots,n-1}$ of size $r$. By Theorem \ref{thm:modp-parabolic}, there is a filtration 
    \begin{align*}
        0 = \rhobar_0 \subset \rhobar_1 \subset \cdots \subset \rhobar_{r+1} = \rhobar
    \end{align*}
    such that $\oplus_{i=1}^{r+1}\gr^i \rhobar$ is the image of $\rhobar$ in $\cC_{\sig_I}$. It is easy to see that for each $i=1,\dots,r+1$, $\det(\gr^i\rhobar)$ is determined by $\sig$ up to an unramified twist. Theorem \ref{thm:modpLGC} shows that $\pi(\rbar)$ completely determines $\det(\gr^i\rhobar)$.

    \item Suppose that $\sig$ is $((e+2)(n-1)+2)$-deep. Suppose that $\rhobar$ is supersingular of weight $\sigma$. Then Theorem \ref{thm:modpLGC} shows that $\sig$ is \textit{supersingular weight} of $\pi(\rbar)$. More precisely, by Frobenius reciprocity, a non-zero morphism $\sig \mono \pi(\rbar)|_\rmK$ induces a non-zero morphism
    \begin{align*}
        \cind_{\rmK}^\rmG \sig \ra \pi(\rbar),
    \end{align*}
    and Theorem \ref{thm:modpLGC} shows that the morphism factors through the universal supersingular quotient $\cind_{\rmK}^\rmG \sig \epi \cind_{\rmK}^\rmG \sig /(\ov{T}_1,\dots,\ov{T}_{n-1})$.
    \end{enumerate}
\end{rmk}

We deduce Theorem \ref{thm:modpLGC} from a stronger result on the action of $\cH(\sig)$ on a Taylor--Wiles--Kisin patched module. Let $R^\square_{\rhobar}$ be the universal $\cO$-lifting ring of $\rhobar$. Suppose that the assumptions (1)-(3) in Theorem \ref{thm:modpLGC} hold. By \cite[Prop.~10.3.2]{LLMPQ-FL}, there exist $R^v$ and
\begin{align*}
    R_\infty := R_{\rhobar}^\square \ctimes_\cO R^v
\end{align*}
complete equidimensional local Noetherian $\cO$-flat algebras with residue field $\F$ 
and an \textit{arithmetic $R_\infty[\GL_n(K)]$-module $M_\infty$} in the sense of Def.~10.2.3 in \loccit. It is an $\cO$-module $M_\infty$ with commuting actions of $R_\infty$ and $\GL_n(K)$ such that 
\begin{enumerate}
    %\item the $R_\infty[\rmK]$-action on $M_\infty$ extends to $R_\infty\DB{\rmK}$ making $M_\infty$ a finitely generated $R_\infty\DB{\rmK}$-module;
    \item if $\tau$ is a tame inertial type and $\sig^\circ(\tau)\subset \sig(\tau)$ is an $\cO$-lattice,
    \begin{align*}
        M_\infty(\sig^\circ(\tau)) := \PR{\Hom_{\cO[\rmK]}(M_\infty, \sig^\circ(\tau)^\vee)}^\vee
    \end{align*}
    is a maximal Cohen--Macaulay module over $R_\infty(\tau):= R_\infty\ctimes_{R_{\rhobar}^\square} R^{\eta,\tau}_{\rhobar}$ (here $\vee$ denotes the Pontrjagin dual);
    \item the action of $\cH(\sig(\tau))$ on $M_\infty(\sig^\circ(\tau))[1/p]$ factors through the composition
    \begin{align*}
        \cH(\sig(\tau)) \xra{\Psi^\tau} R^{\eta,\tau}_{\rhobar}[1/p] \ra R_\infty(\tau)[1/p];
    \end{align*}
    \item if $\sig$ is a Serre weight and $\tau$ is a tame inertial type such that $\sig \in \JH(\osig(\tau))$, then the $R_\infty$-action on
    \begin{align*}
        M_\infty(\sig) := \PR{\Hom_{\cO[\rmK]}(M_\infty, \sig^\vee)}^\vee
    \end{align*}
    factors through $R_\infty(\tau)_\F$ and $M_\infty(\sig)$ is maximal Cohen--Macaulay module over $R_\infty(\tau)_\F$; and
    \item if we denote by $\fm$ the maximal ideal in $R_\infty$, $(M_\infty/\fm)^\vee$ is isomorphic to $\pi(\rbar)$ as $\F[\GL_n(K)]$-module.
\end{enumerate}

\begin{rmk}
    The module $M_\infty$ was first constructed by \cite{6author} to find a candidate for a $p$-adic local Langlands correspondence. Instead of taking $(M_\infty/\fm)^\vee$, if we let $\fp\le R_\infty$ be a prime ideal given as a kernel of a morphism $R_\infty \ra \cO$ with an induced morphism $\rho:R^\square_{\rhobar} \ra \cO$, then $(M_\infty/\fp)^\vee[1/p]$ is an admissible continuous unitary $E$-Banach representation of $\GL_n(K)$ (see Prop.~2.13 in \loccit) and the assignment $\rho \mapsto (M_\infty/\fp)^\vee[1/p]$ is expected to realize the $p$-adic local Langlands correspondence.
\end{rmk}

\begin{defn}
    For a Serre weight $\sig$, we define $R_{\rhobar}^\sig$ to be the quotient of $R_{\rhobar}^{\eta,\tau}$ that acts faithfully on $M_\infty(\sig)$.
\end{defn}
Note that \textit{a priori} $R_{\rhobar}^\sig$ may depend on the choice of $\rbar$.

\begin{thm}\label{thm:modpLGC-family}
    Suppose that we are in the setup of Theorem \ref{thm:modpLGC}. There exists a morphism 
    \begin{align*}
    \ov{\Psi}'_\sig : \cH(\sig) \ra R_{\rhobar}^\sig
\end{align*}
    such that the natural $\cH(\sig)$-action on $M_\infty(\sig)$ coincides with its action given by the composition of $\ov{\Psi}'_\sig$ and the $R_{\rhobar}^\sig$-action on $M_\infty(\sig)$.
\end{thm}

\begin{proof}
    Let $\sig=F(\mu)$, $\tau=\tau(1,\mu)$, and $\sig^\circ(\tau)$ be as in Theorem \ref{thm:spec-satake}. The quotient map $\sig^\circ(\tau) \epi \sig$ induces a quotient map 
\begin{align*}
    M_\infty(\sig^\circ(\tau)) \epi M_\infty(\sig).
\end{align*}
By Theorem \ref{thm:hecke-modp-red}, the $\cH(\sig^\circ(\tau))$-action on $M_\infty(\sig^\circ(\tau))$ induces an $\cH(\sig)$-action on $M_\infty(\sig)$ via the morphism
\begin{align*}
    \cR^{\sig^\circ(\tau)}_\sig : \cH(\sig^\circ(\tau)) \epi \cH(\sig),
\end{align*}
and this coincides with the natural $\cH(\sig)$-action. 
By the property (2) above, this implies that the composition $\cH(\sig^\circ(\tau)) \ra R_{\rhobar}^{\eta,\tau} \ra R_{\rhobar}^\sig$ factors through $\cR^{\sig^\circ(\tau)}_\sig$, and thus induces the desired morphism $\ov{\Psi}'_\sig$.
\end{proof}

\begin{proof}[Proof of Theorem \ref{thm:modpLGC}]
    Since $(M_\infty(\sig)/\fm)^\vee \simeq \Hom_{\rmK}(\sig,\pi(\rbar)|_{\rmK})$, this follows from Theorem \ref{thm:modpLGC-family}.
\end{proof}

\subsection{Breuil--M\'ezard cycles and the categorical $p$-adic local Langlands conjecture}\label{subsec:BMcycle}

As the final remark, we give a mostly speculative discussion on the ring of global functions on the Breuil--M\'ezard cycle associated with a Serre weight $\sigma$. As discussed in \S\ref{sec:intro}, Emerton--Gee--Hellmann conjectured that there exists a fully faithful functor $\fA$ from a certain derived category of smooth representations of $\GL_n(K)$ to a certain derived category of coherent sheaves on $\cX_n$ realizing a $p$-adic local Langlands correspondence. 

Let $\cZ_{\sig}$ be the scheme-theoretic support of $\fA(\cind_\rmK^\rmG \sig)$. It is more natural to relate $\cH(\sig)$ with $\cO(\cZ_\sig)$ rather than $\cO(\cC_\sig)$. It is expected that $\cZ_\sig$ realizes the conjectural Breuil--M\'ezard cycle associated with $\sig$. We briefly discuss what is known about Breuil--M\'ezard cycles. 

When $n=2$, \cite[Thm.~8.6.2]{EGstack} and \cite[Thm.~1.2]{cegsA} show that $\cZ_{\sig}$ is reduced and equals $\cC_{\sig}$ unless $\sig$ is Steinberg. When $n=3$ and $K/\Qp$ is unramified, \cite[Thm.~1.2.2]{LLLMGL3} shows that $\cZ_{\sig}$ is equal to $\cC_{\sig}$ for sufficiently generic $\sig$. However, starting from $n=4$, $\cZ_{\sig}$ is in general expected to be not equal to $\cC_{\sig}$ even when $\sig$ is generic (see \cite[Rmk.~1.5.11]{LLLMlocalmodel}). 

In \cite{LLLMlocalmodel}, the authors proved the geometric Breuil--M\'ezard conjecture for arbitrary $n$ and $K/\Qp$ unramified under an inexplicit genericity condition. Moreover, they proved that the underlying topological space $\abs{\cZ_{\sig}}$ contains $\abs{\cC_\sig}$ and is contained in the union of $\abs{\cC_{\sig'}}$ for $\sig'$ covered by $\sig$ in the sense of Def.~2.3.10 in \loccit. (Recently, \cite{FLH} proved the geometric Breuil--M\'ezard conjecture under an explicit genericity assumption. See \S12.2 in \loccit~for a result on the decomposition of $\cZ_\sig$ into irreducible components using geometric representation theory.)

Another important result in \cite{LLLMlocalmodel} (again proven under an inexplicit genericity condition) is that the set of semisimple $\rhobar:G_K\ra\GL_n(\Fpbar)$ in $\cC_{\sig}(\Fpbar)$ (in other words, the set of closed points in $\cC_\sig$) is equal to the set of $\rhobar$ such that $\sig \in W^?(\rhobar|_{I_K})$ (\cite[Thm.~4.7.6]{LLLMlocalmodel}). Here, $W^?(\rhobar|_{I_K})$ is the set of Serre weights combinatorially defined by Herzig (\cite{HerzigDuke}). It is given by the set $\cR(\JH(\osig(\rhobar|_{I_K})))$ where $\cR$ is an operator on the set of Serre weights sending $F(\mu)$ to $F(w_0\cdot (\mu-p\eta))$ (see the paragraph before Def.~9.2.5 in \cite{GHS-JEMS-2018MR3871496}). 

Assuming these two expectations (that are verified under an inexplicit genericity condition), we prove that $\cO(\cZ_\sig)$ is isomorphic to $\cH(\sig)$ up to nilpotency.

\begin{prop}\label{prop:BMcycle}
    Assume that $K/\Qp$ is unramified. Let $\sig$ be a $(3(n-1)+1)$-deep Serre weight and $\cZ_\sig\subset \cX_{n,\F}$ be a closed substack. We assume that
    \begin{enumerate}
        \item $\abs{\cZ_\sig}$ is a union of irreducible components in $\cX_{n,\red}$;
        \item $\abs{\cZ_\sig}$ contains $\abs{\cC_\sig}$ and is contained in the union of $\abs{\cC_{\sig'}}$ for all $\sig'$ covered by $\sig$;
        \item for $\sig'$ covered by $\sig$, all closed points of $\sig'$ are exactly those semisimple $\rhobar:G_K\ra\GL_n(\Fpbar)$ such that $\rhobar|_{I_K}$ is $2(n-1)$-generic and $\sig'\in W^?(\rhobar|_{I_K})$.
    \end{enumerate}
    Then the restriction map $\cO(\cZ_\sig)\ra\cO(\cC_\sig)$ is surjective with nilpotent kernel.
\end{prop}

\begin{proof}
    Any function $f\in \cO(\cZ_\sig)$ is determined by its value at all closed points of $\cZ_\sig$ up to nilpotency. By item (3) and \cite[Prop.~2.3.12]{LLLMlocalmodel}, all closed points in $\cC_{\sig'}$ for $\sig'$ covered by $\sig$ are contained in $\cC_\sig$. Thus, any closed point in $\cZ_{\sig}$ is contained in $\cC_{\sig}$. This proves the claim.
\end{proof}

\bibliographystyle{alpha}
\bibliography{mybib}
\end{document}